\documentclass[11pt,twoside,a4paper,leqno]{article}
\usepackage{a4}
\usepackage[margin=1in]{geometry}
\usepackage[utf8]{inputenc}
\usepackage[english]{babel}
\usepackage{euscript,euler}
\usepackage{microtype} 
\usepackage{amsmath,amsthm}
\usepackage{amsfonts,amssymb}
\usepackage{enumerate}
\usepackage{array}
\usepackage{float}
\usepackage{xcolor} 
      \definecolor{bleu_sombre}{rgb}{0,0,0.6}  
     \definecolor{rouge_sombre}{rgb}{0.8,0,0}
    \definecolor{vert_sombre}{rgb}{0.1,0.4,0.1}
\usepackage[plainpages=false,colorlinks,linkcolor=bleu_sombre,citecolor=vert_sombre]{hyperref} 
\definecolor{mygray}{gray}{0.6}
\definecolor{mygray2}{gray}{0.8}
\usepackage{etex}
\usepackage{pgfkeys}
\usepackage{subfigure}
\usepackage{tikz}
\usepackage{tikz-cd}
\usetikzlibrary{arrows,calc,matrix,topaths,positioning,scopes,shapes,decorations,decorations.markings}
\usetikzlibrary{backgrounds,automata}
\usepackage{fancyhdr}			
\DeclareMathOperator{\Hom}{Hom}

\DeclareMathOperator{\Ba}{\mathcal{B}}

\DeclareMathOperator{\colim}{colim}
\DeclareMathOperator{\hocolim}{hocolim}
\DeclareMathOperator{\Coend}{Coend}

\makeatletter
\renewcommand{\thefigure}{\ifnum \c@section>\z@ \thesection.\fi
 \@arabic\c@figure}
\@addtoreset{figure}{section}
\makeatother

\title{The relative lattice path operad} 
\author{Alexandre Quesney}
\date{November, 2017}

\AtEndDocument{\bigskip{\footnotesize%
  \textsc{Universidade de São Paulo - Instituto de Ciências Matemáticas e de Computação.  
Avenida Trabalhador São-carlense, 400 - Centro CEP: 13566-590 - São Carlos - SP Brazil} \par  
  \textit{E-mail address:} \texttt{math@quesney.org}\\
}}

\begin{document}

\maketitle

\theoremstyle{definition}
\newtheorem{de}{Definition}[section]
\theoremstyle{remark}
\newtheorem{rmq}[de]{Remark}
\newtheorem{rmk}[de]{Remark}
\newtheorem{data}[de]{Data}

\newtheorem{exple}[de]{Example}
\theoremstyle{plain}
\newtheorem{theo}[de]{Theorem}
\newtheorem{prop}[de]{Proposition}
\newtheorem{lem}[de]{Lemma}
\newtheorem{cor}[de]{Corollary}
\newtheorem{properties}[de]{Properties}
\numberwithin{equation}{section}

\newtheorem{proj}{Projet}

\newtheorem*{theo*}{Theorem}
\newtheorem*{prop*}{Proposition}
\newtheorem*{cor*}{Corollary}

\newcommand{\ot}{\otimes} 
\newcommand{\os}{\Omega \Sigma}
\newcommand{\La}{A}
\newcommand{\ov}[1]{\overline{#1}}
\newcommand{\OS}[1]{\Omega^{#1} \Sigma^{#1}}
\newcommand{\uv}[1]{\underline{#1}}
\newcommand{\wid}[1]{\widetilde{#1}}
\newcommand{\att}[1]{\textcolor{red}{\textit{#1}}\marginpar{\textcolor{red}{~~~$\bullet$}}}
\newcommand{\note}[1]{\marginpar{\begin{footnotesize}\textcolor{red}{#1}\end{footnotesize}}}
\newcommand{\ca}[1]{\mathcal{#1}}
\newcommand{\catsimpop}{\bigtriangleup^{\text{op}}}
\newcommand{\catsimp}{\bigtriangleup}
\newcommand{\Set}{\text{Set}}
\newcommand{\Oun}{\underline{\Omega}}
\newcommand{\Ocub}{\Omega^{\square}}
\newcommand{\Ccub}{C^{\square}_*}
\newcommand{\cat}[1]{\textbf{#1}}
\newcommand{\act}{\text{act}}
\newcommand{\set}[1]{\{#1\}}
\newcommand{\ord}{\text{ord}}
\newcommand{\ring}{\mathbb{Z}}
\newcommand{\close}{\mathfrak{cl}}
\newcommand{\open}{\mathfrak{op}}
\newcommand{\type}{\chi}
\newcommand{\cgraph}{\ca{RK}}
\newcommand{\n}{e}
\newcommand{\no}{\underline{n}}
\newcommand{\un}[1]{\underline{#1}}
\newcommand{\SC}{\ca{SC}}
\newcommand{\ie}{\emph{i.e.} }
\newcommand{\alex}[1]{{\color{magenta}{#1}}}
\newcommand{\sph}{\mathbb{S}}
\newcommand{\lmodm}{\cat{LMod}_{\ca{M}}}
\newcommand{\rmodm}{\cat{RMod}_{\ca{M}}}
\newcommand{\bmodm}{\cat{BiMod}_{\ca{M}-\ca{N}}}
\newcommand{\lmodas}{\cat{LMod}_{\ca{A}s}}
\newcommand{\rmodas}{\cat{RMod}_{\ca{A}s}}
\newcommand{\bmodas}{\cat{BiMod}_{\ca{A}s}}
\newcommand{\bmodmas}{\cat{BiMod}_{\ca{M}-\ca{A}s}}
\newcommand{\bmodasm}{\cat{BiMod}_{\ca{A}s-\ca{M}}}
\newcommand{\Catend}{\cat{End}}

\newcommand{\sq}[1]{\draw [black,fill=white] ($#1-(.1,.1)$) rectangle ($#1+(.1,.1)$)}

\begin{abstract}
We construct a set-theoretic coloured operad $\ca{RL}$ that may be thought of as a combinatorial model for the Swiss Cheese operad. This is the relative (or Swiss Cheese) version of the lattice path operad constructed by Batanin and Berger. By adapting their condensation process we obtain a topological (resp. chain) operad that we show to be weakly equivalent to the topological (resp. chain) Swiss Cheese operad. 
\end{abstract}
\setcounter{tocdepth}{2}

\section{Introduction}

The Swiss Cheese operad $\SC$ is a 2-coloured topological operad that mixes, in its $m$-dimensional part $\SC_m$, the $m$-dimensional and the $(m-1)$-dimensional parts of the little cubes operad $\ca{C}$. 
 It was introduced by A. Voronov \cite{Voronov} as a natural way to define actions of $C_*(\ca{C}_m)$-algebras  on $C_*(\ca{C}_{m-1})$-algebras and has been  used by M. Kontsevich \cite{Operadsandmotives} in deformation quantization. As announced in \cite{Ainf-HLS}, the Swiss Cheese operad $\SC_m$ also recognizes the pair  ($m$-fold loop space, $m$-fold relative loop space).  
The goal of this paper is to provide a convenient combinatorial model for (the operad of  chains of) $\SC_m$, $m\geq1$.  
\\

In  \cite{Batanin-Berger-Lattice}, Batanin and Berger introduce the notion of \emph{condensation} of a coloured operad. 
By applying this condensation to the \emph{lattice path operad} $\ca{L}$ they obtain a model for the (operad of chains of the) little cubes operad. 

We introduce the \emph{relative lattice path operad} $\ca{RL}$. It is a coloured operad in the category of sets that has  two types of colours (closed and open). %
Taking a cosimplicial object $\delta: \catsimp \to \cat{C}$ in a cocomplete closed monoidal symmetric category $\cat{C}$, 
we adapt Batanin-Berger's method to obtain a functor 
\begin{align*}
  \ca{RL}\text{-algebra} \longrightarrow \Coend_{\ca{RL}}(\delta)\text{-algebra}
\end{align*}
that sends algebras over $\ca{RL}$ into algebras over the \emph{condensation operad} of $\ca{RL}$, that is, the \emph{SC type} operad 
$\Coend_{\ca{RL}}(\delta)$.

The operad $\ca{RL}$ has two filtrations by suboperads $\ca{RL}_m$ and $\ca{RL}'_m$, ${m\geq 1}$; they differ one from each other by their open/closed interacting part. %
  
We are interested by two choices for $\delta$:
\begin{equation*}
\scalebox{1}{
\begin{tikzpicture}[>=stealth,thick,draw=black!50, arrow/.style={->,shorten >=1pt}, point/.style={coordinate}, pointille/.style={draw=red, top color=white, bottom color=red},scale=0.8,baseline=-.5ex]
\matrix[row sep=15mm,column sep=12mm,ampersand replacement=\&]
{
 \node (00) {$\delta_{Top}: \triangle$}; \& \node (01){$\cat{Set}^{\catsimpop}$} ;\& \node (02){$\cat{Top}$} ;\\
}; 
\path
	  
   	  (00)     edge[above,->]      node {$\delta_{yon}$}  (01)
 	  (01)     edge[above,arrow,->]      node {$|-|$}  (02)
 	  ; 
\end{tikzpicture}}
\text{ and }
\scalebox{1}{
\begin{tikzpicture}[>=stealth,thick,draw=black!50, arrow/.style={->,shorten >=1pt}, point/.style={coordinate}, pointille/.style={draw=red, top color=white, bottom color=red},scale=0.8,baseline=-.5ex]
\matrix[row sep=15mm,column sep=12mm,ampersand replacement=\&]
{
 \node (00) {$\delta_{\ring}:\triangle$}; \& \node (01){$\cat{Set}^{\catsimpop}$} ;\& \node (02){$\cat{Ch}(\ring)$,} ;\\
}; 
\path
	  
   	  (00)     edge[above,->]      node {$\delta_{yon}$}  (01)
 	  (01)     edge[above,arrow,->]      node {$C_*(-;\ring)$}  (02)
 	  ; 
\end{tikzpicture}}
\end{equation*}
where $\delta_{yon}([n])=\Hom_{\triangle}(-,[n])$ is the Yoneda functor. 
In this manner, the condensation of $\ca{RL}_m$ leads to a topological operad $\Coend_{\ca{RL}_m}(\delta_{\text{Top}})$ and to a chain operad $\Coend_{\ca{RL}_m}(\delta_{\ring})$, and similarly for $\ca{RL}'_m$.

In the topological case, both the SC type operads $\Coend_{\ca{RL}_m}(\delta_{\text{Top}})$ and $\Coend_{\ca{RL}'_m}(\delta_{\text{Top}})$ naturally act on the pair ($m$-fold loop space, $m$-fold relative loop space). 
\\

In order to rely our operads to the Swiss Cheese operad, we use Berger's method of cellular decompositions. 
The Swiss Cheese operad that we consider is denoted by $\SC_m$, $m\geq 1$, and is the augmented (cubical) version of Voronov's Swiss Cheese operad $\ca{SC}_m^{\text{vor}}$ defined in \cite{Voronov}. 
We construct a cellular decomposition of $\SC_m$ that generalizes 
Berger's cell decomposition of the little $m$-cubes operad \cite{Berger:combinatorialmodel}. 
The latter is indexed by the \emph{extended complete graph operad} $\ca{K}_m$. 
In contrast to the non-relative case,  there are two ways for indexing the cells of  $\SC_m$. This 
 naturally leads us to consider two different indexing operads $\cgraph_m$ and $\cgraph'_{m}$ that may be thought of as the relative versions of $\ca{K}_m$. 
One obtains:  
\begin{theo}
 Let $m\geq 1$. Any topological $\cgraph_{m}$-cellular operad (resp. $\cgraph'_{m}$-cellular operad) is weakly equivalent to the Swiss Cheese operad $\SC_m$.
\end{theo}

The operad  $\Coend_{\ca{RL}_m}(\delta)$ benefits of a decomposition by ``cells'' that are indexed by the poset operad $\ca{RK}_m$.  This is obtained by means of a map $c_{tot}:\ca{RL}_m \to \ca{RK}_m$ satisfying  $c_{tot}(x\circ_i y)\leq c_{tot}(x)\circ_i c_{tot}(y)$. From this (and similar considerations for $\ca{RL}'_m$), one obtains:  
\begin{theo}\label{th: intro0}
 Let  $m\geq 1$.  The operads $\Coend_{\ca{RL}_m}(\delta_{\text{Top}})$ and $\Coend_{\ca{RL}'_m}(\delta_{\text{Top}})$
  are weakly equivalent to the topological  Swiss-Cheese operad $\SC_m$. 
  The operads $\Coend_{\ca{RL}_m}(\delta_{\ring})$ and $\Coend_{\ca{RL}'_m}(\delta_{\ring})$ are weakly equivalent to the chain  Swiss-Cheese operad  $C_*(\SC_m)$.
\end{theo}

Note that, for each $m\geq 1$, the operad $\Coend_{\ca{RL}_m}(\delta_{\ring})$ (resp. $\Coend_{\ca{RL}'_m}(\delta_{\ring})$) admits a weakly equivalent suboperad $\ca{RS}_m$ (resp $\ca{RS}'_m$). 
These two operads $\ca{RS}_m$ and $\ca{RS}'_m$ are relative versions of the \emph{surjection operad} $\ca{S}_m$ studied by McClure-Smith \cite{McClureSmithDelignconj,McClureSmithMulti} and Berger-Fresse \cite{BergerFresse}. 
\\

 With regards to future applications, we pay close attention to the operads $\ca{RL}_2$ and $\ca{RL}'_2$. 
 They benefit of a description by planar rooted trees with different types of vertices. 
 The operad $\ca{RL}_2$ encodes pairs $(\ca{M},\ca{Z})$ where $\ca{M}$ is a multiplicative non-symmetric operad and   $\ca{Z}$ is bimodule over $\ca{M}-\ca{A}s$ (left action of $\ca{M}$, right action of $\ca{A}s$) where $\ca{A}s$ is the operad of associative algebras; such a pair is naturally endowed with a  morphism of bimodules $\iota:\ca{M}\to \ca{Z}$. 
 The operad $\ca{RL}'_2$ encodes almost similar pairs $(\ca{M},\ca{Z}')$ where, instead, $\ca{Z}'$ and $\iota:\ca{M}\to \ca{Z}'$ are  in the category of bimodules over $\ca{A}s-\ca{M}$.

\paragraph{Outline of the paper.}
Section \ref{sec: prel} is first devoted in setting our conventions and notations on non-symmetric, symmetric, coloured and SC type operads. In particular, we spend time on modules and weak modules over a non-symmetric operad; this will be used in the last section  \ref{sec: RL2}.  Afterwards, we define the (symmetric) \emph{coloured SC-operads} and we explain how we \emph{condense} them to obtain SC type operads. \\ %
In section \ref{sec: cell decomp} we consider the (cubical) Swiss Cheese operad $\SC$. For each integer $m\geq 1$, we construct two cellular decompositions of $\SC_m$, one indexed by $\cgraph_m$ the other one indexed  by $\cgraph'_{m}$. 
\\
Section \ref{sec: RL} concerns the relative lattice path operad $\ca{RL}$ which is a coloured SC-operad.  
By using the condensation process from Section \ref{sec: prel} one obtains an SC type operad $\Coend_{\ca{RL}}(\delta)$. 
We use results of Section \ref{sec: cell decomp} to prove Theorem \ref{th: intro0}. %
We end the section by exhibiting a few examples of representations of $\ca{RL}_m$ and $\ca{RL}'_m$. In particular, we   show that the operads  $\Coend_{\ca{RL}_m}(\delta_{\text{Top}})$ and $\Coend_{\ca{RL}'_m}(\delta_{\text{Top}})$ act on the pair  ($m$-fold loop space, $m$-fold relative loop space). 
\\ 
In Section \ref{sec: rel surj operad} we exhibit the sub-operads $\ca{RS}_m$ and $\ca{RS}'_m$ and show that the inclusions $\ca{RS}_m\hookrightarrow \Coend_{\ca{RL}_m}(\delta_{\ring})$ and $\ca{RS}'_m\hookrightarrow \Coend_{\ca{RL}'_m}(\delta_{\ring})$ are weak equivalences. \\
In Section \ref{sec: RL2} we focus on the operads $\ca{RL}_2$ and $\ca{RL}'_2$ and their representations. 
We describe  $\ca{RL}_2$ and $\ca{RL}'_2$ in terms of planar rooted trees. This provides a convenient language for describing the representations of  $\ca{RL}_2$ and $\ca{RL}'_2$.  
\\

\paragraph{Acknowledgement}
I would like to warmly thank both Muriel Livernet for  corrections and suggestions on an earlier version of this paper and 
Eduardo Hoefel for his constant support.  
I am highly indebted to the anonymous referee for her/his numerous enlightening comments, suggestions  and corrections. 
The author was partially supported by ``Bolsista da CAPES à Projeto 88881.030367/2013-01''.

\section{Preliminaries}\label{sec: prel}

All along the paper $\cat{C}=(C,\ot,1_{\cat{C}})$ denotes a cocomplete closed monoidal symmetric category. 
In particular, $\cat{C}$ is endowed with an object $0\in\cat{C}$ such that $0\ot X=0=X\ot 0$ for all $X\in\cat{C}$.

For two  $\cat{C}$-categories (\ie enriched over $\cat{C}$) $\cat{A}$ and $\cat{B}$, we denote by $\cat{A}\ot \cat{B}$ the category with the pairs $(a,b)$ for $a\in \cat{A}$ and $b\in \cat{B}$ as objects and 
\begin{equation*}
  \Hom_{\cat{A}\ot\cat{B}}((a,b),(a',b')):= \Hom_{\cat{A}}(a,a')\ot \Hom_{\cat{B}}(b,b')
\end{equation*}
as hom-objects, where the tensor on the right hand side is the tensor of $\cat{C}$.

\subsection{Non-symmetric operads and modules}\label{sec: def modules} 
 By a \emph{non-symmetric (non-$\Sigma$) operad} $\cal{M}$ we mean  a collection $\{\ca{M}(n)\}_{n\geq 0}$ of objects of $\cat{C}$ together with a unit $\eta: 1_{\cat{C}}\to \ca{M}(1)$  and  partial composition maps  $\circ_i:\ca{M}(m)\ot \ca{M}(n)\to \ca{M}(m+n-1)$, for $1\leq i\leq m$, 
that  satisfy associativity and unit conditions (see \cite[Definition 1.14]{MSS}). 
One denotes by $\ca{A}s$ the non-$\Sigma$ operad of associative algebras in $\cat{C}$ given by $\ca{A}s(n)=1_{\cat{C}}$ for $n\geq 0$.  
 A non-$\Sigma$ operad $\ca{M}$ is called \textit{multiplicative} if there is a morphism of operads $\alpha:\ca{A}s\to \ca{M}$. 
 A morphism of multiplicative operads  $\alpha:\ca{A}s\to \ca{M}$ and $\alpha':\ca{A}s\to \ca{M}'$    is a morphism of operads $f:\ca{M}\to \ca{M}'$ such that  $f\circ \alpha = \alpha'$. 

In what follows, one recalls the notions of left, right, bi module over a non-$\Sigma$ operad.

\begin{de}\label{de: modules}
Let $\ca{M}$ and $\ca{N}$ be two non-$\Sigma$ operads and let $\ca{Z}=\{\ca{Z}(m)\}_{m\geq 0}$ be a collection of objects of $\cat{C}$. 
 Consider the following morphisms: 
 \begin{equation*}
  \rho_i: \ca{Z}(m)\ot \ca{N}(k) \to \ca{Z}(k+m-1) ~~  \text{ for } 1\leq i\leq m ~\text{(right action)}
 \end{equation*}
 \begin{equation*}
  \lambda: \ca{M}(k)\ot \ca{Z}(m_1)\ot \cdots \ot \ca{Z}(m_k) \to \ca{Z}(m_1+\cdots +m_k) ~\text{(left action).}
  \end{equation*}
 Consider the following relations. 
\begin{enumerate}
\item Unit condition for the left action:
$\scalebox{1}{
\begin{tikzpicture}[>=stealth,thick,draw=black!50, arrow/.style={->,shorten >=1pt}, point/.style={coordinate}, pointille/.style={draw=red, top color=white, bottom color=red},scale=0.8,baseline=-0.6ex]
\matrix[row sep=10mm,column sep=5mm,ampersand replacement=\&]
{
 \node (00) {$\ca{Z}(m)= 1_{\cat{C}}\ot \ca{Z}(m)$}; 
 \& \node (01){$ \ca{M}(1)\ot \ca{Z}(m)$} ; \&  \node (02){$ \ca{Z}(m)$} ;\\
}; 
\path
   	  (00)     edge[above,->]      node[yshift=0mm] {\scriptsize{$\eta\ot id$}}  (01)
 	  (01)     edge[above,->]      node {\scriptsize{$\lambda$}}  (02)
 	  ; 
\end{tikzpicture}}$
 is the identity. 
  \item Associativity of the left action:
  all the diagrams of the following form commute
  \begin{center}
 \scalebox{.9}{
\begin{tikzpicture}[>=stealth,thick,draw=black!50, arrow/.style={->,shorten >=1pt}, point/.style={coordinate}, pointille/.style={draw=red, top color=white, bottom color=red},scale=0.8]
\matrix[row sep=10mm,column sep=10mm,ampersand replacement=\&]
{
 \node (00) {$\ca{M}(k)\ot \ca{M}(l)\ot \bigotimes_{1\leq p\leq k+l-1} \ca{Z}(m_p)$}; 
 \& \node (01){$\ca{M}(k)  \bigotimes_{1\leq p< i} \ca{Z}(m_p)\ot \ca{Z}(x)\ot \bigotimes_{i+l\leq p\leq k+l-1} \ca{Z}(m_p)$} ;\\
 \node (10){$\ca{M}(k+l-1)\ot \bigotimes_{1\leq p\leq k+l-1} \ca{Z}(m_p)$} ;
 \& \node (11){$\ca{Z}(m_1+...+m_{k+l-1})$} ;\\
}; 
\path
   	  (00)     edge[below,->]      node[yshift=-1.5mm] {$(id\ot \lambda\ot id)\circ \tau_{2,(3,i+1)}$}  (01)
 	  (00)     edge[left,arrow,->]      node {$(-\circ_i-)\ot id $}  (10)
 	  (01)     edge[right,->]      node {$\lambda$}  (11)
 	  (10)     edge[above,arrow,->]      node {$\lambda$}  (11)
 	  ; 
\end{tikzpicture}}
\end{center}
where $x:=m_i+...+m_{i+l-1}$, for $1\leq i\leq k$, and $\tau_{2,(3,i+1)}$ denotes the permutation of the second factor with the $i-1$ factors that follow. 
  \item Unity condition of right action: 
  $\scalebox{1}{
\begin{tikzpicture}[>=stealth,thick,draw=black!50, arrow/.style={->,shorten >=1pt}, point/.style={coordinate}, pointille/.style={draw=red, top color=white, bottom color=red},scale=0.8,baseline=-0.6ex]
\matrix[row sep=10mm,column sep=5mm,ampersand replacement=\&]
{
 \node (00) {$\ca{Z}(m)= \ca{Z}(m)\ot 1_{\cat{C}}$}; 
 \& \node (01){$ \ca{Z}(m)\ot  \ca{N}(1)$} ; \&  \node (02){$ \ca{Z}(m)$} ;\\
}; 
\path
   	  (00)     edge[above,->]      node[yshift=0mm] {\scriptsize{$\eta\ot id$}}  (01)
 	  (01)     edge[above,->]      node {\scriptsize{$\rho_i$}}  (02)
 	  ; 
\end{tikzpicture}}$ 
  is the identity. 

\item Associativity and commutativity of right action on different inputs: all the diagrams of the following form commute
  \begin{equation*}
 \scalebox{.85}{
\begin{tikzpicture}[>=stealth,thick,draw=black!50, arrow/.style={->,shorten >=1pt}, point/.style={coordinate}, pointille/.style={draw=red, top color=white, bottom color=red},scale=0.75]
\matrix[row sep=10mm,column sep=3mm,ampersand replacement=\&]
{
 \node (00) {$\ca{Z}(m)\ot \ca{N}(k)\ot \ca{N}(l)$}; 
 \& \node (01){$\ca{Z}(m+k-1) \ot \ca{N}(l)$} ;\\
 \node (10){$\ca{Z}(m)\ot \ca{N}(k+l-1)$} ;
 \& \node (11){$\ca{Z}(m+k+l-2)$.} ;\\
}; 
\path
   	  (00)     edge[below,->]      node[yshift=-1mm] {$\rho_t \ot id $}  (01)
 	  (00)     edge[left,arrow,->]      node {$id \ot (-\circ_i-)$}  (10)
 	  (01)     edge[right,->]      node {$\rho_{t+i-1}$}  (11)
 	  (10)     edge[above,arrow,->]      node[yshift=0mm] {$\rho_t$}  (11)
 	  ; 
\end{tikzpicture}}
 \scalebox{.85}{
\begin{tikzpicture}[>=stealth,thick,draw=black!50, arrow/.style={->,shorten >=1pt}, point/.style={coordinate}, pointille/.style={draw=red, top color=white, bottom color=red},scale=0.75]
\matrix[row sep=10mm,column sep=3mm,ampersand replacement=\&]
{
 \node (00) {$\ca{Z}(m)\ot \ca{N}(k)\ot \ca{N}(l)$}; 
 \& \node (01){$\ca{Z}(m+k-1) \ot \ca{N}(l)$} ;\\
 \node (10){$\ca{Z}(m+l-1)\ot \ca{N}(k)$} ;
 \& \node (11){$\ca{Z}(m+k+l-2)$.} ;\\
}; 
\path
   	  (00)     edge[below,->]      node[yshift=-1mm] {$\rho_t \ot id $}  (01)
 	  (00)     edge[left,arrow,->]      node {$\rho_t\circ \tau_{2,3}$}  (10)
 	  (01)     edge[right,->]      node {$\rho_{t+i-1}$}  (11)
 	  (10)     edge[above,arrow,->]      node[yshift=0mm] {$\rho_i$}  (11)
 	  ; 
\end{tikzpicture}}
  \end{equation*}
  \item Associativity of right and left actions: \label{item: assoc left action}
  all the diagrams of the following form commute
  \begin{equation*} 
   \scalebox{.9}{
\begin{tikzpicture}[>=stealth,thick,draw=black!50, arrow/.style={->,shorten >=1pt}, point/.style={coordinate}, pointille/.style={draw=red, top color=white, bottom color=red},scale=0.8]
\matrix[row sep=10mm,column sep=25mm,ampersand replacement=\&]
{
 \node (00) {$\ca{M}(k)\ot \bigotimes_{1\leq p\leq k} \ca{Z}(m_p) \ot \ca{N}(l)$}; 
 \& \node (01){$\ca{M}(k)  \bigotimes_{1\leq p< s} \ca{Z}(m_p)\ot \ca{Z}(x)\ot \bigotimes_{s<p\leq k} \ca{Z}(m_p)$} ;\\
 \node (10){$\ca{Z}(m_1+...+m_k)\ot \ca{N}(l)$} ;
 \& \node (11){$\ca{Z}(m_1+...+m_{k}+l-1)$,} ;\\
}; 
\path
   	  (00)     edge[below,->]      node[yshift=-1.6mm] {$(id\ot id^{\ot s-1} \ot \rho_i\ot id^{\ot k-s})\circ \tau_{(s+2,k+1),k+2}$}  (01)
 	  (00)     edge[left,arrow,->]      node {$\lambda\ot id $}  (10)
 	  (01)     edge[right,->]      node {$\lambda$}  (11)
 	  (10)     edge[above,arrow,->]      node[yshift=0mm] {$\rho_{i+m_1+...+m_{s-1}-s+1}$}  (11)
 	  ; 
\end{tikzpicture}}
\end{equation*}
where $x=m_s+l-1$.
\end{enumerate}
We establish the following terminology: 
\begin{itemize}
\item $(\ca{Z},\lambda)$ is called a \emph{left module} over $\cal{M}$ if $\lambda$ satisfies (1) and (2).
\item $(\ca{Z},\rho)$ is called  a \emph{right module} over $\cal{N}$ if the $\rho_i$ satisfy (3) and (4). 
 \item $(\ca{Z},\lambda,\rho)$ is called a \emph{bimodule} over $\cal{M}-\ca{N}$ if $\lambda$ and $\rho_i$ satisfy (1)-(5). 
\end{itemize}
\end{de}

\begin{exple}
A non-$\Sigma$ operad $\ca{M}$ is canonically a bimodule over itself (\ie over $\ca{M}-\ca{M}$).  
\end{exple}

Given two non-$\Sigma$ operads $\ca{M}$ and $\ca{N}$, the category of bimodules over $\ca{M}-\ca{N}$ is denoted by  $\bmodm$; the morphisms are morphisms that commute with each of the actions maps. 
A bimodule over $\cal{M}-\ca{M}$ is simply called a bimodule over $\cal{M}$. 
 
 \begin{de}\label{de: f equiva}
  Let $\ca{Z}'$ be a left (resp. right, bi) module over $\ca{M}'$. 
A morphism of non-$\Sigma$ operads $f:\ca{M}\to \ca{M}'$ endows  $\ca{Z}'$ with a  structure of left (resp. right, bi) module over $\ca{M}$ by pulling back the actions maps along $f$ ; we denote the resulting  left (resp. right, bi) module over $\ca{M}$ by $f^*\ca{Z}'$. Explicitly,  $\lambda^{f^*\ca{Z}'}=\lambda'\circ(f\ot id^{\ot k})$ and $\rho_i^{f^*\ca{Z}'}=\rho_i'\circ( id \ot f)$. 

Moreover, if $\ca{Z}$ is a left (resp. right, bi) module over $\ca{M}$, then a morphism $g: \ca{Z} \to \ca{Z}'$  is said \emph{$f$-equivariant} if it is a morphism of left (resp. right, bi) modules over $\ca{M}$, $g: \ca{Z} \to f^*\ca{Z}'$. 
 \end{de}

\begin{exple}
A left module over a multiplicative operad is, by pullback, a  left module over $\ca{A}s$.   
\end{exple}

In \cite[Definition 4.1]{Turchin-Hodge}, the notion of \emph{weak} bimodule over a non-$\Sigma$ operad $\ca{M}$ (in $\cat{Top}$) is introduced. In our framework, it refers to a right module over $\ca{M}$ together with a \emph{weak} left action 
\begin{equation*}
  \lambda^w_i: \ca{M}(k)\ot \ca{Z}(m) \to \ca{Z}(k+m-1) ~~  \text{ for } 1\leq i\leq k,  
 \end{equation*}
that satisfies natural associativity and unit  conditions \ie (1) above and (W1)-(W3) below.  
\begin{enumerate}[W1.]
 \item Associativity of the weak left action:\label{item: assoc L WEAK}
  all the diagrams of the following form commute
  \begin{equation*}
   \scalebox{.9}{
\begin{tikzpicture}[>=stealth,thick,draw=black!50, arrow/.style={->,shorten >=1pt}, point/.style={coordinate}, pointille/.style={draw=red, top color=white, bottom color=red},scale=0.8,baseline=-0.5ex]
\matrix[row sep=10mm,column sep=10mm,ampersand replacement=\&]
{
 \node (00) {$\ca{M}(k)\ot \ca{M}(l)\ot \ca{Z}(m)$}; 
 \& \node (01){$\ca{M}(k) \ot \ca{Z}(l+m-1)$} ;\\
 \node (10){$\ca{M}(k+l-1)\ot \ca{Z}(m)$} ;
 \& \node (11){$\ca{Z}(k+l+m-2)$} ;\\
}; 
\path
   	  (00)     edge[below,->]      node[yshift=-.4mm] {$id\ot \lambda^w_j$}  (01)
 	  (00)     edge[left,arrow,->]      node {$(-\circ_i-)\ot id $}  (10)
 	  (01)     edge[right,->]      node {$\lambda^w_i$}  (11)
 	  (10)     edge[above,arrow,->]      node {$\lambda^w_{j+i-1}$}  (11)
 	  ; 
\end{tikzpicture}}
\text{ for $1\leq i\leq k$ and $1\leq j\leq l$}. 
\end{equation*}
\item Associativity of right and left actions: \label{item: assoc RL WEAK}
  all the diagrams of the following form commute
  \begin{equation*}
    \scalebox{.9}{
\begin{tikzpicture}[>=stealth,thick,draw=black!50, arrow/.style={->,shorten >=1pt}, point/.style={coordinate}, pointille/.style={draw=red, top color=white, bottom color=red},scale=0.8,baseline=-0.5ex]
\matrix[row sep=10mm,column sep=10mm,ampersand replacement=\&]
{
 \node (00) {$\ca{M}(k)\ot \ca{Z}(m) \ot \ca{M}(l)$}; 
 \& \node (01){$\ca{M}(k) \ot  \ca{Z}(m+l-1)$} ;\\
 \node (10){$\ca{Z}(k+m-1)\ot \ca{M}(l)$} ;
 \& \node (11){$\ca{Z}(k+m+l-2)$} ;\\
}; 
\path
   	  (00)     edge[below,->]      node {$id\ot \rho_t$}  (01)
 	  (00)     edge[left,arrow,->]      node {$\lambda^w_i\ot id $}  (10)
 	  (01)     edge[right,->]      node {$\lambda^w_i$}  (11)
 	  (10)     edge[above,arrow,->]      node {$\rho_{i+t-1}$}  (11)
 	  ; 
\end{tikzpicture}}
\text{ for $1\leq i\leq k$ and $1\leq t\leq m$}. 
\end{equation*}
\item Compatibility between the actions and the operad composition: \label{item: compat op RL WEAK}
  all the diagrams of the following form commute
  \begin{equation*}
 \scalebox{.85}{
\begin{tikzpicture}[>=stealth,thick,draw=black!50, arrow/.style={->,shorten >=1pt}, point/.style={coordinate}, pointille/.style={draw=red, top color=white, bottom color=red},scale=0.75]
\matrix[row sep=10mm,column sep=3mm,ampersand replacement=\&]
{
 \node (00) {$\ca{M}(k)\ot \ca{Z}(m) \ot \ca{M}(l)$}; 
 \& \node (01){$\ca{M}(k+l-1) \ot  \ca{Z}(m)$} ;\\
 \node (10){$\ca{Z}(k+m-1)\ot \ca{M}(l)$} ;
 \& \node (11){$\ca{Z}(k+m+l-2)$} ;\\
}; 
\path
   	  (00)     edge[above,->]      node[yshift=3mm] {$((-\circ_t-)\ot id)\circ \tau_{2,3}$}  (01)
 	  (00)     edge[left,arrow,->]      node {$\lambda^w_i\ot id$}  (10)
 	  (01)     edge[right,->]      node {$\lambda^w_{i+l-1}$}  (11)
 	  (10)     edge[above,arrow,->]      node[yshift=2mm] {$\rho_t$}  (11)
 	  ; 
\end{tikzpicture}}
 \scalebox{.85}{
\begin{tikzpicture}[>=stealth,thick,draw=black!50, arrow/.style={->,shorten >=1pt}, point/.style={coordinate}, pointille/.style={draw=red, top color=white, bottom color=red},scale=0.75]
\matrix[row sep=10mm,column sep=3mm,ampersand replacement=\&]
{
 \node (00) {$\ca{M}(k)\ot \ca{Z}(m) \ot \ca{M}(l)$}; 
 \& \node (01){$\ca{M}(k+l-1) \ot  \ca{Z}(m)$} ;\\
 \node (10){$\ca{Z}(k+m-1)\ot \ca{M}(l)$} ;
 \& \node (11){$\ca{Z}(k+m+l-2)$,} ;\\
}; 
\path
   	  (00)     edge[above,->]      node[yshift=3mm] {$((-\circ_t-)\ot id)\circ \tau_{2,3}$}  (01)
 	  (00)     edge[left,arrow,->]      node {$\lambda^w_{i}\ot id$}  (10)
 	  (01)     edge[right,->]      node {$\lambda^w_{i}$}  (11)
 	  (10)     edge[above,arrow,->]      node[yshift=2mm] {$\rho_{t+m-1}$}  (11)
 	  ; 
\end{tikzpicture}}
  \end{equation*}
for $1\leq t<i\leq k$ in the left-sided diagram and $1\leq i<t\leq k$ in the right-sided diagram. 
\end{enumerate}

\begin{exple}
 A non-$\Sigma$ operad $\ca{M}$ is a weak bimodule over itself. 
\end{exple}

 \begin{exple}\label{ex: weak mod and map}
Consider a bimodule $\ca{Z}$ over $\ca{M}$ and suppose it comes together with a bimodule map $\iota: \ca{M}\to \ca{Z}$; recall that $\eta:1_{\cat{C}}\to \ca{M}(1)$ denotes the unit. 
 Pre-composing the left action $\lambda$ of $\ca{Z}$ by $\iota \eta$ at all but one input provides maps $\lambda^w_i:=\lambda\circ(id\ot (\iota\eta)^{\ot i-1}\ot id \ot (\iota\eta)^{\ot k-i}):\ca{M}(k)\ot \ca{Z}(m) \to \ca{Z}(k+m-1)$ that endow $\ca{Z}$ with a weak bimodule structure over  $\ca{M}$. Moreover, for this structure, $\iota$ is a morphism of weak bimodules.   
 \end{exple}

As observed in \cite[Lemma 4.2]{Turchin-Hodge}, one of the interesting features of the \emph{weak} bimodules is the following. 

\begin{lem}\label{lem: ASSOC-weak=cosimpl}
  The structure of a cosimplicial object  is equivalent to the structure of a weak bimodule over $\ca{A}s$.
 \end{lem}

\subsection{Coloured operads}

Coloured (symmetric) operads are also known as \emph{small symmetric multicategories}, see \cite[Definition 2.2.21]{Leinster}. For our purposes one uses left actions for the symmetric groups $\Sigma_n$, so that a coloured operad with set of colours $Col$ (or a $Col$-coloured operad) consists of: 
\begin{itemize}
 \item for each $k\geq 0$ and each $(k+1)$-tuple $(\n_1,...,\n_k;\n)$ of colours $\n_i,\n\in Col$, an object
  $\ca{O}(\n_1,...,\n_k;\n)$ in $\cat{C}$; 
  \item for each colour $\n$, a unit $1_{\cat{C}}\to \ca{O}(\n;\n)$; 
  \item for each  $1\leq i\leq k$,  $(\n_1,\dots,\n_k;\n)$ and $(f_1,\dots,f_l;\n_i)$, substitution maps
 \begin{equation*}
  \circ_i:\ca{O}(\n_1,\dots,\n_k;\n)\ot \ca{O}(f_1,\dots,f_l;\n_i) \to \ca{O}(\n_1,\dots,\n_{i-1},f_1,\dots,f_l,\n_{i+1},\dots,\n_k;\n)
 \end{equation*}
 \item for each $\sigma\in \Sigma_k$, a map $\sigma_*:\ca{O}(\n_1,\dots,\n_k;\n)\to \ca{O}(\n_{\sigma^{-1}(1)},\dots,\n_{\sigma^{-1}(k)};\n)$.
\end{itemize}
Substitution maps are required to satisfy the natural unit, associativity and equivariance axioms.

\begin{exple}
 A \emph{symmetric} operad $\ca{P}$ is a $1$-coloured operad with  $\ca{P}(k)=\ca{P}(\underbrace{*,...,*}_k;*)$. 
\end{exple}

 For a coloured operad $\ca{O}$, the category of its unary operations is called its  \emph{underlying category} and is denoted by $\ca{O}_u$: the objects of $\ca{O}_u$ are the colours; the morphisms   are $\ca{O}_u(\n,f):=\ca{O}(\n;f)$ for $\n,f\in {Col}$.  
The operadic structure of $\ca{O}$ is encoded as functors 
\begin{equation*}
 \ca{O}(\underbrace{-,\cdots, -}_{k}; -): (\ca{O}_u^{\text{op}})^{\ot k} \ot \ca{O}_u \to \cat{C}, ~~~ k\geq 0.
\end{equation*}

Recall that an 
 $\ca{O}$-algebra $X$ is a family $\{X(\n)\}_{\n\in Col}$ of objects $X(\n)\in \cat{C}$ equipped with 
 morphisms 
 \begin{align}\label{eq: action OX}
  \ca{O}(\n_1,...,\n_k;\n)\ot X(\n_1)\ot \dots \ot X(\n_k)\to X(\n),~~~ \n_1,...,\n_k,\n\in Col
 \end{align}
subject to the natural unit, associative and equivariance axioms.
\begin{de}
Let $\cat{C}$ be a \emph{monoidal model category}; see \cite{Hov}.  
 A morphism of $Col$-coloured operads $\ca{O}\to \ca{O}'$ is a \emph{weak equivalence} if and only if each of its components $\ca{O}(n_1,\dots,n_k;n)\to \ca{O}'(n_1,\dots,n_k;n)$ is a weak equivalence.  
Two $Col$-coloured operads $\ca{O}$ and $\ca{O}'$  are said \emph{weakly equivalent} if there is a zig-zag of weak equivalences of $Col$-coloured operads $\ca{O}\leftarrow \cdot \to \ca{O}'$. 
\end{de}
\begin{rmq}
The weak equivalence as described above is  part of the model category structure on the category of $Col$-coloured operads  as established in \cite[Section 3]{BergerMoerdijkResolutionOfColouredOperads} whenever $\cat{C}$ satisfies properties of \textit{loc. cit., Theorem 2.1}.  
\end{rmq}

\subsection{SC-operads and coloured SC-operads}

The SC type operads (see \cite{BataninSymmetrization})  are a special type of  $2$-coloured operads whose structure mimics that of the Swiss-Cheese operad introduced in \cite{Voronov}. Explicitly, by an \emph{SC type operad} (or SC-operad),  one means  a  $\{\close, \open\}$-coloured operad $\ca{O}$  such that  $\ca{O}(c_1,...,c_n;\close)=0$ if there exists a $1\leq i\leq n$ such that $c_i=\open$.  
The colour $\close$ is called the  \emph{closed} colour; the colour $\open$ is called the  \emph{open} colour.

Let us define the coloured version of the SC type operads. 
\begin{de}\label{de: col SC oper} Let $Col$ be a set (of colours). 
 A \emph{coloured SC-operad} is a $Col$-coloured operad $\ca{O}$ that satisfies the following hypotheses.
 \begin{itemize}
 \item[H1.] $Col= {Col}_{\close} \sqcup {Col}_{\open}$.
 \item[H2.] The collection of the  $\ca{O}(\n_1,\dots, \n_k;\n)$ for $\n_i,\n\in {Col}_{\close}$ and $k\geq 0$ forms a sub-operad of $\ca{O}$.
 \item[H3.] The collection of the  $\ca{O}(\n_1,\dots, \n_j;\n)$ for $\n_i,\n\in {Col}_{\open}$ and $j\geq 0$ forms a sub-operad of $\ca{O}$.
 \item[H4.] $\ca{O}(\n_1,\dots, \n_j;\n)=0 \in \cat{C}$ for any $\n\in Col_{\close}$ if there exists $1\leq i\leq j$, such that $\n_i\in {Col}_{\open}$, where $j\geq 1$.
\end{itemize}
  The sub-operad in H2 is called the \emph{closed part of $\ca{O}$}; the sub-operad in H3 is called the \emph{open part of $\ca{O}$}. For $c\in \{\close; \open\}$,  a colour of ${Col}_{c}$ is  called \emph{colour of type} $c$ or  \emph{$c$ colour}.
\end{de}

\begin{exple}
 An SC type operad is a $Col$-coloured SC-operad with $Col_\close=\{\close\}$ and $Col_\open=\{\open\}$. 
\end{exple}

The underlying category $\ca{O}_u$ of a coloured SC-operad contains two particular categories:
\begin{itemize}
 \item $\ca{O}_u^{\close}$ is the sub-category of $\ca{O}_u$ with objects the colours in ${Col}_{\close}$ and morphisms the  $\ca{O}_u(\n,f)$ for $\n,f\in {Col}_{\close}$;
 \item $\ca{O}_u^{\open}$ is the sub-category of $\ca{O}_u$ with objects the colours in ${Col}_{\open}$ and morphisms the $\ca{O}_u(\n,f)$ for $\n,f\in {Col}_{\open}$.
\end{itemize}

Note also that, from the hypothesis on $\ca{O}$, an $\ca{O}$-algebra $X$ can be seen as a pair $(X_{\close},X_{\open})$ where $X_{\close}$ is the sub-family $\{X_{\close}(\n)\}_{\n\in Col_{\close}}$ and $X_{\open}$ is the sub-family $\{X_{\open}(\n)\}_{\n\in Col_{\open}}$.

\subsection{SC functor-operads}\label{sec: SC funt-op}

Here one defines the SC analogues to the functor-operads and their algebras. 
The notion of functor-operad is introduced in \cite{McClureSmith-cosimplicial} and generalizes the notion of operad (see also \cite{Batanin-Berger-Lattice}). 

 Let us fix $\cat{A}$ and $\cat{B}$ be two $\cat{C}$-categories.

For $k\geq 0$, for a collection of $\cat{C}$-functors  
 $ \{ \xi_{\cat{A}_1,...,\cat{A}_k;\cat{A}_{k+1}}: \cat{A}_{1}\ot \cdots \ot \cat{A}_{k} \to \cat{A}_{k+1}\}_{\cat{A}_i\in \{\cat{A},\cat{B}\}}$ 
 and for  $\sigma \in \Sigma_{k}$, 
 we denote by  $ \xi_{\cat{A}_1,...,\cat{A}_k;\cat{A}_{k+1}}^{\sigma}: \cat{A}_{1}\ot \cdots \ot \cat{A}_{k} \to \cat{A}_{k+1}$ the functor 
 \begin{equation*}
\xi_{\cat{A}_1,...,\cat{A}_k;\cat{A}_{k+1}}^{\sigma}(X_1,...,X_{k})=\xi_{\cat{A}_{\sigma^{-1}(1)},...,\cat{A}_{\sigma^{-1}(k)};\cat{A}_{k+1}}(X_{\sigma^{-1}(1)},...,X_{\sigma^{-1}(k)}).
 \end{equation*}
 A collection of $\cat{C}$-functors 
 $\{ \xi_{\cat{A}_1,...,\cat{A}_k;\cat{A}_{k+1}}: \cat{A}_{1}\ot \cdots \ot \cat{A}_{k} \to \cat{A}_{k+1}\}_{\cat{A}_i\in \{\cat{A},\cat{B}\}}$
  is called \emph{twisted symmetric} if there exist  $\cat{C}$-natural transformations 
 $\phi_{\sigma,\cat{A}_1,...,\cat{A}_k;\cat{A}_{k+1}}: \xi_{\cat{A}_1,...,\cat{A}_k;\cat{A}_{k+1}}\to \xi_{\cat{A}_1,...,\cat{A}_k;\cat{A}_{k+1}}^{\sigma}$ for $\sigma\in \Sigma_k$, such that 
 $\phi_{\sigma_1\sigma_2,\cat{A}_1,...,\cat{A}_k;\cat{A}_{k+1}}=(\phi_{\sigma_1,\cat{A}_1,...,\cat{A}_k;\cat{A}_{k+1}})^{\sigma_2}\phi_{\sigma_2,\cat{A}_1,...,\cat{A}_k;\cat{A}_{k+1}} $ 
 and such that $\phi_{id,\cat{A}_1,...,\cat{A}_k;\cat{A}_{k+1}}$ is the identity transformation where $id$ denotes the neutral element of $\Sigma_{k}$.

 \begin{de}\label{de: SC functor-operad}
 An \emph{SC functor-operad} $\xi= \{\xi_{\cat{A}_1,...,\cat{A}_k;\cat{A}_{k+1}}\}_{k,\cat{A}_i}$ over $(\cat{A},\cat{B})$  is the  
 data, for each $k\geq 0$, of a twisted symmetric collection 
$   \xi_{\cat{A}_1,...,\cat{A}_k;\cat{A}_{k+1}}: \cat{A}_{1}\ot \cdots \ot \cat{A}_{k} \to \cat{A}_{k+1}$ 
 indexed by the $(k+1)$-tuples $(\cat{A}_1,...,\cat{A}_k;\cat{A}_{k+1})$ of categories in $\{\cat{A},\cat{B}\}$ such that  $\cat{A}_{k+1}=\cat{B}$ whenever there exists $1\leq i\leq k$ such that $\cat{A}_{i}=\cat{B}$. 
 Such a collection is required to be endowed with natural transformations
 \begin{multline*}
  \mu_{[\cat{A}]_{1,i_1},...,[\cat{A}]_{k,i_k};\cat{A}_{k+1}}: 
  \xi_{\cat{A}_1,...,\cat{A}_{k};\cat{A}_{k+1}}\circ(\xi_{\cat{A}_{1,1},...,\cat{A}_{1,i_1};\cat{A}_{1}}\ot \dots \ot \xi_{\cat{A}_{k,1},...,\cat{A}_{k,i_k};\cat{A}_k}) \\
  \to \xi_{\cat{A}_{1,1},...,\cat{A}_{k,i_k};\cat{A}_{k+1}},~~~\text{ for } i_1,...,i_k\geq 0,
 \end{multline*}
 where $[\cat{A}]_{a,b}=(\cat{A}_{a,1},...,\cat{A}_{a,b};\cat{A}_{a})$. 
These natural transformations have to satisfy the following three conditions.
\begin{enumerate}
 \item For $\cat{A}_0\in \{\cat{A},\cat{B}\}$, the functor  $\xi_{\cat{A}_0;\cat{A}_0}$ is the identity and $\xi_{\cat{A}_{k+1};\cat{A}_{k+1}}\circ \xi_{\cat{A}_1,...,\cat{A}_{k};\cat{A}_{k+1}}=\xi_{\cat{A}_1,...,\cat{A}_{k};\cat{A}_{k+1}}= \xi_{\cat{A}_1,...,\cat{A}_{k};\cat{A}_{k+1}}\circ (\xi_{\cat{A}_1;\cat{A}_1}\ot \dots \ot \xi_{\cat{A}_k;\cat{A}_k})$ where the equalities are obtained via  $\mu_{(\cat{A}_1,...,\cat{A}_k;\cat{A}_{k+1});\cat{A}_{k+1}}$ and $\mu_{(\cat{A}_1;\cat{A}_1),...,(\cat{A}_{k};\cat{A}_k);\cat{A}_{k+1}}$ respectively.
 \item The natural transformations $\mu_{[\cat{A}]_{1,i_1},...,[\cat{A}]_{k,i_k};\cat{A}_{k+1}}$ are associative.
 \item All diagrams of the following forms commute:
  \begin{center}
\begin{tikzpicture}[>=stealth,thick,draw=black!50, arrow/.style={->,shorten >=1pt}, point/.style={coordinate}, pointille/.style={draw=red, top color=white, bottom color=red},scale=0.8]
\matrix[row sep=12mm,column sep=12mm,ampersand replacement=\&]
{
 \node (00) {$\xi_{\cat{A}_1,...,\cat{A}_{k};\cat{A}_{k+1}}\circ(\xi_{\cat{A}_{1,1},...,\cat{A}_{1,i_1};\cat{A}_{1}}\ot \dots \ot \xi_{\cat{A}_{k,1},...,\cat{A}_{k,i_k};\cat{A}_k})$}; \& \node (01){$ \xi_{\cat{A}_{1,1},...,\cat{A}_{k,i_k};\cat{A}_{k+1}}$} ;\\
 \node (10){$\xi_{\cat{A}_1,...,\cat{A}_k;\cat{A}_{k+1}}^{\sigma}\circ(\xi_{\cat{A}_{1,1},...,\cat{A}_{1,i_1};\cat{A}_1}^{\sigma_1}\ot \dots \ot \xi_{\cat{A}_{k,1},...,\cat{A}_{k,i_k};\cat{A}_{k}}^{\sigma_k})$} ;\& \node (11){$ \xi_{\cat{A}_{1,1},...,\cat{A}_{k,i_k};\cat{A}_{k+1}}^{\sigma(\sigma_1,...,\sigma_{k+j})}$.} ; \\
}; 
\path
   	  (00)     edge[left,->]      node {$\phi_{\sigma}\circ (\phi_{\sigma_1}\ot\dots \ot \phi_{\sigma_{k+j}})$}  (10)
 	  (00)     edge[above,arrow,->]     node[yshift=2mm] {$ \mu_{[\cat{A}]_{1,i_1},...,[\cat{A}]_{k,i_k};\cat{A}_{k+1}}$}  (01)
 	  (01)     edge[left,->]      node {$\phi_{\sigma(\sigma_1,\dots,\sigma_{k+j})}$}  (11)
 	  (10)     edge[below,arrow,->]   node[yshift=-3mm]  {$ \mu_{[\cat{A}]_{1,i_1},...,[\cat{A}]_{k,i_k};\cat{A}_{k+1}}$}  (11)
 	  ; 
\end{tikzpicture}
\end{center}
\end{enumerate}
\end{de}

\begin{de}
 Let $\xi= \{\xi_{\cat{A}_1,...,\cat{A}_k;\cat{A}_{k+1}}\}_{k\geq 0}$ be an SC functor-operad over $(\cat{A},\cat{B})$.
 A \emph{$\xi$-algebra} $X$ is a pair $(X_\cat{A},X_\cat{B})\in \cat{A}\ot \cat{B}$ equipped with morphisms in $\cat{A}_{k+1}$,  
 \begin{align*}
  \alpha_{\cat{A}_1,...,\cat{A}_k;\cat{A}_{k+1}}:\xi_{\cat{A}_1,...,\cat{A}_k;\cat{A}_{k+1}}(X_{\cat{A}_1},...,X_{\cat{A}_k})&\to X_{\cat{A}_{k+1}},~~~ k\geq 0,
 \end{align*}
 subject to the following conditions.
\begin{enumerate}
 \item $\alpha_{\cat{A}_1}=1_{X_{\cat{A}_1}}$;
 \item $\alpha_{\cat{A}_1,...,\cat{A}_k;\cat{A}_{k+1}}\circ \phi_{\sigma}=\alpha_{\cat{A}_1,...,\cat{A}_k;\cat{A}_{k+1}}$, for all $\sigma \in \Sigma_k$;  
 \item  all the diagrams of the following form commute
 \end{enumerate}\begin{center}
 \scalebox{0.98}{
\begin{tikzpicture}[>=stealth,thick,draw=black!50, arrow/.style={->,shorten >=1pt}, point/.style={coordinate}, pointille/.style={draw=red, top color=white, bottom color=red},scale=0.8]
\matrix[row sep=15mm,column sep=10mm,ampersand replacement=\&]
{
\node (00) {$\xi_{\cat{A}_1,...,\cat{A}_k;\cat{A}_{k+1}}\circ(\xi_{[\cat{A}]_{1,i_1}}(X_{1,i_1})\ot \dots \ot \xi_{[\cat{A}]_{k,i_k}}(X_{k,i_k}))$}; \& \node (01){$\xi_{\cat{A}_{1,1},...,\cat{A}_{k,i_k};\cat{A}_{k+1}}(X_{\cat{A}_{1,1}},...,X_{\cat{A}_{k,i_k}})$} ;\\
 \node (10){$\xi_{\cat{A}_1,...,\cat{A}_k;\cat{A}_{k+1}}(X_{\cat{A}_{1}},...,X_{\cat{A}_{k+1}})$} ;\& \node (11){$X_{\cat{A}_{k+1}}$} ; \\
}; 
\path
   	  (00)     edge[right,->]      node {$\xi_{\cat{A}_1,...,\cat{A}_k;\cat{A}_{k+1}}(\alpha_{[\cat{A}]_{1,i_1}}\ot\dots \ot \alpha_{[\cat{A}]_{k,i_k}})$}  (10)
 	  (00)     edge[above,arrow,->]     node[yshift=2mm] {$\mu_{[\cat{A}]_{1,i_1},...,[\cat{A}]_{k,i_k};\cat{A}_{k+1}}$}  (01)
 	  (01)     edge[right,->]      node {$\alpha_{\cat{A}_{1,1},...,\cat{A}_{k,i_k};\cat{A}_{k+1}}$}  (11)
 	  (10)     edge[below,arrow,->]   node[yshift=-1mm]  {$\alpha_{\cat{A}_1,...,\cat{A}_k;\cat{A}_{k+1}}$}  (11)
 	  ; 
\end{tikzpicture}}
\end{center}
 where $X_{a,b}$ denotes $X_{\cat{A}_{a,1}},...,X_{\cat{A}_{a,b}}$ and where $[\cat{A}]_{a,b}=(\cat{A}_{a,1},...,\cat{A}_{a,b};\cat{A}_{a})$. 
\end{de}

\begin{rmq}
An {SC functor-operad} over $(\cat{A},\cat{B})$  is a particular example of an \emph{internal symmetric operad} in $\Catend_{\cat{A},\cat{B}}$ (the endomorphism SC type operad  of $(\cat{A},\cat{B})$ in $\cat{Cat}$). 
The notion of internal symmetric operad was introduced in \cite[Definition 9.3]{Bataninhigher}. 
\end{rmq}

 \subsection{Condensation}\label{sec: condensation}

 In this section we explain how we extend the condensation process described in \cite[Section 1]{Batanin-Berger-Lattice} to the case of coloured SC-operads. 
 
 Let  $\ca{O}$ be a coloured SC-operad and let $\delta=(\delta_{\close},\delta_{\open})$ be a pair of two functors $\delta_{\close}: \ca{O}_u^{\close}\to \cat{C}$ and  $\delta_{\open}: \ca{O}_u^{\open}\to \cat{C}$. 
We will define the \emph{condensation operad} of $\ca{O}$, denoted by $\Coend_{\ca{O}}(\delta)$, 
as well as an associated functor 
 \begin{align*}
  \ca{O}\text{-algebra} \longrightarrow \Coend_{\ca{O}}(\delta)\text{-algebra}. 
  \end{align*}
 The operad $\Coend_{\ca{O}}(\delta)$ is obtained by condensing each type of colours into one colour, so that it is an SC type operad. 
 It is obtained in two steps. 
\paragraph{From $\ca{O}$ to the SC functor-operad $\xi(\ca{O})$.}

Recall that, by the hypotheses H2 and H3 from Definition \ref{de: col SC oper}, both $\ca{O}_u^{\close}$ and $\ca{O}_u^{\open}$ are $\cat{C}$-categories. Moreover, the category $\cat{C}^{\ca{O}_u^{\close}}$ (resp. $\cat{C}^{\ca{O}_u^{\open}}$) of $\cat{C}$-functors from $\ca{O}_u^{\close}$ (resp. from $\ca{O}_u^{\close}$) to $\cat{C}$ is a $\cat{C}$-category. 
 For $k\geq 0$ and $(c_1,...,c_k;c_{k+1})$ a tuple of elements in $\{\close;\open\}$ satisfying 
 \begin{equation}\label{eq: col cond}
c_{k+1}=\open \text{ if there exists $1\leq i\leq k$ such that  $c_i=\open$},
 \end{equation}
 one lets $\cat{A}_i:= \cat{C}^{ \ca{O}_u^{c_i}}$ and one defines the $\cat{C}$-functor
\begin{align*}
 \xi(\ca{O})_{c_1,...,c_k;c_{k+1}}: \cat{A}_1\ot \cdots \ot \cat{A}_k \to \cat{A}_{k+1} 
\end{align*}
as the coend
\begin{equation*}
 \xi(\ca{O})_{c_1,...,c_k;c_{k+1}}(X_{c_1},...,X_{c_k})(n)=
 \ca{O}({\underbrace{-,\dots,-}_k};n)\ot_{\ca{O}_u^{c_1}\ot \cdots \ot \ca{O}_u^{c_k}}X_{c_1}(-)\ot \cdots \ot X_{c_k}(-).
\end{equation*}

We have an SC analogue to \cite[Proposition 1.8]{Batanin-Berger-Lattice} or \cite{Day-Street}:
\begin{prop} \label{prop: SC func op}
The functors $\xi(\ca{O})_{c_1,...,c_k;c_{k+1}}$ extend to an SC functor-operad $\xi(\ca{O})$, such that the category of $\ca{O}$-algebras and the category of $\xi(\ca{O})$-algebras are isomorphic.
\end{prop}
\begin{proof}
A straightforward verification, along the lines of  \cite{Day-Street},  shows that the family of the $\xi(\ca{O})_{c_1,...,c_k;c_{k+1}}$ forms an SC functor-operad.

Via the hypothesis H2 and H3 from Definition \ref{de: col SC oper}, an $\ca{O}$-algebra $X$ can be seen as a pair $(X_{\close},X_{\open})$ where  $X_{\close}$ and $X_{\open}$ are functors $X_{\close}:\ca{O}_u^{\close}\to  \cat{C}$ and $X_{\open}:\ca{O}_u^{\open}\to  \cat{C}$ respectively.
 Therefore, the  maps   \eqref{eq: action OX}   induce maps 
\begin{multline*} 
\xi_{c_1,...,c_k;c_{k+1}}(\ca{O})(X_{c_1},...,X_{c_k})(\n_{k+1})= \\ \int^{\n_1,...,\n_{k}}\ca{O}(\n_1,...,\n_k;\n_{k+1})\ot X_{c_1}(\n_1)\ot \dots \ot X_{c_k}(\n_k) 
\to X_{c_{k+1}}(\n_{k+1}),  
 \end{multline*} 
 for $\n_{k+1}\in Col$. They form the maps 
  $\alpha_{c_1,...,c_k;c_{k+1}}:
  \xi(\ca{O})_{c_1,...,c_k;c_{k+1}}(X_{c_1},...,X_{c_k})\to X_{c_{k+1}}$, for $k\geq 0$.  
We conclude that $X$ is a  $\xi(\ca{O})$-algebra because of the unit, associativity and equivariance properties of maps \eqref{eq: action OX}. The isomorphism follows from the universal property of the coend. 
 \end{proof}

\paragraph{From $\xi(\ca{O})$ to the coendomorphism operad $\Coend_{\ca{O}}(\delta)$.} 
The operad  $\Coend_{\ca{O}}(\delta)$ is the coendomorphism operad of the SC functor-operad $\xi(\ca{O})$. 
Explicitly, it is given by: 
 \begin{align*}
  \Coend_{\ca{O}}(\delta)(c_1,...,c_k;c_{k+1})=
  \Hom_{\cat{C}^{\ca{O}_u^{c_{k+1}}}}(\delta_{c_{k+1}},\xi(\ca{O})_{c_1,...,c_k;c_{k+1}}({\delta_{c_1},...,\delta_{c_k}})),
 \end{align*}
 for $(c_1,...,c_k;c_{k+1})$ satisfying \eqref{eq: col cond}. 
The composition maps
 \begin{multline*}
 \Coend_{\ca{O}}(\delta)(c_1,...,c_k;c_{k+1})\ot 
 \Coend_{\ca{O}}(\delta)(c_{1,1},...,c_{1,i_1};c_{1})\ot \cdots \ot \Coend_{\ca{O}}(\delta)(c_{k,1},...,c_{k,i_k};c_{k})\\
 \to  \Coend_{\ca{O}}(\delta)(c_{1,1},...,c_{k,i_k};c_{k+1})
\end{multline*}
 are given by sending maps $f\ot g_1\ot \dots \ot g_k$ to the composite 
 \begin{center}
\begin{tikzpicture}[>=stealth,thick,draw=black!50, arrow/.style={->,shorten >=1pt}, point/.style={coordinate}, pointille/.style={draw=red, top color=white, bottom color=red},scale=0.8]
\matrix[row sep=10mm,column sep=13mm,ampersand replacement=\&]
{
 \node (00) {$\delta_{c_{k+1}}$};\&  \node (03){$\xi(\ca{O})_{c_{1,1},...,c_{k,i_k};c_{k+1}}(\delta_{c_{1,1}},...,\delta_{c_{k,i_k}})$} ; \\
   \node   (01){$\xi(\ca{O})_{c_1,...,c_k;c_{k+1}}(\delta_{c_1},...,\delta_{c_k})$}  ;
 \& \node (02){$\xi(\ca{O})_{c_1,...,c_{k};c_{k+1}}(\xi(\ca{O})_{[c]_{1,i_1}}(\delta),...,\xi(\ca{O})_{[c]_{k,i_k}}(\delta))$.} 
 ;\\
}; 
\path
   	  (00)     edge[left,->]      node {$f$}  (01)
 	  (01)     edge[above,->]      node[yshift=3mm] {\small{$\xi(\ca{O})_{c_{1},...,c_k;c_{k+1}}(g_1,...,g_k)$}}  (02)
 	  (02)     edge[right,arrow,->]   node[yshift=0mm]  {$\alpha_{c_{1},...,c_{k};c_{k+1}}$}  (03)
 	  ; 
\end{tikzpicture}
\end{center}
The action of $\Sigma_k$ on $\Coend_{\ca{O}}(\delta)$  is given by post-composing with the natural transformations $\phi_{\sigma,c_1,...,c_k}({\delta_{c_1},...,\delta_{c_k}}):\xi(\ca{O})_{c_{1},...,c_k;c_{k+1}}({\delta_{c_1},...,\delta_{c_k}})\to \xi(\ca{O})_{c_{\sigma^{-1}(1)},...,c_{\sigma^{-1}(k)};c_{k+1}}({\delta_{c_{\sigma^{-1}(1)}},...,\delta_{c_{\sigma^{-1}(k)}}})$. 
 \\

Given an $\ca{O}$-algebra $X=(X_{\close},X_{\open})$, we set 
\begin{align*}
 Tot_{\delta_{\close}}X_{\close}:=\Hom_{\cat{C}^{\ca{O}_u^{\close}}}(\delta_{\close},X_{\close}) ~~~\text{ and }~~~
Tot_{\delta_{\open}}X_{\open}:=\Hom_{\cat{C}^{\ca{O}_u^{\open}}}(\delta_{\open},X_{\open}).
\end{align*}
Since $X$ is a $\xi(\ca{O})$-algebra by  Proposition \ref{prop: SC func op} and since $\Coend_{\ca{O}}(\delta)$ is the coendomorphism operad of $\xi(\ca{O})$, it follows that  the pair  $(Tot_{\delta_{\close}}X_{\close},Tot_{\delta_{\open}}X_{\open})$ is a ${\Coend}_{\ca{O}}(\delta)$-algebra: 
the action maps 
\begin{equation}\label{eq: action maps}
 \Coend_{\xi(\ca{O})}(\delta)(c_1,...,c_k;c_{k+1})\ot Tot_{\delta_{c_1}}X_{c_1}\ot\cdots \ot Tot_{\delta_{c_k}}X_{c_k} \to Tot_{\delta_{c_{k+1}}}X_{c_{k+1}}
\end{equation}
 are given by sending maps $f\ot g_1\ot \dots \ot g_k$ to the composite 
 \begin{center}
\begin{tikzpicture}[>=stealth,thick,draw=black!50, arrow/.style={->,shorten >=1pt}, point/.style={coordinate}, pointille/.style={draw=red, top color=white, bottom color=red},scale=0.8]
\matrix[row sep=10mm,column sep=15mm,ampersand replacement=\&]
{
 \node (00) {$\delta_{c_{k+1}}$};\& \node (03){$X_{c_{k+1}}$} ; \\
   \node (01){$\xi(\ca{O})_{c_1,...,c_k;c_{k+1}}(\delta_{c_1},...,\delta_{c_k})$} ;
 \& \node (02){$\xi(\ca{O})_{c_1,...,c_k;c_{k+1}}(X_{c_1},...,X_{c_k})$.} ;\\
}; 
\path
   	  (00)     edge[left,->]      node {$f$}  (01)
 	  (01)     edge[above,->]      node[yshift=3mm] {\small $\xi(\ca{O})_{c_1,...,c_k;c_{k+1}}(g_1,...,g_k)$}  (02)
 	  (02)     edge[right,arrow,->]   node[yshift=1mm]  {$\alpha_{c_1,...,c_k;c_{k+1}}$}  (03)
 	  ; 
\end{tikzpicture}
\end{center}

Unit, associative and equivariance properties of the maps \eqref{eq: action maps}  are deduced from the SC functor-operad properties of $\xi(\ca{O})$.

\section{Cellular decompositions of the Swiss Cheese operad}\label{sec: cell decomp}
The little cubes operad $\ca{C}$ has a cellular decomposition indexed by the extended complete graph operad $\ca{K}$, see \cite{Berger:combinatorialmodel} and \cite[4.1]{BFV}. 
We extend this result to the Swiss Cheese operads $\SC_m$, $m\geq 1$ what provides a recognition principle for SC type operads. 
In particular, we construct a poset operad $\cgraph_{m}$ that indexes the cells $(\SC_m)^{(\alpha)}$ of $\SC_m$.  
This leads to  a zig-zag of weak equivalences of operads
 \begin{center}
\scalebox{1}{
\begin{tikzpicture}[>=stealth,thick,draw=black!50, arrow/.style={->,shorten >=1pt}, point/.style={coordinate}, pointille/.style={draw=red, top color=white, bottom color=red},scale=0.8]
\matrix[row sep=5mm,column sep=15mm,ampersand replacement=\&]
{
 \node (00) {$\SC_m$};  \& \node (01){$\hocolim_{\alpha \in \cgraph_{m}} (\SC_m)^{(\alpha) }$} ; \& \node (10){$\Ba \cgraph_{m}$,} ;\\
}; 
\path
   	  (00)     edge[above,<-]      node {$\sim$}  (01)
 	  (01)     edge[above,->]      node {$\sim$}  (10)
 	  ; 
\end{tikzpicture}}
\end{center}
between the Swiss Cheese operad $\SC_m$ and the operad of the geometric realization of the nerve of $\cgraph_{m}$.  
Besides, there is a second way to index the cells $(\SC_m)^{(\alpha)}$;  this is done by another  poset operad $\cgraph'_m$, providing a  similar zig-zag.

\subsection{The Swiss Cheese operad}
The Swiss Cheese operad that we use is the cubical version of the one defined in \cite{Operadsandmotives}.

Let $m\geq 1$. Let $Sym: \mathbb{R}^m\to \mathbb{R}^m$ be the reflection $Sym(x_1,...x_m)=(x_1,...,-x_m)$, and let $Half_+$ be the upper half space 
\begin{equation*}
 Half_+= \{ (x_1,...,x_m)\in \mathbb{R}^m| x_m>0\}.
\end{equation*}
The standard cube $C_0$ in $\mathbb{R}^m$ is $C_0=[-1,1]^{\times m}$.
A cube $C$ in the standard cube is of the form ${C}=[x_1,y_1]\times [x_2,y_2]\times \dots \times [x_m,y_m]$ with  {$-1<x_j<y_j<1$} for $1\leq j\leq m$.

\begin{de}\label{de: operad SC+}
 For  $n \geq 0$ and $c_i,c \in \{\close, \open\}$  we define a topological $\Sigma_n$-space $\SC_m(c_1,...,c_n;c)$ as
 the empty-set if  $c=\close$ and there exists $1\leq i\leq n$ such that $c_i=\open$; for the other cases, it is  
 \begin{itemize}
  \item the space of the little $m$-cubes operad $\ca{C}^{(m)}(n)$ defined in \cite{May} for $c=\close$; 
  \item the empty set if $n=0$;
  \item the one-point space if $n=1$;
  \item in the case $s+t=n\geq 2$ with $s,t\geq 0$ such that $s$ colours $c_i$ are $\close$ and $t$ colours $c_j$ are $\open$, the space of configuration of $2s+t$  disjoint cubes $(C_1,...,C_{2s+t})$ in the standard cube $C_0\in \mathbb{R}^m$ such that $Sym(C_i)=C_{i+s}$ for $1\leq i\leq s$ and $Sym(C_i)=C_{i}$ for $2s+1\leq i\leq 2s+t$ and such that all the cubes $(C_1,...,C_{s})$ are in the upper half space.
 \end{itemize}
\end{de}
\begin{rmq}\label{rmq: half cube}
Because of the symmetry conditions imposed by $Sym$, we may think of $\SC_m(c_1,...,c_n;\open)$ as the configuration space of cubes $(C_1,...,C_s)$  and semi-cubes $(C_{s+1},...,C_{s+t})$ lying in the standard semi-cube $Half_+\cap C_0$. 
\end{rmq}
Similarly to the little $m$-cubes operad $\ca{C}^{(m)}$ the composition maps  
 \begin{align*}
  \circ_i: \SC_m(c_1,...,c_n;c) \times  \SC_m(d_1,...,d_r;c_i) \to  \SC_m(c_1,...,c_{i-1},d_1,...,d_r,c_{i+1},...,c_n;c)
 \end{align*}
 are defined as substitutions of cubes. 
We denote the resulting SC type operad by $\SC_m$.

\newcommand{\catpreop}{\mathbf{\Lambda}^\text{SC}}

\subsection{The {SC extended complete graph} operad}

We define the \emph{SC (or relative) extended complete graph} operad $\cgraph$.
 It is an SC type operad in the category of posets. We provide two filtrations by sub-operads $\{\cgraph_{m}\}_{m\geq 1}$ and $\{\cgraph'_{m}\}_{m\geq 1}$. Their closed part are isomorphic to $\ca{K}_{m}$ and their open part are isomorphic to  $\ca{K}_{m-1}$, where $\{\ca{K}_{m}\}_{m\geq 1}$ denotes the extended complete graph operad defined in \cite[Section 4.1]{BFV}.

 \subsubsection{Definition of $\cgraph$}
 
 Given $n$ colours $c_i\in \{\close, \open\}$, we denote by $\{\widetilde{c}_1,...,\widetilde{c}_n\}$ the set with 
 \begin{equation}\label{eq: notation ctilde}
  \widetilde{c}_i=\begin{cases}
                   i &\text{ if } c_i=\close;\\
                   \underline{i} &\text{ if } c_i=\open.
                  \end{cases}
 \end{equation}
 A colouring and an orientation $(\mu,\sigma)$ on the complete graph on $\{\widetilde{c}_1,...,\widetilde{c}_n\}$ is, for each $1\leq i<j\leq n$, a strict positive natural number  $\mu_{i,j}\in \mathbb{N}^{>0}$  and an orientation  $\sigma_{i,j}$ (that is, $\widetilde{c}_i\to \widetilde{c}_j$ or $\widetilde{c}_i\leftarrow \widetilde{c}_j$). 
 A \emph{monochromatic acyclic orientation} of a complete graph is a colouring and an orientation such that there exist no oriented cycles with the same colour, \emph{i.e.} there are no configurations of the form $\widetilde{c}_{i_1}\to \widetilde{c}_{i_2}\to \cdots \to \widetilde{c}_{i_k}\to \widetilde{c}_{i_1}$ with $\mu_{i_1,i_2}=\mu_{i_2,i_3}=\cdots =\mu_{i_{k-1},i_k}=\mu_{i_{k},i_1}$. 
 A \emph{marked} monochromatic acyclic orientation $(\mu,\sigma)^c$ is a monochromatic acyclic orientation $(\mu,\sigma)$ together with   a colour $c\in \{\close, \open\}$. %
 
 If there exists an $i$ such that $c_i=\open$, then we define $\cgraph(c_1,...,c_n;\close)$ as the empty set. 
 Otherwise, $\cgraph(c_1,...,c_n;c)$ is the set of the marked monochromatic acyclic orientations $(\mu,\sigma)^c$ of the complete graph on $\{\widetilde{c}_1,...,\widetilde{c}_n\}$. 
 
 The poset structure is given by 
  \begin{equation*}
 (\mu,\sigma)^c\leq (\mu',\sigma')^c\Leftrightarrow ~\forall i<j, \text{ or }(\mu_{i,j},\sigma_{i,j})= (\mu'_{i,j},\sigma'_{i,j})~\text{or }  \mu_{i,j}<\mu'_{i,j}. 
 \end{equation*}
 
  Given a permutation $\sigma\in \Sigma_n$ and an element $(\mu,\tau)^c\in \cgraph(c_1,...,c_n;c)$, the resulting element $\sigma\cdot(\mu,\tau)^c\in \cgraph(c_{\sigma^{-1}(1)},...,c_{\sigma^{-1}(n)};c)$ is given by permuting the numbers $i$ by $\sigma$ without changing neither the underline nor the orientation nor the colouring. For example, 
  the edges $\underline{i}\to j$ of $(\mu,\tau)^c$ with colours $\mu_{i,j}$ become the edges  $\underline{\sigma(i)}\to \sigma(j)$ with the same colours $\mu_{i,j}$. 
 
 The compositions 
 \begin{multline*}
  \gamma^{\cgraph}:\cgraph(c_1,...,c_n;c)\times\cgraph(c_{1,1},...,c_{1,k_1};c_{1})\times... \times  \cgraph(c_{n,1},...,c_{n,k_n};c_{n})
  \to \cgraph(c_{1,1},...,c_{n,k_n};c)
 \end{multline*}
send a tuple $(\alpha;\alpha_1,...,\alpha_n)$ of  $\cgraph(c_1,...,c_n;c)\times\cgraph(c_{1,1},...,c_{1,k_1};c_{1})\times\dots \times  \cgraph(c_{n,1},...,c_{n,k_n};c_{n})$ to an element in $\cgraph(c_{1,1},...,c_{n,k_n};c)$ obtained as follows. 
The sub complete graphs with vertices in the same block $\{c_{i,1},...,c_{i,k_i}\}$ is oriented and coloured as  $\alpha_i\in \cgraph(c_{i,1},...,c_{i,k_i};c_{i})$; the edges with vertices in two different blocks  are oriented and coloured as the edges between the corresponding vertices in  $\alpha\in \cgraph(c_1,...,c_n;c)$. 
For example, 
\begin{equation*}
\gamma^{\cgraph}\left(
\begin{tikzpicture}[>=stealth, arrow/.style={->,shorten >=3pt}, point/.style={coordinate}, pointille/.style={draw=red, top color=white, bottom color=red},scale=0.6,baseline=0ex,decoration={markings, mark=at position 0.65 with {\arrow[scale=1.2]{>}}}]%
 \coordinate (A) at (0,0);
 \coordinate (B) at (2,0);
  \draw [decoration={markings, mark=at position 0.65 with {\arrow[scale=1.5,blue]{>}}},postaction={decorate}] (A) -- ++(B) node [below,midway,blue] {\footnotesize{$2$}};
  \draw [black,fill=black] (A) circle (0.1)  ;
 \draw [black,fill=black] (B) circle (0.1)  ;
 \draw (A) node [above] {$\un{1}$};
 \draw (B) node [above] {$2$};
 \draw ($(B)+(.4,.45)$) node [above] {$\open$};
\end{tikzpicture}%
;
\begin{tikzpicture}[>=stealth, arrow/.style={->,shorten >=3pt}, point/.style={coordinate}, pointille/.style={draw=red, top color=white, bottom color=red},scale=0.6,baseline=0ex,decoration={markings, mark=at position 0.65 with {\arrow[scale=1.2]{>}}}]%
 \coordinate (A) at (0,0);
 \coordinate (B) at (2,0);
  \draw [decoration={markings, mark=at position 0.65 with {\arrow[scale=1.5,red]{>}}},postaction={decorate}] (A) -- ++(B) node [below,midway,red] {\footnotesize{$3$}};
  \draw [black,fill=black] (A) circle (0.1)  ;
 \draw [black,fill=black] (B) circle (0.1)  ;
 \draw (A) node [above] {${1}$};
 \draw (B) node [above] {$\un{2}$};
 \draw ($(B)+(.4,.45)$) node [above] {$\open$};
\end{tikzpicture}%
,
\begin{tikzpicture}[>=stealth, arrow/.style={->,shorten >=3pt}, point/.style={coordinate}, pointille/.style={draw=red, top color=white, bottom color=red},scale=0.6,baseline=0ex,decoration={markings, mark=at position 0.65 with {\arrow[scale=1.2]{>}}}]%
 \coordinate (A) at (0,0);
 \coordinate (B) at (2,0);
  \draw [postaction={decorate}] (B) -- (A) node [below,midway] {\footnotesize{$5$}};
  \draw [black,fill=black] (A) circle (0.1)  ;
 \draw [black,fill=black] (B) circle (0.1)  ;
 \draw (A) node [above] {$1$};
 \draw (B) node [above] {$2$};
\end{tikzpicture}%
\right) 
=
\begin{tikzpicture}[>=stealth, arrow/.style={->,shorten >=3pt}, point/.style={coordinate}, pointille/.style={draw=red, top color=white, bottom color=red},scale=0.6,baseline=0ex,decoration={markings, mark=at position 0.65 with {\arrow[scale=1.2]{>}}},baseline=3ex]%
 \coordinate (A) at (0,0);
 \coordinate (B) at (2,0);
 \coordinate (C) at (0,2);
 \coordinate (D) at (2,2);
  \draw [decoration={markings, mark=at position 0.65 with {\arrow[scale=1.5,blue]{>}}},postaction={decorate}] (A) -- (B) node [below,midway,blue] {\footnotesize{$2$}};
  \draw [decoration={markings, mark=at position 0.65 with {\arrow[scale=1.5,red]{>}}},postaction={decorate}] (A) -- (C) node [left,midway,red] {\footnotesize{$3$}};
  \draw [decoration={markings, mark=at position 0.65 with {\arrow[scale=1.5,blue]{>}}},postaction={decorate}] (A) -- (D) node [above,midway,blue,yshift=1mm] {\footnotesize{$2$}};
  \draw [decoration={markings, mark=at position 0.65 with {\arrow[scale=1.5,blue]{>}}},postaction={decorate}] (C) -- (B) node [below,midway,blue,yshift=-1mm] {\footnotesize{$2$}};
  \draw [postaction={decorate}] (B) -- (D) node [right,midway] {\footnotesize{$5$}};
  \draw [decoration={markings, mark=at position 0.65 with {\arrow[scale=1.5,blue]{>}}},postaction={decorate}] (C) -- (D) node [above,midway,blue] {\footnotesize{$2$}};
  \draw [black,fill=black] (A) circle (0.1)  ;
 \draw [black,fill=black] (B) circle (0.1)  ;
 \draw [black,fill=black] (C) circle (0.1)  ;
 \draw [black,fill=black] (D) circle (0.1)  ;
 \draw (A) node [left] {$1$};
 \draw (B) node [right] {$4$};
 \draw (C) node [left] {$\un{2}$};
 \draw (D) node [right] {$3$};
 \draw ($(D)+(.7,.1)$) node [above] {$\open$};
\end{tikzpicture}%
.
\end{equation*}

\subsubsection{Filtrations of $\cgraph$}  
We define two different filtrations of $\cgraph$ by sub-operads $(\cgraph_{m})_{m\geq 1}$ and $(\cgraph'_{m})_{m\geq 1}$.
\\ 
 For $m\geq 1$, the sub-operad $\cgraph_{m}\subset \cgraph$ is defined as follows.  
 The closed part  is defined by 
 \begin{equation*}
   \cgraph_{m}(\close,...,\close;\close) =\{ (\mu,\sigma)^{\close} \in \cgraph(\close,...,\close;\close) ~|~ \mu_{i,j}\leq m ~\forall~i<j\}. 
  \end{equation*}
  The non-closed part is defined, for an $(n+1)$-tuple of colours $(c_1,...,c_n;\open)$, by: 
  \begin{align}\label{de: compl graph altern}
   \cgraph_{m}(c_1,...,c_n;\open) =\{ (\mu,\sigma)^{\open} \in \cgraph(c_1,...,c_n;\open) ~|  \nonumber 
   &~ \mu_{i,j}\leq m &&\text{ if } c_i=c_j=\close, \\\nonumber
   &~ \mu_{i,j}\leq m-1 &&\text{ if }  c_i=c_j=\open, \\ 
   &~\mu_{i,j}\leq m-1 &&\text{ if }  \underline{i}\to j \text{ or } i\leftarrow \underline{j} \\ \nonumber
   &~\mu_{i,j}\leq m &&\text{ if }  i\to \underline{j}  \text{ or } \un{i} \leftarrow {j} \nonumber
   \}.
  \end{align}
  The second sub-operad $\cgraph'_{m}\subset \cgraph$ is obtained by exchanging the last two conditions on $\mu_{i,j}$ above. 
  Explicitly, 
 \begin{equation*}
   \cgraph'_{m}(\close,...,\close;\close) =\cgraph_{m}(\close,...,\close;\close),
  \end{equation*}
  and
  \begin{align}\label{de: compl graph}
   \cgraph'_{m}(c_1,...,c_n;\open) =\{ (\mu,\sigma)^{\open} \in \cgraph(c_1,...,c_n;\open) ~|  \nonumber 
   &~ \mu_{i,j}\leq m &&\text{ if } c_i=c_j=\close, \\\nonumber
   &~ \mu_{i,j}\leq m-1 &&\text{ if }  c_i=c_j=\open, \\ 
   &~ \mu_{i,j}\leq m &&\text{ if }  \underline{i}\to j  \text{ or } i\leftarrow \underline{j} \\ \nonumber 
   &~ \mu_{i,j}\leq m-1 &&\text{ if }  i\to \underline{j}  \text{ or } \un{i} \leftarrow {j} \nonumber
   \}.
  \end{align}

  \begin{rmq}
 For $m=1$ the conditions where $\mu_{i,j}\leq m-1$ cannot be satisfied. It follows that $\cgraph_{1}(c_1,...,c_n;\open)$ and $\cgraph'_{1}(c_1,...,c_n;\open)$ are empty whenever the tuple $(c_1,...,c_n)$ has more than one open colour.  
\end{rmq}

\subsection{Cellular decompositions of the Swiss Cheese operad}

The idea of cellular decomposition of operads comes from \cite{Berger:combinatorialmodel}. It consists in a cellular decomposition of each spaces that is compatible with the operad structure.

Recall from \cite[3.1]{Batanin-Berger-Lattice} that, given a topological space $X$ and a poset $\ca{A}$, 
an $\ca{A}$-\emph{cellulation} of $X$ is a functor $\Theta:\ca{A}\to \cat{Top}$ such that: 
\begin{enumerate}
 \item $\colim_{\alpha \in \ca{A}} \Theta(\alpha)\cong X$;\label{item 1}
 \item for each $\alpha\in \ca{A}$, the canonical map $\colim_{\beta<\alpha}\Theta(\beta)\to \Theta(\alpha)$ is  a closed cofibration;\label{item 3}
 \item for each $\alpha\in \ca{A}$, the ``cell``  $\Theta(\alpha)$ is contractible.\label{item 4}
\end{enumerate} 
Such an $\ca{A}$-cellulation provides a zig-zag of weak equivalences: 
 \begin{equation}\label{lem: weak equiv BA X}
 X\cong \colim_{\alpha \in \ca{A}} \Theta(\alpha) \longleftarrow \hocolim_{\alpha \in \ca{A}} \Theta(\alpha) \longrightarrow \hocolim_{\alpha \in \ca{A}} (*)\cong \Ba \ca{A},
  \end{equation}
where $\Ba \ca{A}$ denotes the realization of the nerve of the category $\ca{A}$.
Moreover, if $X$ and $\ca{A}$ are operads and if the cellular decomposition of $X$ is compatible with its operadic structure, then all the objects in \eqref{lem: weak equiv BA X} are operads and the weak equivalences are morphisms of operads. It is straightforward to check that this holds for SC type operads for which the notion of a compatible cellular decomposition is as follows.

\begin{de} Let $\ca{A}$ be a poset SC type operad.  
 A topological SC type operad $\ca{O}$ is called an $\ca{A}$-\emph{cellular operad} if, for each $(n+1)$-tuple of colours $(c_1,...,c_n;c)$, there is an $\ca{A}(c_1,...,c_n;c)$-cellulation of $\ca{O}(c_1,...,c_n;c)$, 
 \begin{equation*}
  \Theta_{c_1,...,c_n;c}: \ca{A}(c_1,...,c_n;c)\to \cat{Top},
 \end{equation*} 
   subject to the following two compatibilities. 
 \begin{enumerate}
  \item Compatibility with the $\Sigma_n$-action: 
$\Theta_{c_{\sigma^{-1}(1)},...,c_{\sigma^{-1}(n)};c}(\sigma\cdot \alpha)=\sigma\cdot \Theta_{c_1,...,c_n;c}(\alpha)$  for all  $\sigma\in \Sigma_n$  and $\alpha\in \ca{A}(c_1,...,c_n;c)$. 
 \item Compatibility with the operadic composition: 
 \begin{multline*}
 \gamma^{\ca{O}}\Big( \Theta_{c_1,...,c_n;c}(\alpha)\times \Theta_{c_{1,1},...,c_{1,k_1};c_1}(\alpha_1)\times \dots \times \Theta_{c_{n,1},...,c_{n,k_n};c_n}(\alpha_n)\Big)\\
 \subseteq \Theta_{c_{1,1},...,c_{n,k_n};c}\left(\gamma^{\ca{A}}(\alpha;\alpha_1,...,\alpha_n)\right),
 \end{multline*}
 for all variables $c,c_i,c_{i,j},\alpha,\alpha_i$, where $\gamma^{\ca{O}}$ and $\gamma^{\ca{A}}$ denote the composition maps of $\ca{O}$ and $\ca{A}$ respectively. 
 \end{enumerate}
 \end{de}
 
 In what follows, one shows that Berger's cellular decomposition of the little $m$-cubes operad (see \cite[Theorem 1.16]{Berger:combinatorialmodel} and \cite{BFV}) extends to $\SC_m$. 
\begin{theo}\label{th: decompo}
Let $m\geq 1$. The Swiss Cheese operad $\SC_m$ has the structure of an $\cgraph_{m}$-cellular operad as well as the  structure  of  an  $\cgraph'_{m}$-cellular operad.
 \end{theo}
\begin{proof}
We start by the $\cgraph'_{m}$-cellular decomposition; we explain at the end of the proof how to proceed for the other decomposition. 

We use the description of $\SC_m$ via cubes and semi-cubes given in Remark \ref{rmq: half cube}. The integer $m\geq 1$ is fixed. 
 
 \textbf{Notation:} 
 \textit{For $C_1$ either a cube or a semi-cube and $C_2$ either a cube or a semi-cube,  we write $C_1\square_{\mu} C_2$ if they are separated by a hyperplane  $H_i$ orthogonal to the $i$-th coordinate axis for some $i\leq \mu$, such that whenever there is no separating hyperplane $H_i$  for $i<\mu$, the left element $C_1$ lies in the \underline{negative} side of $H_{\mu}$ and $C_2$ lies in the \underline{positive} side of $H_{\mu}$.}
\\
 
 For $\alpha=(\mu,\sigma)^c \in  \cgraph'_m(c_1,...,c_k;c)$, we define $\SC_m(c_1,...,c_k;c)^{(\alpha)}$ to be the cell 
 \begin{align*}
  \{ (C_1,...,C_{k})\in \SC_m(c_1,...,c_k;c)~|~ C_i\square_{\mu_{i,j}} C_j \text{ if } \widetilde{c}_i\to \widetilde{c}_j \text{ and } 
   C_j\square_{\mu_{i,j}} C_i \text{ if } \widetilde{c}_i \leftarrow \widetilde{c}_j  \}.
 \end{align*}
 
 For example, consider the following configurations in $\SC_2(\open,\close;\open)$, 
 \begin{equation*}
X=\begin{tikzpicture}[>=stealth, arrow/.style={->,shorten >=3pt}, point/.style={coordinate}, pointille/.style={draw=red, top color=white, bottom color=red},scale=.6,baseline=4ex]
  \draw [black,fill=mygray2] (0,0) rectangle (6,3)  ;
  \draw [black,fill=mygray2] (1,0) rectangle (3,1)  ;
  \draw [black,fill=mygray] (4,.5) rectangle (5,2) ;
  \draw [line width=1.2,-] (0,0) -- (6,0);
  \draw (2,.5) node  {$\underline{1}$};
  \draw (4.5,1.25) node  {$2$};
  \draw [line width=.9,-] (3.5,3.5) -- (3.5,-.5);
  \draw (4,3.5) node  {$H_1$};
\end{tikzpicture}
~,~ 
Y=
\begin{tikzpicture}[>=stealth, arrow/.style={->,shorten >=3pt}, point/.style={coordinate}, pointille/.style={draw=red, top color=white, bottom color=red},scale=.6,baseline=4ex]
  \draw [black,fill=mygray2] (0,0) rectangle (6,3)  ;
  \draw [black,fill=mygray2] ($(1,0)+(2,0)$) rectangle ($(3,1)+(2,0)$)  ;
  \draw [black,fill=mygray] ($(4,.5)+(-3,0)$) rectangle ($(5,2)+(-3,0)$) ;
  \draw [line width=1.2,-] (0,0) -- (6,0);
  \draw ($(2,.5)+(2,0)$) node  {$\underline{1}$};
  \draw ($(4.5,1.25)+(-3,0)$) node  {$2$};
  \draw [line width=.9,-] (2.5,3.5) -- (2.5,-.5);
  \draw (3,3.5) node  {$H_1$};
\end{tikzpicture}
\text{ and } 
Z= 
\begin{tikzpicture}[>=stealth, arrow/.style={->,shorten >=3pt}, point/.style={coordinate}, pointille/.style={draw=red, top color=white, bottom color=red},scale=.6,baseline=4ex]
  \draw [black,fill=mygray2] (0,0) rectangle (6,3)  ;
  \draw [black,fill=mygray2] (1,0) rectangle (3,1)  ;
  \draw [black,fill=mygray] ($(4,.5)+(-1.5,.775)$) rectangle ($(5,2)+(-1.5,.775)$) ;
  \draw [line width=1.2,-] (0,0) -- (6,0);
  \draw (2,.5) node  {$\underline{1}$};
  \draw ($(4.5,1.25)+(-1.5,.775)$) node  {$2$};
  \draw [line width=.9,-] (-.5,1.15) -- (6.5,1.15);
  \draw (-.5,1.5) node  {$H_2$};
\end{tikzpicture} ,
\end{equation*} 
and the following elements 
\begin{equation*}
 \alpha= 
 \begin{tikzpicture}[>=stealth, arrow/.style={->,shorten >=3pt}, point/.style={coordinate}, pointille/.style={draw=red, top color=white, bottom color=red},scale=0.6,baseline=0ex,decoration={markings, mark=at position 0.65 with {\arrow[scale=1.2]{>}}}]%
 \coordinate (A) at (0,0);
 \coordinate (B) at (2,0);
  \draw [decoration={markings, mark=at position 0.65 with {\arrow[scale=1.5,black]{>}}},postaction={decorate}] (A) -- ++(B) node [below,midway,black] {\footnotesize{$1$}};
  \draw [black,fill=black] (A) circle (0.1)  ;
 \draw [black,fill=black] (B) circle (0.1)  ;
 \draw (A) node [above] {$\un{1}$};
 \draw (B) node [above] {$2$};
 \draw ($(B)+(.4,.45)$) node [above] {$\open$};
\end{tikzpicture}%
~,~ \beta= 
 \begin{tikzpicture}[>=stealth, arrow/.style={->,shorten >=3pt}, point/.style={coordinate}, pointille/.style={draw=red, top color=white, bottom color=red},scale=0.6,baseline=0ex,decoration={markings, mark=at position 0.65 with {\arrow[scale=1.2]{<}}}]%
 \coordinate (A) at (0,0);
 \coordinate (B) at (2,0);
  \draw [decoration={markings, mark=at position 0.65 with {\arrow[scale=1.5,black]{<}}},postaction={decorate}] (A) -- ++(B) node [below,midway,black] {\footnotesize{$1$}};
  \draw [black,fill=black] (A) circle (0.1)  ;
 \draw [black,fill=black] (B) circle (0.1)  ;
 \draw (A) node [above] {$\un{1}$};
 \draw (B) node [above] {$2$};
 \draw ($(B)+(.4,.45)$) node [above] {$\open$};
\end{tikzpicture}%
\text{ and } \gamma= \begin{tikzpicture}[>=stealth, arrow/.style={->,shorten >=3pt}, point/.style={coordinate}, pointille/.style={draw=red, top color=white, bottom color=red},scale=0.6,baseline=0ex,decoration={markings, mark=at position 0.65 with {\arrow[scale=1.2]{>}}}]%
 \coordinate (A) at (0,0);
 \coordinate (B) at (2,0);
  \draw [decoration={markings, mark=at position 0.65 with {\arrow[scale=1.5,black]{>}}},postaction={decorate}] (A) -- ++(B) node [below,midway,black] {\footnotesize{$2$}};
  \draw [black,fill=black] (A) circle (0.1)  ;
 \draw [black,fill=black] (B) circle (0.1)  ;
 \draw (A) node [above] {$\un{1}$};
 \draw (B) node [above] {$2$};
 \draw ($(B)+(.4,.45)$) node [above] {$\open$};
\end{tikzpicture}%
\text{ that form } \ca{RK}'_2(\open,\close;\open).
\end{equation*}

The cell $\SC_2(\open,\close;\open)^{(\alpha)}$ is made of configurations of type $X$; the cell 
$\SC_2(\open,\close;\open)^{(\beta)}$ is made of configurations of type $Y$; and, 
the cell  $\SC_2(\open,\close;\open)^{(\gamma)}$ is made of configurations of type $X$, $Y$ and configurations of type  $Z$ (which intersect  with configurations of type $X$ or $Y$  whenever $C_1$ and $C_2$ are separated by $2$ hyperplanes $H_1$ and $H_2$). 
\\

 As a remark, note that in $\SC_m$, whenever $C_1$ is a semi-cube and $C_2$ is a cube, if $H_m$ exists, then $C_1$ lies in the negative side of $H_m$ (e.g. $Z$). 
 From this one can see that, with this definition of the cells, the condition  $\mu_{i,j}\leq m-1$ if $i\to \underline{j}$ given in the definition of $\ca{RK}'_m$  is necessary for ensuring the contractibility of the cells. 
Indeed, if one removes the above condition, then there exists $\delta_m= \begin{tikzpicture}[>=stealth, arrow/.style={->,shorten >=3pt}, point/.style={coordinate}, pointille/.style={draw=red, top color=white, bottom color=red},scale=0.6,baseline=0ex,decoration={markings, mark=at position 0.65 with {\arrow[scale=1.2]{<}}}]%
 \coordinate (A) at (0,0);
 \coordinate (B) at (2,0);
  \draw [decoration={markings, mark=at position 0.65 with {\arrow[scale=1.5,black]{<}}},postaction={decorate}] (A) -- ++(B) node [below,midway,black] {\footnotesize{$m$}};
  \draw [black,fill=black] (A) circle (0.1)  ;
 \draw [black,fill=black] (B) circle (0.1)  ;
 \draw (A) node [above] {$\un{1}$};
 \draw (B) node [above] {$2$};
 \draw ($(B)+(.4,.45)$) node [above] {$\open$};
\end{tikzpicture}%
$
that indexes the space $\SC_m(\open,\close;\open)^{(\delta_m)}$, which turns out to be  homotopic to the $m-2$-sphere (for $m=2$, one has  
$\SC_2(\open,\close;\open)^{(\delta_2)}= \SC_2(\open,\close;\open)^{(\alpha)} \sqcup \SC_2(\open,\close;\open)^{(\beta)}$). 
\\

 Our definition recovers that of Berger when considering the open part and the closed part of $\SC_m$ separately. 
 The main input of our construction, then, resides in the interaction between cubes and semi-cubes.  
 At first sight it could appear useless to consider separating hyperplanes between cubes and semi-cubes since contractibility is not hindered by their relative positions (for example $\SC_m(\open,\close;\open)$ is contractible).  
 However it is important to do so for the operadic composition of the cells.  
This is because, in a configuration of cubes and semi-cubes, if one substitutes a semi-cube, then a cube may appear; the position of such a cube has to be compared with that of the other cubes. 
For instance, consider the following substitution: 
 \begin{equation*}
\begin{tikzpicture}[>=stealth, arrow/.style={->,shorten >=3pt}, point/.style={coordinate}, pointille/.style={draw=black, densely dotted, line width=1.2},scale=.6,baseline=+4ex]
  \draw [black,fill=mygray2] (0,0) rectangle (6,3)  ;
  \draw [black,fill=mygray2] (1,0) rectangle (3,1)  ;
  \draw [black,fill=mygray] (2,1.5) rectangle (3,2.5)  ;
  \draw [black,fill=mygray] (4,.5) rectangle (5,2) ;
  \draw [line width=1.2,-] (0,0) -- (6,0);
  \draw (2,.5) node  {$\underline{1}$};
  \draw (2.5,2) node  {$2$};
  \draw (4.5,1.25) node  {$3$};
  \draw [line width=1.2,-] (3.5,4) -- (3.5,-1);
  \draw (4,4) node  {$H_1$};
  \draw [pointille] (-1,1.25) -- (3.25,1.25);
  \draw (-1,1.6) node  {$H_2$};
\end{tikzpicture}
\circ_{\un{1}}
\begin{tikzpicture}[>=stealth, arrow/.style={->,shorten >=3pt}, point/.style={coordinate}, pointille/.style={draw=red, top color=white, bottom color=red},scale=.6,baseline=4ex]
  \draw [black,fill=mygray2] (0,0) rectangle (6,3)  ;
  \draw [black,fill=mygray2] (1,0) rectangle (3,1)  ;
  \draw [black,fill=mygray] (4,.5) rectangle (5,2) ;
  \draw [line width=1.2,-] (0,0) -- (6,0);
  \draw (2,.5) node  {$\underline{1}$};
  \draw (4.5,1.25) node  {$2$};
\end{tikzpicture}
=
\begin{tikzpicture}[>=stealth, arrow/.style={->,shorten >=3pt}, point/.style={coordinate}, pointille/.style={draw=black, densely dotted, line width=1.2},scale=.6,baseline=4ex]
  \draw [black,fill=mygray2] (0,0) rectangle (6,3)  ;
  \draw [black,fill=mygray2] (1.25,0) rectangle (1.75,.35)  ;
  \draw [black,fill=mygray] (2.35,.25) rectangle (2.65,.75) ;
  \draw [black,fill=mygray] (2,1.5) rectangle (3,2.5)  ;
  \draw [black,fill=mygray] (4,.5) rectangle (5,2) ;
  \draw [line width=1.2,-] (0,0) -- (6,0);
  \draw (1,.5) node  {$\underline{1}$};
  \draw (2.5,2) node  {$3$};
  \draw (4.5,1.25) node  {$4$};
  \draw (3,.5) node  {$2$};
  \draw [line width=1.2,-] (3.5,4) -- (3.5,-1);
  \draw (4,4) node  {$H_1$};
  \draw [pointille] (-1,1.25) -- (3.25,1.25);
  \draw (-1,1.65) node  {$H_2$};
\end{tikzpicture},
\end{equation*}
 then, in the resulting configuration (the right-hand side standard semi-cube), knowing the position of $2$ relative to $3$ and $4$ requires the knowledge of  the position of $\un{1}$ relative to $2$ and $3$ in the first term (the left-hand side standard semi-cube).

   For the contractibility of the cells,  both closed and open cells are known as being contractible; the  same argument as \cite[Theorem 1.16]{Berger:combinatorialmodel}  shows that open/closed cells also are. 
   
   The fact that $\colim_{\alpha \in \ca{RK}'_m(c_1,...,c_k;c)} \SC_m(c_1,...,c_k;c)^{(\alpha)}\cong \SC_m(c_1,...,c_k;c)$, essentially follows from the $2$ following facts.  
   \begin{itemize}
    \item If $x\in \SC_m(c_1,...,c_k;c)$ then there exists an $\alpha \in  \ca{RK}'_m(c_1,...,c_k;c)$ such that $x\in \SC_m(c_1,...,c_k;c)^{(\alpha)}$. This is immediate since any two cubes/semi-cubes are separated by a hyperplane.   For example, if $C_i$ is a semi-cube and $C_j$ a cube, then  $\alpha=(\mu,\sigma)^c$ may be chosen such that  $\mu_{i,j}=m$ and $\sigma_{i,j}= \un{i}\to j$. 
    \item  If  $\alpha=(\mu,\sigma)^c$ and $\beta=(\mu',\sigma')^c$ are incomparable (neither $\alpha\leq \beta$ nor $\alpha\geq \beta$) and are such that there exists  $x=(C_1,...,C_k)\in \SC_m(c_1,...,c_k;c)^{(\alpha)}\cap \SC_m(c_1,...,c_k;c)^{(\beta)}$, then  $x\in \SC_m(c_1,...,c_k;c)^{(\gamma)}$ for a $\gamma$ such that $\gamma < \alpha$ and $\gamma< \beta$.  
   This follows from the observation that, for each $i<j$ such that $(\mu_{i,j},\sigma_{i,j})$ and $(\mu'_{i,j},\sigma'_{i,j})$ are not comparable, there exists  $\mu''_{i,j}< \mu_{i,j}=\mu'_{i,j}$ such that $C_i\square_{\mu''_{i,j}}C_j$ or  $C_j\square_{\mu''_{i,j}}C_i$; let us denote by $\sigma''_{i,j}$ the corresponding orientation. The element $\gamma=(\mu'',\sigma'')^c$ is given by $(\mu''_{i,j},\sigma''_{i,j})$ as above  whenever $i<j$ indexes not comparable components and by $(\mu''_{i,j},\sigma''_{i,j})=\min\{(\mu_{i,j},\sigma_{i,j}),(\mu'_{i,j},\sigma'_{i,j})\}$ otherwise. 
 \end{itemize}

  The other cellular decomposition (indexed by $\cgraph_m$)  is obtained by {exchanging}, in \textbf{Notation} above, the terms \underline{\emph{negative}} and \underline{\emph{positive}}. In what concerns the closed and the open parts, such an exchange is not relevant. However, as we have remarked before, whenever $C_1$ is a semi-cube and $C_2$ is a cube, if $H_m$ exists, then $C_1$ lies in the negative side of $H_m$, hence 
  the condition  $\mu_{i,j}\leq m-1$ if $i\leftarrow \underline{j}$ given in the definition of $\ca{RK}_m$. 
\end{proof}

\section{The operad $\ca{RL}$}\label{sec: RL}
\subsection{Definition of the operad $\ca{RL}$}\label{sec: RL def}

We describe a coloured SC-operad $\ca{RL}$ in the category of sets, $\cat{Set}$.

The operad $\ca{RL}$ has  two natural filtrations by sub operads $\ca{RL}_m$ and $\ca{RL}'_m$ for $m\geq1$. For each $m\geq 1$, we can think of $\ca{RL}_m$  (resp. $\ca{RL}'_m$) as mixes between the sub operads $\ca{L}_m$ and $\ca{L}_{m-1}$ of the lattice paths operad $\ca{L}$ introduced in \cite{Batanin-Berger-Lattice}. 
The operad   $\ca{RL}$ has $2$ types of colours, the closed colours $\mathbb{N}$ and the open  colours $\un{\mathbb{N}}$,  while $\ca{L}$ has one type of colours  $\mathbb{N}$. 
\\

We generalize \cite[Section 2]{Batanin-Berger-Lattice} which serves as a basis for this section. In particular we refer to \textit{loc. cit.} for the definition of the category of bipointed small categories $\cat{Cat}_{*,*}$, ordinals $[n]$  and for the tensor product of ordinals $[i]\ot [j]$. Decorating each object of the ordinal $[n]$ with an underline gives the  \emph{underlined ordinal} $[\un{n}]$.   
One denotes by $\un{\mathbb{N}}$ the set of the natural numbers decorated by an underline; for $\un{n}\in \un{\mathbb{N}}$ and $k\in \mathbb{Z}_{\geq -n}$ one lets $\un{n}+k:=\un{n+k}$. 
Let us denote by $ev:\mathbb{N}\sqcup \un{\mathbb{N}}\to \mathbb{N}$ the map defined by 
$ev(\n)=n$ if $\n=n$ or $\n=\un{n}$.

\paragraph{Definition of $\ca{RL}$.}

The set of colours of $\ca{RL}$ is $Col= Col_{\close} \sqcup Col_{\open}$ 
where $Col_{\close}:=\mathbb{N}$ and $Col_{\open}:=\un{\mathbb{N}}$. Hence, $n\in Col_{\close}$ whereas $\underline{n}\in Col_{\open}$.

For a $(k+1)$-tuple of colours $(\n_1,...,\n_k;\n)$ in $Col$, the set $\ca{RL}(\n_1,...,\n_k;\n)$ is defined as:
\begin{itemize}
 \item the empty set $\emptyset$, if $\n\in Col_{\close}$ and if there is an $i$ such that $\n_i\in Col_{\open}$; 
 \item $\textbf{Cat}_{*,*}([\n+1],[\n_1+1]\ot [\n_2+1]\ot \dots \ot [\n_k+1])$, otherwise. 
\end{itemize}
The substitutions maps are the natural extension to that of $\ca{L}$, that is, they are given by tensor and composition in $\cat{Cat}_{*,*}$. 

\begin{rmq}
Accordingly, one recovers the lattice path operad,  
\begin{align*}
 \ca{RL}(n_1,...,n_k;n)=\ca{L}(n_1,...,n_k;n)= \textbf{Cat}_{*,*}([n+1],[n_1+1]\ot [n_2+1]\ot \dots \ot [n_k+1]),
\end{align*} 
for all $(k+1)$-tuple $(n_1,...,n_k;n)$ of colours in $Col_{\close}=\mathbb{N}$.  
\end{rmq}

For instance, an element $x\in \ca{RL}(n_1,\un{n_2};\no)$ 
is a functor  $x: [\un{n+1}] \to [n_{1}+1]\ot [\un{n_2+1}]$  that sends $(0,\un{n+1})$ on $((0,\un{0}),(n_{1}+1,\un{n_2+1}))$ and is determined by the image of the $n$ remaining objects of $[\un{n+1}]$ and the morphisms into the lattice $[n_{1}+1]\ot [\un{n_2+1}]$. 

\begin{exple}\label{ex: lattice 1}
 The following lattice path $x$ belongs to $\ca{RL}(3,\un{2};\un{3})$:
\begin{equation}\label{eq: lattice ex} 
 \scalebox{0.75}{
\begin{tikzpicture}[>=stealth,thick,draw=black!55, arrow/.style={->,shorten >=1pt}, point/.style={coordinate}, pointille/.style={draw=black!43, top color=white, bottom color=gray},scale=0.8,baseline=1ex]
\matrix[row sep=5mm,column sep=5mm,ampersand replacement=\&]
{
 \node (30) {$(0,\un{3})$};		\& \node (31){} ; 		\&\node (32) {};		\& \node (33) {$\bullet$};	\& \node (34) {$x(\un{4})$};\\
 \node (20) {$\cdot$};		\& \node (21){$\bullet$} ;  	\&\node (22) {$\bullet$};	\& \node (23) {$x(\un{2})=x(\un{3})$};	\& \node (24) {}; \\
 \node (10) {};		\& \node (11){$x(\un{1})$} ;  	\&\node (12) {};		\& \node (13) {};		\& \node (14) {};\\
 \node (00) {$x(\un{0})$};	\& \node (01){$\bullet$} ;  	\&\node (02) {};		\& \node (03) {};		\& \node (04) {$(4,\un{0})$};\\
}; 
\path
   	  (00)     edge[above,->]      node {$1$}  (01)
   	  (01)     edge[left,->]      node {$\un{2}$}  (11)
   	  (11)     edge[left,->]      node {$\un{2}$}  (21)
   	  (21)     edge[above,->]      node {$1$}  (22)
   	  (22)     edge[above,->]      node {$1$}  (23)
   	  (23)     edge[left,->]      node {$\un{2}$}  (33)
   	  (33)     edge[above,->]      node {$1$}  (34)
   	  (00)     edge[pointille,-]      node {}  (30)
   	  (21)     edge[pointille,-]      node {}  (31)
   	  (22)     edge[pointille,-]      node {}  (32)
   	  (02)     edge[pointille,-]      node {}  (22)
   	  (03)     edge[pointille,-]      node {}  (23)
   	  (04)     edge[pointille,-]      node {}  (34)
   	  (01)     edge[pointille,-]      node {}  (04)
   	  (11)     edge[pointille,-]      node {}  (14)
   	  (23)     edge[pointille,-]      node {}  (24)
   	  (10)     edge[pointille,-]      node {}  (11)
   	  (20)     edge[pointille,-]      node {}  (21)
   	  (30)     edge[pointille,-]      node {}  (33)
 	  ; 
\end{tikzpicture}}
\end{equation}
\end{exple}
\paragraph{Integer-strings representation.} 
In \cite[paragraph 2.2]{Batanin-Berger-Lattice} a description of $\ca{L}$  in terms of integer-strings is given. 
One has the obvious similar bijective correspondence  for $\ca{RL}$, where one considers natural numbers and underlined natural numbers; the latter correspond to the open colours. 
 Additionally, we put an extra label according to the nature of the output colour. 
Explicitly, given a lattice path $x\in \ca{RL}(\n_1,...,\n_k;\n)$,  one runs through it starting from $x(0)$ (or $x(\un{0})$) and ending at $x(\n+1)$. Along the way, each time one meets an edge that is  parallel to the $i$-axis one writes down $i$   if $\n_i\in Col_{\close}$ and $\un{i}$  if $\n_i\in Col_{\open}$; one writes down a vertical bar each time one meets an  $x(a)$ for $1\leq ev(a)\leq ev(\n)$. One adds an $\open$ if the output colour belongs to $Col_{\open}$. 
 
 The $ev(\n)$ vertical bars of the integer-string $x$ subdivide it into $ev(\n)+1$ (possibly empty) \emph{substrings}. 
 The substrings are indexed by $[ev(\n)]$. 
 \begin{exple}\label{ex: integer-string}
 The lattice path $x$ from \eqref{eq: lattice ex} corresponds to the integer-string $(1\un{2}|\un{2}11||\un{2}1)^{\open}$.
\end{exple}
 \begin{exple}
 $(121)^{\open}\in \ca{RL}(1,0;\un{0})$  whereas $ (121)\in \ca{RL}(1,0;0)$. 
\end{exple}

Let us exhibit the corresponding composition on integer-string representations via an example.
\begin{exple}
 \begin{equation*}
 (1\un{2}|1\un{4}\un{2}31||\un{2}\un{4})^{\open}\circ_{\un{2}} (1\un{3}|21\un{3}|\un{3}1)^{\open}=(1{2\un{4}}|1\un{6}{32\un{4}}51||{\un{4}2}\un{6})^{\open}.
\end{equation*}
We have renumbered the integer-string $(1\un{2}|1\un{4}\un{2}31||\un{2}\un{4})^{\open}$ by increasing the integers greater than $\un{2}$ by $2$ (=one less than the number of integers in the second integer-string): $(1\un{2}|1\un{6}\un{2}51||\un{2}\un{6})^{\open}$. We have increased the integers of the second integer-string by $1$ (= one less than the value of $\un{2}$):  $(2\un{4}|32\un{4}|\un{4}2)^{\open}$. Finally, we have has substituted the three occurrences of $\un{2}$ by the three substrings  $2\un{4}$,  $32\un{4}$ and  $\un{4}2$. 
\end{exple}

The action of the symmetric group $\sigma\cdot x\in \ca{RL}(\n_{\sigma^{-1}(1)},...,\n_{\sigma^{-1}(k)};\n)$ is obtained by permuting the number $i$ (resp. $\un{i}$) of the string-integer representation of $x$ by the number $\sigma(i)$ (resp. $\un{\sigma(i)}$). 
\begin{exple}
For $x=(1\un{2}|3\un{2}11||\un{2}1)^{\open}$ and $\sigma(1)=2$, $\sigma(2)=3$, $\sigma(3)=1$, 
one has $\sigma\cdot x=(2\un{3}|1\un{3}22||\un{3}2)^{\open}$. 
\end{exple}

\paragraph{The underlying category of $\ca{RL}$.} 
It follows directly from \cite[paragraph 2.4]{Batanin-Berger-Lattice} and the definition of $\ca{RL}$ that 
the two underlying sub-categories $(\ca{RL})_u^{\close}$ and $(\ca{RL})_u^{\open}$ are (canonically isomorphic to) the simplicial category $\triangle$. 
Thus, for each $k\geq 0$, one has a functor
\begin{align*}
 \ca{RL}(-,...,-;-) : (\triangle^{\text{op}})^k\times \triangle \to \cat{Set},
\end{align*} 
that is, a multisimplicial/cosimplicial set.

\paragraph{Dual interpretation of $\ca{RL}$.}\label{sec: dual interp}
For later use, let us mention that $\ca{RL}$ has a dual interpretation as given in \cite[Lemma 2.7]{Batanin-Berger-Lattice}. In particular, an element $x\in \ca{RL}(\n_1,...,\n_k;\n)$ determines a $k$-tuple $(x_1,...,x_k)\in \triangle([\n_1],[\n])\times \cdots \times \triangle([\n_k],[\n])$. 
The value of $x_i(r)$ corresponds to the substring of the integer-string $x$ in which the $r+1$-th occurrence of $i$ (or $\un{i}$)  appears.  We refer to \textit{loc. cit.} for more details.  
\begin{exple}
The integer-string $x$  from Example \ref{ex: integer-string} determines the pair $(x_1,x_2)\in \triangle([3],[\un{3}])\times \triangle([\un{2}],[\un{3}])$ given by 
  $x_1:(0,1,2,3)\mapsto (\un{0},\un{1},\un{1},\un{3}) \text{ and } x_2:(\un{0},\un{1},\un{2})\mapsto (\un{0},\un{1},\un{3})$. 
\end{exple}

\paragraph{Filtrations of $\ca{RL}$ by sub operads.}

Let us define two maps 
\begin{equation*}
 c_{i,j},\hat{c}_{i,j}: \ca{RL}(\n_1,...,\n_k;\n)\to \mathbb{N}.
\end{equation*}
The map $ c_{i,j}$ is defined as in \cite[paragraph 2.9]{Batanin-Berger-Lattice} and will be used to defined the complexity index for the closed and the open parts of $\ca{RL}$. The map $\hat{c}_{i,j}$ will be used to control the interaction between the closed and the open part. 
\\

For $1\leq i<j\leq k$, we denote by  
\begin{equation*}
 \phi_{ij}:\ca{RL}(\n_1,...,\n_k;\n)\to \ca{RL}(\n_i,\n_j;\n)
\end{equation*}
the projection induced by the canonical projection 
\begin{equation*}
 p_{ij}:  [\n_1+1]\ot \cdots \ot [\n_k+1] \to [\n_i+1]\ot [\n_j+1].
\end{equation*}

For  $x\in \ca{RL}(\n_1,...,\n_k;\n)$ and $1\leq i<j\leq k$, we define $c_{ij}(x)$ as the number of changes of directions in the lattice paths $\phi_{ij}(x)$. 

The second number $\hat{c}_{i,j}(x)$ is defined as follows. 
Recall that if $x\in \ca{RL}(\n_1,...,\n_k;\n)$, then its integer-string representation is, in particular, a sequence of numbers (underlined or not) between $1$ and $k$. For $1\leq i\leq k$, we set $i^-$ (resp. $\un{i}^-$) to be the first occurrence of $i$ (resp. $\un{i}$) in the integer-string representation. Equivalently, $i^-$ (resp. $\un{i}^-$) is the first edge of the lattice $x$ which is in the $i$-th direction. 
We write $r^-<s^-$ if the element $r^-$  precedes $s^-$. 

For $1\leq i<j\leq k$, we set:
\begin{equation}\label{eq: filtra}
           \hat{c}_{i,j}(x)= \begin{cases}
            c_{i,j}(x)  &\text{ if } ~i^-<\un{j}^-;\\
	    c_{i,j}(x)+1 &\text{ if } ~i^->\un{j}^-;\\
	    c_{i,j}(x)+1  &\text{ if } ~\un{i}^-<j^-;\\
	    c_{i,j}(x) &\text{ if } ~\un{i}^->j^-.
           \end{cases}
\end{equation}

For $m\geq 1$, we define $\ca{RL}_m(\n_1,...,\n_k;\n)$ as the set of elements $x \in \ca{RL}(\n_1,...,\n_k;\n)$ satisfying the three conditions:
\begin{align*}
 \max_{(i,j)} c_{i,j}(x)\leq m,~~~ 
 \max_{(\un{i},\un{j})} c_{i,j}(x)\leq m-1 ~ \text{ and } 
 \max_{(\un{i},j)\text{ or }(i,\un{j})} \hat{c}_{i,j}(x) \leq m.
\end{align*}

  Changing the number defined in \eqref{eq: filtra} to:
  \begin{equation}\label{eq: filtra2}  
           \hat{c}'_{i,j}(x)= \begin{cases}
            c_{i,j}(x)+1  &\text{ if } ~i^-<\un{j}^-;\\
	    c_{i,j}(x) &\text{ if } ~i^->\un{j}^-;\\
	    c_{i,j}(x)  &\text{ if } ~\un{i}^-<j^-;\\
	    c_{i,j}(x)+1 &\text{ if } ~\un{i}^->j^-,
           \end{cases}
\end{equation}
provides  another filtration of $\ca{RL}$ by sub-operads $\ca{RL}'_m$. Explicitly,   
 $\ca{RL}'_m(\n_1,...,\n_k;\n)$ is the set of elements $x \in \ca{RL}(\n_1,...,\n_k;\n)$ satisfying the three conditions:
\begin{align*}
 \max_{(i,j)} c_{i,j}(x)\leq m,~~~ 
 \max_{(\un{i},\un{j})} c_{i,j}(x)\leq m-1 ~ \text{ and } 
 \max_{(\un{i},j)\text{ or }(i,\un{j})} \hat{c}'_{i,j}(x) \leq m.
\end{align*}

A relation between these two filtrations and the two cellular decompositions of $\ca{SC}_m$ from Theorem  \ref{th: decompo} is given in the next section.

\subsection{The operad $\Coend_{\ca{RL}_m}(\delta_{})$ as an SC type operad}

Given a functor $\delta:\catsimp\to \cat{C}$ where $\cat{C}$ is a \emph{monoidal model category}, by following the method developed in \cite[Sections 3.5-3.6]{Batanin-Berger-Lattice}, we construct a zig-zag of weak equivalences of operads 
  \begin{equation}\label{eq: zig-zag}
  \Coend_{\ca{RL}_m}(\delta) \longleftarrow \Coend_{\widehat{\ca{RL}}_m}(\delta) \longrightarrow \Ba_{\delta} \cgraph_{m},
\end{equation}
whenever $\delta$ satisfies some conditions. Here, $\Ba_{\delta} \ca{A}$ denotes the $\delta$-realization of the nerve of the category $\ca{A}$. 
The intermediate operad $\widehat{\ca{RL}}_m$ is defined using homotopy colimits in $\cat{C}$ applied on a decomposition of $\ca{RL}_m$ indexed by $\cgraph_{m}$.  
\\

We use most of the material from and the same conventions as in \cite[Sections 3.5-3.6]{Batanin-Berger-Lattice}. 
In particular, we require $\delta$ to be a  \emph{standard system of simplices}. This confers on  homotopy colimits good properties, such as  compatibility with the symmetric monoidal structure of $\cat{C}$. 
The following functors $\delta_{yon}$, $\delta_{\text{Top}}$ and $\delta_{\ring}$ are such standard system of simplices.  

Let 
\begin{equation*}
\scalebox{1}{
\begin{tikzpicture}[>=stealth,thick,draw=black!50, arrow/.style={->,shorten >=1pt}, point/.style={coordinate}, pointille/.style={draw=red, top color=white, bottom color=red},scale=0.8,baseline=-.6ex]
\matrix[row sep=15mm,column sep=12mm,ampersand replacement=\&]
{
 \node (00) {$\delta_{Top}: \triangle$}; \& \node (01){$\cat{Set}^{\catsimpop}$} ;\& \node (02){$\cat{Top}$} ;\\
}; 
\path
	  
   	  (00)     edge[above,->]      node {$\delta_{yon}$}  (01)
 	  (01)     edge[above,arrow,->]      node {$|-|$}  (02)
 	  ; 
\end{tikzpicture}}
\text{ and } 
\scalebox{1}{
\begin{tikzpicture}[>=stealth,thick,draw=black!50, arrow/.style={->,shorten >=1pt}, point/.style={coordinate}, pointille/.style={draw=red, top color=white, bottom color=red},scale=0.8,baseline=-.6ex]
\matrix[row sep=15mm,column sep=12mm,ampersand replacement=\&]
{
 \node (00) {$\delta_{\ring}:\triangle$}; \& \node (01){\cat{Set}$^{\catsimpop}$} ;\& \node (02){$\cat{Ch}(\ring)$} ;\\
}; 
\path
	  
   	  (00)     edge[above,->]      node {$\delta_{yon}$}  (01)
 	  (01)     edge[above,arrow,->]      node {$C_*(-;\ring)$}  (02)
 	  ; 
\end{tikzpicture}} 
\end{equation*}
be the two functors where:
\begin{itemize}
 \item $\delta_{yon}([n])=\Hom_{\triangle}(-,[n])$ is the Yoneda functor;
 \item $|-|:\cat{Set}^{\catsimpop} \to \cat{Top}$ is the geometric realization;  and,
 \item $C_*(-;\ring): \cat{Set}^{\catsimpop} \to \cat{Ch}(\ring)$ is the normalized chain complex. 
\end{itemize}

Let us recall  that for $\delta=(\delta_{\close},\delta_{\open})$ with $\delta_{\close},\delta_{\open}: \catsimp \to \cat{C}$, the functor $ \xi(\ca{RL}_m)_{c_1,...,c_k;c}(\delta)$  denotes the $\delta$-realization: 
\begin{align*}
 \xi(\ca{RL}_m)_{c_1,...,c_k;c}(\delta)(n)= \ca{RL}_m(\underbrace{-,\dots,-}_k;n)\ot_{\catsimp^k} \delta_{c_1}(-)\ot \cdots \ot \delta_{c_k}(-),
\end{align*}
where we use implicitly the strong monoidal functor $\cat{Set}\to \cat{C}$,  $E\mapsto \coprod_{e\in E} 1_{\cat{C}}$.

We use the same functor  $\delta_{\close}=\delta_{\open}$ and we denote it by $\delta$. This way, by using  condensation from Section \ref{sec: condensation}, we define the SC type operad $\Coend_{\ca{RL}_m}(\delta)$, and similarly for $\Coend_{\ca{RL}'_m}(\delta)$. 
\\

The idea for providing  the zig-zag \eqref{eq: zig-zag} follows that of \eqref{lem: weak equiv BA X} and relies on 
a decomposition of $\xi(\ca{RL}_m)_{c_1,...,c_k;c}(\delta)$ by ``cells'' $\xi(\ca{RL}_m)_{\alpha}(\delta)$ indexed by the $\alpha\in \cgraph_m(c_1,...,c_k;c)$. 
 Under some properties, the right-sided weak equivalence results from ``the contraction of the cells'', that is, from weak equivalences  $\xi(\ca{RL}_m)_{\alpha}(\delta)\to I$, where $I$ denotes the constant cosimplicial object $I^n=1_{\cat{C}}$; the left-sided  weak equivalence is induced by the natural map $\hocolim \xi(\ca{RL}_m)_{\alpha}(\delta)\to \colim \xi(\ca{RL}_m)_{\alpha}(\delta)$.

In fact, Batanin and Berger \cite[Theorem 3.8]{Batanin-Berger-Lattice} show that, in the closed case, \eqref{eq: zig-zag} holds provided that $\ca{L}_m$ is \emph{strongly $\delta$-reductive}. 
Here we extend their result to $\ca{RL}_m$. For consistency we recall the notion of \emph{strong $\delta$-reductivity} in our context.  Let us also recall  that a weak equivalence in $\cat{C}$ is called \emph{universal} if any pullback of it is again a weak equivalence. 

\begin{de}
 Let $\delta$ be a standard system of simplices in $\cat{C}$. 
The operad $\ca{RL}_m$ is called \emph{$\delta$-reductive} if  for any $n\geq 0$ and  $k\geq 0$ and any colours $c_i,c\in \{\close; \open\}$ satisfying \eqref{eq: col cond}, the map $\xi(\ca{RL}_m)_{c_1,...,c_k;c}(\delta)^n\to \xi(\ca{RL}_m)_{c_1,...,c_k;c}(\delta)^0$ is a universal weak equivalence. 

 The operad $\ca{RL}_m$ is called \emph{strongly $\delta$-reductive} if in addition the induced  maps \\
 $\Coend_{\ca{RL}_m}(\delta)(c_1,...,c_k;c)\to \xi(\ca{RL}_m)_{c_1,...,c_k;c}(\delta)^0$ are universal weak equivalences in $\cat{C}$.
\end{de}

\begin{theo}\label{th: strong red}
 Let $\delta$ be a standard system of simplices in $\cat{C}$. 
 If the operad $\ca{RL}_m$ (resp. $\ca{RL}'_m$) is strongly $\delta$-reductive, then the operad $\Coend_{\ca{RL}_m}(\delta)$ (resp. $\Coend_{\ca{RL}'_m}(\delta)$) is  weakly equivalent to $\Ba_{\delta} \cgraph_m$ (resp.  $\Ba_{\delta} \cgraph'_m$). 
\end{theo}

\begin{proof}
 The cells $\xi(\ca{RL}_m)_{\alpha}(\delta)$ are obtained via a map 
\begin{equation*}%
c_{tot}: \ca{RL}_m(\n_1,...,\n_k;\n)\to \ca{RK}_m(c_1,...,c_k;c),
\end{equation*}
 for  each $(k+1)$-tuple of colours $(c_1,...,c_k;c)$ in $\{\close,\open\}$ and $\n_i\in Col_{c_i}$, $\n\in Col_{c}$. 
Such a map is defined in \eqref{eq: ctot def} below and recovers the maps from  \cite[Proposition 3.4]{Batanin-Berger-Lattice} in the closed case (\ie $c=\close$). 
Note that  there is a slight inaccuracy in \cite{Batanin-Berger-Lattice} since, as we will show in Example \ref{ex: contre-exemple} and Lemma \ref{lem: filtr},  including in the closed case, the map $c_{tot}$ is not a morphism of coloured operads but instead  satisfies  
\begin{equation}\label{eq: ineg}
c_{tot}(x\circ_i y)\leq c_{tot}(x) \circ_i c_{tot}( y). 
\end{equation} 
However, such an inequality is sufficient to apply the method developed in \cite[Sections 3.5-3.6]{Batanin-Berger-Lattice}. 
Indeed, for $\alpha\in \cgraph_{m}(c_1,...,c_k;c)$ and $\n_i\in Col_{c_i}$, $\n\in Col_{c}$, let us define 
\begin{equation}\label{eq: cell RL}
 (\ca{RL}_m)_{\alpha}(\n_1,\dots,\n_k;\n):=\{x\in \ca{RL}_m(\n_1,\dots,\n_k;\n)~|~ c_{tot}(x)\leq \alpha\}. 
\end{equation}
It follows that 
 $\ca{RL}_m(\n_1,\dots,\n_k;\n)=\colim_{\cgraph_{m}(c_1,...,c_k;c)}(\ca{RL}_m)_{\alpha}(\n_1,\dots,\n_k;\n)$;  
the inequality \eqref{eq: ineg} ensures the compatibility of the decomposition with the operadic structures, so that it  implies that this is an equality of coloured operads. 
Moreover,  the operad $\widehat{\ca{RL}}_m$, defined as 
\begin{equation*}
 \widehat{\ca{RL}}_m(\n_1,\dots,\n_k;\n)=\hocolim_{\cgraph_{m}(c_1,...,c_k;c)}(\ca{RL}_m)_{\alpha}(\n_1,\dots,\n_k;\n),
\end{equation*} 
 is an operad (again, because of \eqref{eq: ineg} and because of the compatibility of $\hocolim$ with symmetric monoidal structure). 
It turns out that \eqref{eq: cell RL} form a multisimplicial sub-complex of $\ca{RL}_m(-,...,-;\n)$, so it makes sense to take its $\delta$-realization $\xi(\ca{RL}_m)_{\alpha}(\delta)$. In fact, as cosimplicial objects,   
$ \xi(\ca{RL}_m)_{c_1,...,c_k;c}(\delta)=\colim_{\cgraph_{m}(c_1,...,c_k;c)}\xi(\ca{RL}_m)_{\alpha}(\delta)$ 
and  $\xi(\widehat{\ca{RL}}_m)_{c_1,...,c_k;c}(\delta)$ identifies with $\hocolim_{\cgraph_{m}(c_1,...,c_k;c)}\xi(\ca{RL}_m)_{\alpha}(\delta)$.

From these considerations, it is  straightforward to verify that the proof of \cite[Theorem 3.8]{Batanin-Berger-Lattice} generalizes to our case: the $\delta$-reductivity implies that $\xi(\ca{RL}_m)_{\alpha}(\delta)\to I$ is a weak equivalence; and, 
via the strongly $\delta$-reductivity, the left-sided weak equivalence of the zig-zag \eqref{eq: zig-zag} results from the weak equivalence $\hocolim \xi(\ca{RL}_m)_{\alpha}(\delta)^0\to \colim \xi(\ca{RL}_m)_{\alpha}(\delta)^0$.
\end{proof}

\begin{prop}\label{prop: weak eq}
 For $\delta$ being $\delta_{Top}$ or $\delta_{\ring}$, the operads $\ca{RL}_m$ and $\ca{RL}'_m$ are strongly $\delta$-reductive. 
 Consequently, the operads $\Coend_{\ca{RL}_m}(\delta)$ and $\Coend_{\ca{RL}'_m}(\delta)$ are weakly equivalent to the topological (resp. chain) Swiss Cheese operad $\SC_m$ (resp. $C_*\SC_m$) for $\delta$ being $\delta_{Top}$ (resp. $\delta_{\ring}$).
\end{prop}
\begin{proof}
  Again, this is a straightforward generalization of \cite[Examples 3.10]{Batanin-Berger-Lattice}, where it is shown that 
  $\ca{L}_m$ is strongly $\delta$-reductive for $\delta$ being $\delta_{Top}$ or $\delta_{\ring}$. The very same method applies in our context, by considering  $\xi(\ca{RL}_m)_{c_1,...,c_k;c}(\delta)$ instead of $\xi(\ca{L}_m)_{k}(\delta)$. 
\end{proof}

For a $(k+1)$-tuple $(c_1,...,c_k;c)$ of colours in $\{\close,\open\}$ and $\n_i\in Col_{c_i}$, $\n\in Col_{c}$, let 
\begin{equation*}%
 c_{tot}: \ca{RL}_m(\n_1,...,\n_k;\n)\to \ca{RK}_m(c_1,...,c_k;c),
\end{equation*}
be as follows.
Recall the notation $\wid{c}_i$ from \eqref{eq: notation ctilde}. 
 The element $c_{tot}(x)=(\mu,\sigma)\in \cgraph_m$ is defined, for  $1\leq i<j\leq  k$, by: 
\begin{equation}\label{eq: ctot def}
 \mu_{i,j}=c_{i,j}(x)  ~\text{ and }~
\sigma_{i,j}=
 \begin{cases}
  \widetilde{c}_i \to \widetilde{c}_j &~\text{ if }~ \wid{c}_i^-<\wid{c}_j^- \\
 \widetilde{c}_i \leftarrow \widetilde{c}_j &~\text{ if }~ \wid{c}_i^->\wid{c}_j^-.
 \end{cases}
\end{equation}  
Similarly, let $c'_{tot}: \ca{RL}'_m\to \cgraph'_m$ be the map defined by  the formula \eqref{eq: ctot def} for  $x\in \ca{RL}'_m$. 

\begin{lem}\label{lem: filtr}
For all $ x,y\in \ca{RL}_m$  and $i$ such that  $x\circ_i y$ makes sense, one has the following inequality 
$c_{tot}(x\circ_i y)\leq c_{tot}(x) \circ_i c_{tot}( y)$. %
For all $ x,y\in \ca{RL}'_m$  and $i$ such that  $x\circ_i y$ makes sense, one has the  following inequality   
$c'_{tot}(x\circ_i y)\leq c'_{tot}(x) \circ_i c'_{tot}( y)$. %
\end{lem}
\begin{proof}
We show the first assertion, the second one is similar. In the following arguments the type of colours does not matter, so  we abusively forget about the underline. 
Let $x\in \ca{RL}_m(\n_1,\dots,\n_p;\n)$ and $y\in \ca{RL}_m(f_1,\dots,f_q;\n_i)$ for some $1\leq i\leq p$. 

For an integer $a\neq i$, one defines $a':=a$  if $a<i$ and $a':=a+q-1$ if $a>i$. \\ 
   Recall that, by definition,  the complete graph  $(\mu,\sigma)=c_{tot}(x) \circ_i c_{tot}(y)$ is given by:  
   \begin{align*}
   (\mu_{k+i-1,l+i-1},\sigma_{k+i-1,l+i-1}) &=(c_{k,l}(y),\sigma_{k,l}(c_{tot}(y))) \text{ for } 1\leq k<l\leq q;\\
   (\mu_{r',s'},\sigma_{r',s'})&=(c_{r,s}(x),\sigma_{r,s}(c_{tot}(x))) \text{ for } 1\leq r<s\leq p \text{ such that } i\notin \{r,s\}; \\
   (\mu_{r,k+i-1},\sigma_{r,k+i-1})& =(c_{r,i}(x),\sigma_{r,i}(c_{tot}(x))) 
   \text{ for } 1\leq r< i \text{ and } 1\leq k\leq q;\\
   (\mu_{k+i-1,s'},\sigma_{k+i-1,s'})& =(c_{i,s}(x),\sigma_{i,s}(c_{tot}(x))) 
   \text{ for } i< s\leq p \text{ and } 1\leq k\leq q.
   \end{align*}
On the other hand, let us recall that the integer-string $y$ is subdivided in $n_i+1$ substrings that are delimited by the $n_i$ vertical bars ($n_i:=ev(\n_i)$). Recall that the integer-string $x\circ_i y$ is obtained by substituting the $b$-th occurrence of $i$ by the $b$-th substring (indexed by $b-1$) of $y$ together with a re-indexation of the values. 
 It follows that $c_{k+i-1,l+i-1}(x\circ_i y)= c_{k,l}(y)$ and that the order in which the pair $((k+i-1)^-,(l+i-1)^-)$ appears in $x\circ_i y$ is the same as the order in which the pair $(k^-,l^-)$ appears in $y$. 
Similarly, one has $c_{r',s'}(x\circ_i y)= c_{r,s}(x)$, and the order in which $(r'^-,s'^-)$ appears in $x\circ_i y$ is the same than the order in which the pair $(r^-,s^-)$ appears  in $x$.

Moreover, it is straightforward to see that $c_{r,k+i-1}(x\circ_i y)\leq c_{r,i}(x)$ and that the equality holds if and only if $k$ is present in, at least, one substring per switch between $r$ and $i$. Such an equality  is  illustrated below: there is at least one $k$ in each of the $c_{r,i}(x)$ blocks of substrings
\begin{equation*}\label{eq: compo substring}
\scalebox{.8}{\begin{tikzpicture}[>=stealth, arrow/.style={->,shorten >=3pt}, point/.style={coordinate}, pointille/.style={draw=red, top color=white, bottom color=red},scale=1.6,baseline=-4ex]
 \draw [->] (-1.8,-.35) to (-1.65,-.1);
 \draw [->] (-1.25,-.3) to (-1.3,-.1);
 \draw [->] (-.35,-.5) to (-0.55,-.1);
 \draw [->] (.25,-.5) to (0.05,-.1);
 \draw [->] (1.25,-.5) to (1.3,-.1);
 \draw [->] (2.2,-.35) to (2.15,-.1);
 \draw (0,0) node [] {{$x=(~i^-~i~~\cdots~~ i~~r^-~i~~r~~\cdots ~~r~ ~i~~\cdots ~~i~)$}};
 \draw (0,-.85) node [] {{$\phantom{,}y=(\underbrace{\overbrace{\phantom{k}}^0|\overbrace{{k^-}}^1|\cdots |\overbrace{\phantom{k}}^{}}_{1}|\underbrace{\overbrace{\vphantom{p}k}^{}}_{2}|\cdots |\underbrace{\overbrace{k}^{}|\cdots |\overbrace{k}^{n_i}}_{c_{r,i}(x)})$;}};
  \end{tikzpicture}}
  \end{equation*}
  here, only the integers $r$, $i$ and $k$ are written; the arrows indicate the substitutions. 
   
   In the equality case, the order in which  $(r^-,(k+i-1)^-)$ appears in $x\circ_i y$ is the same as the order in which  $(r^-,i^-)$ appears in $x$. 
Similar  considerations hold for the inequality $c_{k+i-1,s'}(x\circ_i y)\leq  c_{i,s}(x)$. 
   
It follows from the above paragraphs that  $c_{tot}(x\circ_i y)\leq c_{tot}(x) \circ_i c_{tot}( y)$. 
\end{proof}

\begin{exple}\label{ex: contre-exemple}
Here is an example of $x$ and $y$ such that $c_{tot}(x\circ_i y)< c_{tot}(x) \circ_i c_{tot}( y)$. 
For  $x=(121|1)$ and $y=(1|12|2)$, one has $(121|1) \circ_1 (1|12|2) = (1312|2) $ and  
\begin{equation*}
c_{tot}(x\circ_1 y)=
\begin{tikzpicture}[>=stealth, arrow/.style={->,shorten >=3pt}, point/.style={coordinate}, pointille/.style={draw=red, top color=white, bottom color=red},scale=0.6,baseline=0.9ex,decoration={markings, mark=at position 0.65 with {\arrow[scale=1.2]{>}}}]%
 \coordinate (A) at (0,0);
 \coordinate (B) at (2,0);
 \coordinate (C) at (1,1.25);
  \draw [postaction={decorate}] (A) -- ++(B) node [below,midway] {\footnotesize{$1$}};
  \draw [postaction={decorate}] (A) -- (C) node [left,midway] {\footnotesize{$2$}}; 
  \draw [decoration={markings, mark=at position 0.65 with {\arrow[scale=1.5,blue]{>}}},postaction={decorate}] (C) -- (B) node [right,midway,blue] {\footnotesize{$1$}};  
  \draw [black,fill=black] (A) circle (0.1)  ;
 \draw [black,fill=black] (B) circle (0.1)  ;
 \draw [black,fill=black] (C) circle (0.1)  ;
 \draw (A) node [left] {$1$};
 \draw (B) node [right] {$2$};
 \draw (C) node [above] {$3$};
\end{tikzpicture}%
\leq 
\begin{tikzpicture}[>=stealth, arrow/.style={->,shorten >=3pt}, point/.style={coordinate}, pointille/.style={draw=red, top color=white, bottom color=red},scale=0.6,baseline=0.9ex,decoration={markings, mark=at position 0.65 with {\arrow[scale=1.2]{>}}}]%
 \coordinate (A) at (0,0);
 \coordinate (B) at (2,0);
 \coordinate (C) at (1,1.25);
  \draw [postaction={decorate}] (A) -- ++(B) node [below,midway] {\footnotesize{$1$}};
  \draw [postaction={decorate}] (A) -- (C) node [left,midway] {\footnotesize{$2$}}; 
  \draw [postaction={decorate}] (B) -- (C) node [right,midway] {\footnotesize{$2$}};  
  \draw [black,fill=black] (A) circle (0.1)  ;
 \draw [black,fill=black] (B) circle (0.1)  ;
 \draw [black,fill=black] (C) circle (0.1)  ;
 \draw (A) node [left] {$1$};
 \draw (B) node [right] {$2$};
 \draw (C) node [above] {$3$};
\end{tikzpicture}%
=
\begin{tikzpicture}[>=stealth, arrow/.style={->,shorten >=3pt}, point/.style={coordinate}, pointille/.style={draw=red, top color=white, bottom color=red},scale=0.6,baseline=0ex,decoration={markings, mark=at position 0.65 with {\arrow[scale=1.2]{>}}}]%
 \coordinate (A) at (0,0);
 \coordinate (B) at (2,0);
  \draw [postaction={decorate}] (A) -- ++(B) node [below,midway] {\footnotesize{$2$}};
  \draw [black,fill=black] (A) circle (0.1)  ;
 \draw [black,fill=black] (B) circle (0.1)  ;
 \draw (A) node [above] {$1$};
 \draw (B) node [above] {$2$};
\end{tikzpicture}%
\circ_1
\begin{tikzpicture}[>=stealth, arrow/.style={->,shorten >=3pt}, point/.style={coordinate}, pointille/.style={draw=red, top color=white, bottom color=red},scale=0.6,baseline=0ex,decoration={markings, mark=at position 0.65 with {\arrow[scale=1.2]{>}}}]%
 \coordinate (A) at (0,0);
 \coordinate (B) at (2,0);
  \draw [postaction={decorate}] (A) -- ++(B) node [below,midway] {\footnotesize{$1$}};
  \draw [black,fill=black] (A) circle (0.1)  ;
 \draw [black,fill=black] (B) circle (0.1)  ;
 \draw (A) node [above] {$1$};
 \draw (B) node [above] {$2$};
\end{tikzpicture}%
=c_{tot}(x)\circ_1 c_{tot}(y).
 \end{equation*} 
 \end{exple}

\subsection{Action on cochains}

One has the obvious relative version of \cite[Proposition 2.20]{Batanin-Berger-Lattice}: 
\begin{prop}\label{prop: action cosimpl abelian}
 Let $X$ and $Y$ be two simplicial sets equipped with a simplicial map $f:Y\to X$. 
 The pair $(\ring[X],\ring[Y])$ { (resp. $(\Hom(\ring[X],\ring),\Hom(\ring[Y],\ring))$) } is a coalgebra (resp. an algebra) over $\ca{RL}$. In particular, $\Coend_{\ca{RL}}(\delta_{\ring})$ acts on $(C^*(X;\ring),C^*(Y;\ring))$. 
\end{prop}
\begin{proof}
The $\ca{RL}$-coaction in the closed case is as in \cite[Proposition 2.20]{Batanin-Berger-Lattice}. 
Otherwise 
\[
\ring[\ca{RL}(\n_1,...,\n_k;\un{n})] \ot \ring[Y_n]\to \text{tensor products of $\ring[X_{n_i}]$ and $\ring[Y_{n_j}]$}, 
\]
is given  by $x\ot y\mapsto \text{tensor products of } x^*_i(y)$ and $f(x^*_j(y))$, where the $x_i$ and $x_j$ are the components of the dual interpretation of $x$ from Section \ref{sec: dual interp} \ie $x$ corresponds to the $k$-tuple $(x_1,...,x_k)\in \triangle([\n_1],[\n])\times ...\times \triangle([\n_k],[\n])$. 
Since $f$ is a map of simplicial sets the result directly follows from \textit{loc. cit.}
\end{proof}

\subsection{Action on iterated relative loop spaces}\label{sec: cosimpl rel loop space}

In this section we show that, for any two topological spaces $Y\subset X$ pointed at the same point $*$, the operads $\Coend_{\ca{RL}_m}(\delta_{Top})$ and $\Coend_{\ca{RL}'_m}(\delta_{Top})$ acts on the pair $(\Omega^mX,\Omega^m(X,Y))$. 
Here,  $\Omega^mX$ denotes the $m$-fold loop space of $X$ and $\Omega^m(X,Y)$ denotes the $m$-fold relative loop spaces of $(X,Y)$. 
\\

For $m\geq 1$ let $\Delta^m$ be the simplicial $m$-simplex and  $\partial\Delta^m$ its boundary. 
For $0\leq p\leq m$, let $\Lambda_p^m$  be the $m$-horn at $p$. 
Explicitly, $(\Lambda_p^m)_k \subset \Delta^m_k$ is given by elements of the form $yz:[k]\to [m]$ for $z:[k]\to [n]$ and $y:[n]\to[m]$ with $n<m$ and, if $n=m-1$, then  $p\in Im(y)$. 
For more details on these simplicial sets we refer to \cite[p.6-7]{Goerss-Jardine}. 
It immediately follows from the definitions that: 
\begin{lem}\label{lem: not pb}
 Let $x\in\Delta^m$. If $x\notin  {\Lambda_p^m}$, then there is a $k\in\{m-1,m\}$ such that $x=\iota \rho:[n]\twoheadrightarrow [k] \hookrightarrow [m]$. %
 Moreover, if  $k=m-1$, then $\iota=\partial_p:[m-1]\hookrightarrow [m]$ is the $p$-face map ($p\notin Im(\partial_p)$); 
 if $k=m$, then $\iota=id:[m]\to [m]$ is the identity map. 
  In particular, if $x\notin \partial \Delta^m$, then $k=m$, hence $x$ is surjective. 
\end{lem}

Let $\sph^m= \Delta^m/\partial \Delta^m$ be the simplicial $m$-sphere. 
Let $K_{\close}[m]:=(\sph^m,*)$ and $K^p_{\open}[m]:=(\Delta^m/{\Lambda_p^m}, {\partial\Delta^m}/{\Lambda_p^m})$. 
Both $K_{\close}[m]$ and $K^p_{\open}[m]$ are simplicial objects in the category $\cat{CFin}$ whose objects are  pairs  $(A,B)$ of finite sets pointed at the same point with $*\subset B\subset A$ and morphisms $f:(A,B)\to (A',B')$ are morphisms of pointed sets $f:A\to A'$  such that $f(B)\subset B'$. The \emph{wedge sum} is given by $(A,B)\vee (A',B')= (A\vee A',B\vee B')$.

The inclusion $\Lambda_p^m\subset \partial \Delta^m$ induces the projection 
$\pi:K^p_{\open}[m] \to K_{\close}[m]$ that is a map of simplicial objects in $\cat{CFin}$.

\begin{prop}\label{prop: RL-coalg}
 Let $m\geq 1$ and $0\leq p\leq m$. 
 If $p$ is odd, then  the pair $(K_{\close}[m],K^p_{\open}[m])$ is a coalgebra over $\ca{RL}_m$ in $\cat{CFin}$.
 If $p$ is even, then  the pair $(K_{\close}[m],K^p_{\open}[m])$ is a coalgebra over $\ca{RL}'_m$ in $\cat{\cat{CFin}}$.  
\end{prop}
\begin{proof}
Let us fix $m\geq 1$. In order to alleviate notations one will denote $K_\close[m]$ (resp. $K^p_\open[m]$)  by $K^\close$ (resp. $K^\open$). 
 An element $x\in \ca{RL}(\n_1,...,\n_k;\n)$ induces maps 
 \begin{equation*}
  x^*:K^{\close}_n \to K^{\close}_{\n_1}\times \cdots \times K^{\close}_{\n_k} \text{ if } \n=n  \text{ and } 
  \hat{x}^*: K^{\open}_{\un{n}} \to K^{c_1}_{\n_1}\times \cdots \times K^{c_k}_{\n_k} \text{ if } \n=\un{n}.
    \end{equation*}
 For the closed part ($\n=n$), this is \cite[Proposition 2.16]{Batanin-Berger-Lattice}. 
 For $\n=\un{n}$,  let $(x_1,...,x_k)\in \triangle ([\n_1],[\un{n}])\times \cdots \times \triangle ([\n_k],[\un{n}])$ be the $k$-tuple corresponding to $x$ in the dual interpretation of $\ca{RL}$, see Section \ref{sec: dual interp}. 
 For $1\leq i\leq k$, let $n_i=ev(\n_i)$. 
 For $y\in \Delta^m/{\Lambda_p^m}$, the $i$-th component of $\hat{x}^*(y)$ is given by: 
 \begin{equation*}
  \hat{x}_i^*(y)=
  \begin{cases}
  x_i^*(y) &\text{ if } x_i:[\un{n_i}]\to [\un{n}] \\
  \pi(x_i^*(y)) &\text{ if } x_i:[{n_i}]\to [\un{n}]. 
 \end{cases}
\end{equation*}
 We prove that if $x\in \ca{RL}_m$ for $p$ odd (resp. $x\in \ca{RL}'_m$ for $p$ even)  then $\hat{x}^*(y)$ belongs to $K^{c_1}_{\n_1}\vee \cdots \vee K^{c_k}_{\n_k}$. 
 From now on, let us suppose that there are $i\neq j$ such that $x_i^*(y)$ and $x_j^*(y)$ are not at the base point.

Suppose $i$ corresponds to an open colour (\ie $\n_i=\un{n_i}$). 
   In this case, $x_i^*(y):[n_i]\to [n]\to [m]$ is not at the base point  \ie $x_i^*(y)\notin \Lambda_p^m$. From Lemma \ref{lem: not pb}, this means that  
   $Im(x_i)\cap y^{-1}(r)\neq \emptyset$ for each $r\in [m]$ such that $r\neq p$.    
\begin{enumerate}
 \item Suppose $j$ corresponds to an open colour.  
 This means that  $x_j^*(y_s)$ is not at the base point and then $Im(x_j)\cap y^{-1}(r)\neq \emptyset$ for each $r\in [m]$ such that $r\neq p$.  Therefore,   $\un{i}$ and $\un{j}$ appear in $m$ common fibres of $y$. Consequently,  $c_{i,j}(x)\geq m>m-1$.  
 \item Suppose $j$ corresponds to a closed colour. This means that $\pi(x_j^*(y))\notin \partial \Delta^m$, that is,  $x_j^*(y):[n_j]\to [n]\to [m]$ does not belong to $\partial \Delta^m$ and thus is surjective (Lemma \ref{lem: not pb}). It follows that   $Im(x_j)$ intersects all the fibres of $y$ \ie  $Im(x_j)\cap y^{-1}(r)\neq \emptyset$ for each ${r\in [m]}$. In particular, $y:[n]\to [m]$ is surjective and then there are two cases for $x_i$: 
 \begin{enumerate}
  \item   $Im(x_i)\cap y^{-1}(p)\neq \emptyset$ (\ie $Im(x_i)\cap y^{-1}(r)\neq \emptyset$ for each $r\in [m]$) and then $c_{i,j}(x)\geq m+1>m$ (because $\un{i}$ and ${j}$ appear in $m+1$ common fibres of $y$); or,   
 \item  $Im(x_i)\cap y^{-1}(p)= \emptyset$. 
 When $p$ is odd, this implies that $c_{{i},j}(x)>m$ if $j^-<\un{i}^-$ and $c_{{i},j}(x)>m-1$ if $j^->\un{i}^-$. 
 When $p$ is even, this implies that $c_{{i},j}(x)>m$ if $j^->\un{i}^-$ and $c_{{i},j}(x)>m-1$ if $j^-<\un{i}^-$. 
 \end{enumerate}
\end{enumerate}
 Let us illustrate what happens in 2(b) when $p$ is odd and $j^-<\un{i}^-$. In the following integer-string representation of $x$, 
 \begin{equation*}
  x=(\underbrace{\overbrace{j\un{i}}^{0}|\cdots | \overbrace{\un{i}j}^{p-2}|\overbrace{j\un{i}}^{p-1}}_{c_{{i},j}(x)_{|0,p-1}\geq p}|\overbrace{j}^{p}|\underbrace{\overbrace{j\un{i}}^{p+1}|\cdots | \overbrace{\un{i}j}^{m-1}|\overbrace{j\un{i}}^m}_{c_{{i},j}(x)_{|p+1,m}\geq m-p}) ~~ \text{with } j^-<\un{i}^-,
 \end{equation*}  
  we only represent (relevant occurrences of) the integers $\un{i}$ and $j$; the $m+1$ substrings represent the fibres $y^{-1}(r)$, $r\in [m]$ so that, if $n>m$, then these substrings would be subdivided (to end up with $n$ vertical bars). 
 Condition 2 says that the $j$ are present in each of the $m+1$ substrings.  
 The case 2(b) says that the $\un{i}$ are present in all the $m+1$ substrings except the substring $p$. 
 The number of switches $c_{{i},j}(x)_{|0,p-1}$ (resp. $c_{{i},j}(x)_{|p+1,m}$) between $\un{i}$ and $j$ in the first $p$ substrings (resp. in the  last $m-p$ substrings) is then $\geq p$ (resp. $\geq m-p$).  
 If $c_{{i},j}(x)_{|0,p-1}=p$, the conditions $j^-<\un{i}^-$ and $p$ odd imply that, in the substring $p-1$, the occurrences of $j$ all appear before that of $\un{i}$. Therefore, because of the presence of $j$ in the substring $p$, there is an additional switch \ie  $c_{{i},j}(x)\geq c_{{i},j}(x)_{|0,p-1}+1+c_{{i},j}(x)_{|p+1,m}\geq p+1+m-p>m$. 
\end{proof}

\begin{cor}
 Let $*\subset Y\subset X$ be topological spaces. 
 For $m\geq 1$, the pair $(\Omega^mX,\Omega^m(X,Y))$ is an algebra over $\Coend_{\ca{RL}_m}(\delta_{Top})$ and over $\Coend_{\ca{RL}'_m}(\delta_{Top})$. 
 \end{cor}
\begin{proof}
 By adjunction, the $m$-fold relative loop space $\Omega^m(X,Y)=Hom_{Top_*}(|K^p_\open[m]|_{\delta_{Top}},(X,Y))$ is homeomorphic to the $\delta_{Top}$-totalization of $(X,Y,*)^{(K^p_\open[m],*)}$. Similarly, the $m$-fold loop space $\Omega^mX=Hom_{Top_*}(|K_{\close}[m]|_{\delta_{Top}},X)$ is homeomorphic to the $\delta_{Top}$-totalization of $(X,*)^{(K_{\close}[m],*)}$. 
 Proposition \ref{prop: RL-coalg} implies that $((X,*)^{(K_{\close}[m],*)},(X,Y,*)^{(K^p_\open[m],*)})$ is an $\ca{RL}_m$-algebra for $p$ odd, and an $\ca{RL}'_m$-algebra for $p$ even. The result follows from condensation (Section \ref{sec: condensation}). 
\end{proof}

\section{The relative surjection operad}\label{sec: rel surj operad}

\newcommand{\e}{id}

We define two SC type operads $\ca{RS}_m$ and $\ca{RS}'_m$ that are sub-operads of $\Coend_{\ca{RL}_m}(\delta_{\ring})$ and $\Coend_{\ca{RL}'_m}(\delta_{\ring})$, respectively. 
We show that these inclusions are weak equivalences. 
\\

Since we are using $\delta_{\ring}$-realization, the Dold-Kan correspondence provides a convenient way to present the cosimplicial chain complex $\xi_{c_1,...,c_k;c}(\ca{RL}_m)( \delta_{\ring})$  as well as the operad $\Coend_{\ca{RL}_m}(\delta_{\ring})$.  We closely follows \cite[Section 3]{BBM} in which this point of view is adopted. In particular, we refer to \cite[Section 3.3]{BBM} for normalized totalizations of (co)simplicial abelian  groups; we adopt same notations except that, since we have defined  $\ca{RL}_m$ in the category of sets, we apply $\mathbb{Z}[-]$ to make it a multisimplicial/cosimplicial abelian  group. 
 This way, one identifies 
   $\Coend_{\ca{RL}_m}(\delta_{\ring})$ 
with $\overline{Nor}\un{Nor}(\mathbb{Z}[\ca{RL}_m])$, where $\un{Nor}(-)$ stands for the normalized realization of multisimplicial abelian groups and $\overline{Nor}(-)$ stands for the normalized totalization of cosimplicial dg-abelian groups.

Let $(c_1,...,c_k;c)$ be a $(k+1)$-tuple of colours in $ \{\close,\open\}$. 
As complexes, we set 
\[\ca{RS}_m(c_1,...,c_k;c)^*:=\un{Nor}(\mathbb{Z}[\ca{RL}_m(-,...,-;\n)]), 
\]
where $\n$ is of type $c$ such that $ev(\n)=0$  and the colours in the $i$-th input are of type $c_i$ for $1\leq i\leq k$.  
More explicitly, $\ca{RS}_m(c_1,...,c_k;c)^*$ is obtained from 
$\bigoplus_{*=-(n_1+...+n_k)}\mathbb{Z}[\ca{RL}_m(\n_1,...,\n_k;\n)]$, 
where $\n$ is of type $c$ such that $ev(\n)=0$ and $n_i:=ev(\n_i)$ and  $\n_i$ is of type $c_i$ for $1\leq i\leq k$,   
by modding out the images of the simplicial degeneracies.

 We closely follow \cite[equation (4) and (5)]{BBM} in which the following \emph{whiskering} map and partial compositions are defined for the closed case.

On an integer-string $x\in \ca{RL}_m(\n_1,...,\n_k;\n)$ with $ev(\n)=0$, the $n$-whiskering $w_n(x)$ is a signed sum of integer-strings  obtained from $x$ by: 
copying integers and adding a vertical bar between each copy (e.g. $i\mapsto i|i|i$) 
with the requirement that the total number of vertical bars is $n$. Note that there is no bar between two adjacent integers with different values (e.g. $i|j$ is not allowed). 
\begin{exple}
 $(\un{1}|\un{1}2|2|2)^\open$ is a term of $w_3((\un{1}2)^\open)$. 
\end{exple}
The whiskering 
\begin{align*}
w:\ca{RS}_m \to \overline{Nor}\un{Nor}(\mathbb{Z}[\ca{RL}_m])
\end{align*}
is defined by $x\mapsto \Pi w_n(x)$ for $x\in \ca{RL}_m(\n_1,...,\n_k;\n)$ with $ev(\n)=0$, and is extended by linearity.

Let us describe the partial compositions of ${\ca{RS}}_m$. 
For two integer-string $f\in{\ca{RS}}_m(c_1,...,c_k;c)$ and $g\in{\ca{RS}}_m(d_1,...,d_j;c_i)$
we set 
\begin{equation}\label{eq: compo RS}
 f\circ_{i}^{\ca{RS}_m} g = f\circ_{i} w_{val_i(f)} g,
\end{equation}
where $val_i(f)$ is one less than the number of occurrences of $ \wid{c}_i$ in the integer-string $f$. %
 We extend the compositions $\circ_i^{\ca{RS}_m}$ by linearity.

We define an operad $\ca{RS}'_m$ in the same way as $\ca{RS}_m$, by replacing  $\ca{RL}_m$ by $\ca{RL}'_m$ in the above paragraphs. 

\begin{prop}
The partial compositions $\circ_i^{\ca{RS}_m}$ (resp. $\circ_i^{\ca{RS}'_m}$) endow $\ca{RS}_m$ (resp. $\ca{RS}'_m$) with an operad structure. 
 The inclusion $w:\ca{RS}_m \hookrightarrow \Coend_{\ca{RL}_m}(\delta_{\ring})$ (resp. $w':\ca{RS}'_m \hookrightarrow \Coend_{\ca{RL}'_m}(\delta_{\ring})$) is a weak equivalence of operads. 
\end{prop}
\begin{proof}
 Except for signs, the fact that $w$ is compatible with the operadic structures is straightforward from the definition. 
  To get signs for the whiskering, we can proceed by a laborious induction, by requiring the whiskering to be a morphism of operads. 
  One can also remark that, concerning signs, the type of the colours (closed or open) does not matter. 
  This is because closed and open colours obey the same (co)simplicial and composition rules (with different constraints). 
   Thus, signs can be chosen using the same method than in the non-relative case (disregarding -for signs only- the type of colours): one could proceed as proposed in \cite[Proposition 3.2]{BBM}, that is, embed $\ca{RS}_m$ into the operad of coendomorphisms of the chains of a high dimensional simplex and then, choose signs in such a way that they match with those of \cite[Section 2.2]{BergerFresse} in the closed and in the open cases and, in fact, in the open/closed case, disregarding type of colours when choosing signs.  This can be compared with Proposition \ref{prop: action cosimpl abelian}, by asking the map $f$ to be $id:\Delta^n\to \Delta^n$.

Let us denote by 
\begin{equation*}
\pi:  \overline{Nor}\un{Nor}(\mathbb{Z}[\ca{RL}_m])\to  \un{Nor}(\mathbb{Z}[\ca{RL}_m])^0= \ca{RS}_m  
\end{equation*}
 the map induced by the projection $p:\un{Nor}(\mathbb{Z}[\ca{RL}_m])^\cdot\to  \un{Nor}(\mathbb{Z}[\ca{RL}_m])^0$. 
Such a map is a weak equivalence by Proposition \ref{prop: weak eq}.   Moreover it satisfies  $\pi \circ w=id$. Thus, $w$ is a weak equivalence. 
\end{proof}

\section{The operads $\ca{RL}_2$ and $\ca{RL}'_2$}\label{sec: RL2}
\subsection{The operad $\ca{RL}_2$ in term of trees}

\subsubsection{The sets  of $\ca{RT}$}\label{sec: pre def RT}

In what follows one will consider \emph{planar rooted trees}; one refers to \cite[II, Section 1.5]{MSS} for the terminology. 
 Our trees have only one external edge, called the \emph{root}; all the other edges have $2$ adjacent vertices. The external vertices (vertices that are adjacent to only one edge) are called the \emph{leaves}.    
Given a vertex $v$ of a rooted tree $T$, the minimal sub-tree of $T$ containing both the root and $v$ has only one edge originating from $v$; such an edge is called the \emph{output of $v$}. The edges originating from a vertex that are not the output are called the \emph{inputs}. 
In our planar trees, the set of the edges originating from a vertex $v$ is cyclically ordered in the clockwise direction. 
This induces a linear order on the set of the inputs of $v$. This also canonically endows the set of the leaves with a linear order.  

One considers planar rooted trees with $4$ types of vertices: 
white round-shaped  
$\begin{tikzpicture}[>=stealth, arrow/.style={->,shorten >=3pt}, point/.style={coordinate}, pointille/.style={draw=red, top color=white, bottom color=red},scale=1]
 \draw [black,fill=white] (0,0) circle (0.11)  ;
\end{tikzpicture}$,
called \emph{closed vertices};  
white square-shaped 
$\begin{tikzpicture}[>=stealth, arrow/.style={->,shorten >=3pt}, point/.style={coordinate}, pointille/.style={draw=red, top color=white, bottom color=red},scale=1]
 \sq{(0,0)}; 
\end{tikzpicture}$,
called \emph{open vertices}; 
black round-shaped 
$\begin{tikzpicture}[>=stealth, arrow/.style={->,shorten >=3pt}, point/.style={coordinate}, pointille/.style={draw=red, top color=white, bottom color=red},scale=1]
 \draw [black,fill=black] (0,0) circle (0.11)  ;
\end{tikzpicture}$, called \emph{neutral vertices};  
 and, black arrow-shaped 
 $~\begin{tikzpicture}[>=stealth, arrow/.style={->,shorten >=0pt}, point/.style={coordinate}, pointille/.style={draw=red, top color=white, bottom color=red},scale=1 ,decoration={markings, mark=at position 0.65 with {\arrow[scale=1.2]{>}}},baseline=-1.5ex]%
   \draw [decoration={markings, mark=at position 0 with {\arrow[scale=2,black]{>}}},postaction={decorate}] (0,0) -- (0,0.01); 
\end{tikzpicture}~$, called the \emph{arrows}. 
A \emph{white} vertex is either a closed vertex or an open vertex. %
\begin{de}\label{de: trees}
Let $T$ be a planar rooted tree. 
 Let $\nu$ be a vertex of $T$. We denote by $T_{\nu}$ the maximal sub-tree of $T$ such that the output of $\nu$ is the root of $T_{\nu}$. 
\end{de}

Let $(\n_1,...,\n_{k},\n)$ be a $k+1$-tuple of colours in $Col$, the set 
 $\ca{RT}(\n_1,...,\n_{k};\n)$ is the empty set if $\n\in Col_{\close}$ and there exists a $i$ such that $\n_i\in Col_{\open}$; 
 otherwise, $\ca{RT}(\n_1,...,\n_{k};\n)$ 
 is the set of equivalence classes of planar rooted trees $T$, satisfying the following: 
\begin{itemize}
 \item the set of the arrows  is a subset of the leaves of $T$ and is of cardinal $ev(\n)$.  
\item the set of the closed (resp. open) vertices  is labelled by the set $\{i\in\{1,...,k\} ~|~ e_i=n_i \}$  (resp. $\{i\in\{1,...,k\} ~|~ e_i=\un{n}_i\}$),  
in such a way that:
 \begin{itemize}
 \item[(F1)]  the vertex labelled by $i\in \{1,...,k\}$ has $ev(\n_i)$ inputs, 
 \item[(F2)] there is no white vertex above an open vertex \ie  if
 $\nu$ is an open vertex, 
 then 
 in the tree $T_{\nu}$  the vertex $\nu$ is the unique white vertex,  \label{item: F1}
 \end{itemize}
 \item if $e$ is an open colour (\ie $e=\un{n}$ for some $n$), then the root of $T$ is decorated by an $\open$. 
\end{itemize}

The equivalence class is the same as in \cite[3.2.1]{DTT-SC}. Explicitly, it is the finest one in which two planar rooted trees are equivalent if one of them can be obtained from the other by either:\\
- the contraction of an edge with neutral  adjacent vertices; or,\\ 
- removing an neutral  vertex with only one edge originating from it and joining the two edges adjacent to this vertex into one edge.

\begin{rmq}
 In the closed case, the trees of $\ca{RT}(n_1,...,n_{k};n)$ have no open vertices. 
\end{rmq}

\begin{figure}[H]
\centering  
\subfigure[An element of $\ca{RT}(3,2,2;6)$.\label{fig: tree closed}]
{
\begin{tikzpicture}[>=stealth, arrow/.style={->,shorten >=3pt}, point/.style={coordinate}, pointille/.style={draw=red, top color=white, bottom color=red},scale=0.65]
 \coordinate (SEPARATOR) at (-3,0);
 \coordinate (SEPARATOR2) at (3,0);
 \draw (SEPARATOR) node [left] {\phantom{a}};
 \draw (SEPARATOR2) node [left] {\phantom{a}};
 \coordinate (A) at (0,0);
 \coordinate (B) at (-1,1);
 \coordinate (C) at (1,1.5);
 \coordinate (D) at (.75,2);
 \draw [-] ($(A)+(0,-.65)$) -- (A);
 \draw [-] (A) -- (B);
 \draw [-] (A) -- (C);
 \draw [-] (C) -- (D);
  \draw [->] (B) -- ($(B)+(-.5,.5)$);
  \draw [->] (B) -- ($(B)+(.25,.75)$);
  \draw [->] (B) -- ($(B)+(.75,.5)$);
  \draw [->] (C) -- ($(C)+(+1.5,.75)$);
  \draw [->] (D) -- ($(D)+(-.5,.5)$);
  \draw [->] (D) -- ($(D)+(.5,.5)$);
 \draw ($(A)+(0,-.2)$) node [below,left] {${3}$};
 \draw ($(B)+(0,-.1)$) node [left] {$1$};
 \draw ($(C)+(.3,0)$) node [below] {$\scriptsize{2}$};
  \draw [black,fill=white] (A) circle (0.12)  ;
 \draw [black,fill=white] (B) circle (0.12)  ;
 \draw [black,fill=white] (C) circle (0.12)  ;
 \draw [black,fill=black] (D) circle (0.11)  ;
\end{tikzpicture}
}
\subfigure[An element of $\ca{RT}(2,\un{0},3,\un{3};\un{5})$.\label{fig: tree open-close2}]  
{  
\begin{tikzpicture}[>=stealth, arrow/.style={->,shorten >=3pt}, point/.style={coordinate}, pointille/.style={draw=red, top color=white, bottom color=red},scale=0.65]
 \coordinate (SEPARATOR) at (4,0);
 \coordinate (SEPARATOR2) at (-3,0);
 \draw (SEPARATOR) node [left] {\phantom{a}};
 \draw (SEPARATOR2) node [left] {\phantom{a}};
 \coordinate (A) at (0,0);
 \coordinate (B) at (-1,1);
 \coordinate (C) at (1,1.5);
 \coordinate (D) at (.5,2);
 \coordinate (E) at (1.75,2.5);
 \coordinate (F) at (2.8,1.9);
 \coordinate (G) at (3.5,2.7);
 \draw [-] ($(A)+(0,-.65)$) -- (A);
 \draw [-] (A) -- (B);
 \draw [-] (A) -- (C);
 \draw [-] (C) -- (D);
 \draw [-] (C) -- (E);
 \draw [-] (C) -- (F);
 \draw [-] (F) -- (G);
  \draw [->] (B) -- ($(B)+(-.5,.5)$);
  \draw [->] (B) -- ($(B)+(.25,.75)$);
  \draw [->] (D) -- ($(D)+(-.5,.5)$);
  \draw [->] (D) -- ($(D)+(.5,.5)$);
  \draw [->] (F) -- ($(F)+(-.5,.75)$);
  \draw [->] (F) -- ($(F)+(0,.75)$);
 \draw ($(B)+(0,-.1)$) node [left] {$1$};
 \draw ($(C)+(.3,0)$) node [below] {$\scriptsize{3}$};
 \draw ($(E)$) node [above] {$\scriptsize{{2}}$};
 \draw ($(F)+(.3,0)$) node [below] {$\scriptsize{{4}}$};
  \draw [black,fill=black] (A) circle (0.11)  ;
 \draw [black,fill=white] (B) circle (0.12)  ;
 \draw [black,fill=white] (C) circle (0.12)  ;
 \draw [black,fill=black] (D) circle (0.11)  ;
 \sq{(E)};
 \sq{(F)};
 \draw [black,fill=black] (G) circle (0.11)  ;
 \draw ($(A)+(.5,-.65)$) node {$\open$};
\end{tikzpicture}
}
\subfigure[An element of $\ca{RT}(2,2;\un{5})$.\label{fig: tree open2}]  
{  
\begin{tikzpicture}[>=stealth, arrow/.style={->,shorten >=3pt}, point/.style={coordinate}, pointille/.style={draw=red, top color=white, bottom color=red},scale=0.65]
 \coordinate (SEPARATOR) at (-3,0);
 \coordinate (SEPARATOR2) at (-3,0);
  \draw (SEPARATOR) node [left] {\phantom{a}};
  \draw (SEPARATOR2) node [left] {\phantom{a}};
 \coordinate (A) at (0,0);
 \coordinate (B) at (-1,1);
 \coordinate (C) at (1,1.5);
 \coordinate (D) at (.75,2);
 \draw [-] ($(A)+(0,-.65)$) -- (A);
 \draw [-] (A) -- (B);
 \draw [-] (A) -- (C);
 \draw [-] (C) -- (D);
  \draw [->] (B) -- ($(B)+(-.5,.5)$);
  \draw [->] (B) -- ($(B)+(.25,.75)$);
  \draw [->] (C) -- ($(C)+(+1.5,.75)$);
  \draw [->] (D) -- ($(D)+(-.5,.5)$);
  \draw [->] (D) -- ($(D)+(.5,.5)$);
 \draw ($(B)+(0,-.1)$) node [left] {$1$};
 \draw ($(C)+(.3,0)$) node [below] {$\scriptsize{2}$};
  \draw [black,fill=black] (A) circle (0.11)  ;
 \draw [black,fill=white] (B) circle (0.12)  ;
 \draw [black,fill=white] (C) circle (0.12)  ;
 \draw [black,fill=black] (D) circle (0.11)  ;
 \draw ($(A)+(.5,-.65)$) node {$\open$};
\end{tikzpicture}
}
\caption{Examples of trees in $\ca{RT}$.}
\end{figure}

The operadic structure of $\ca{RT}$ is explicitly given in the next section and corresponds to the substitution of trees into white vertices.  
\subsubsection{Correspondence between $\ca{RT}$ and $\ca{RL}_2$}\label{sec: corresp RT RL2}

Let us start by constructing a bijection of sets $\Phi:\ca{RT}(\n_1,...,\n_k;\n)\to\ca{RL}_2(\n_1,...,\n_k;\n)$ for each $k+1$-tuple $(\n_1,...,\n_k;\n)$  of colours. 
\\

\begin{itemize}
 \item The map $\Phi$. 
 For a  $T\in \ca{RT}(\n_1,...,\n_k;\n)$ let us construct an integer-string in $\ca{RL}_2(\n_1,...,\n_k;\n)$ as follows.  
One runs through the tree $T$ in clockwise direction starting from the root in such a way that one passes exactly two times on each edges (once per direction). 
 On our way, each time one meets a closed (resp. an open) vertex labelled by an $i\in\{1,...,k\}$ one writes down the corresponding label $i$ (resp. the corresponding label with an underline $\un{i}$) and  each time one meets an arrow one writes down a vertical bar. One adds an extra label $\open$ if the root is decorated by $\open$. 
 
 \item The map $\Phi^{-1}$.   To an integer-string representation one assigns a tree with: one closed (resp. open)  vertex for each different integer (resp. underlined integer) and one arrow for each vertical bar;  
  the white vertices have one input less than the number of occurrences for the corresponding integer; 
 { the corresponding tree is constructed such that its order fits with the reading (from the left to the right) of the integer-string.} 
 One adds an extra label $\open$ on the root if the integer-string is decorated by $\open$. 
 Note that when two equal integers (or two vertical bars) are adjacent in the integer-string this forces the creation of a neutral  vertex. 
\end{itemize}

  \begin{exple}
The tree from Figure \ref{fig: tree open-close2} corresponds to the integer-string $(1|1|13||3\un{2}3\un{4}|\un{4}|\un{4}\un{4}3)^{\open}$.    
  \end{exple}

  We endow $\ca{RT}$ with an operadic structure by transferring  the composition maps of $\ca{RL}_2$ along the bijections above.  %
  Explicitly, $T\circ_v T':=  \Phi^{-1}(\Phi(T)\circ_v \Phi(T'))$ for all composable $T,T'\in \ca{RT}$. 
One can check that the composition maps in $\ca{RT}$ are given by substitution of planar rooted trees into white  vertices: 
 let $T\in \ca{RT}(\n_1,...,\n_k;\n)$ and let $v$ be a vertex of $T$ labelled by $\n_i$. For a tree $T'\in \ca{RT}(f_1,...,f_j;\n_i)$, the tree $T\circ_v T'\in \ca{RT}(\n_1,...,\n_{i-1},f_1,...,f_j,\n_{i+1},...,\n_k;\n)$ is given as follows. The vertex $v$ of $T$ is substituted by the tree $T'$ in such a way that: 
 \begin{itemize}
 \item the root of $T'$ is identified with the root of $v$; 
  \item the ordered set of the  $n_i=ev(\n_i)$ arrows of $T'$ is identified with the ordered set of the   $n_i$ inputs of $v$. 
 \end{itemize}
 The set of the white vertices of $T\circ_v T'$ is labelled by the set  $\{1,...,{k+j-1}\}$  associated to $(\n_1,...,\n_{i-1},f_1,...,f_j,\n_{i+1},...,\n_k)$.

 \begin{figure}[H]
\centering
\begin{tikzpicture}[>=stealth, arrow/.style={->,shorten >=3pt}, point/.style={coordinate}, pointille/.style={draw=red, top color=white, bottom color=red},scale=0.8]
 \coordinate (A) at (0,0);
 \draw [-] (0,-0.5) -- (A) -- (-1,0.6);
 \draw [->] (A) -- (1,1);
 \draw [-] (A) -- (0.1,0.9) -- (0.5,2.5);
 \draw [<->] (-0.8,1.5) -- (0.1,0.9) -- (-0.8,2.4);
 \draw [->] (0.5,2.5) -- (0.9,2.9);
  \draw [black,fill=white] (0,0) circle (0.1)  ;
  \draw [black,fill=white] (0.5,2.5) circle (0.1)  ;
  \draw (0.1,0.9) node {$\bullet$}; 
  \draw [black,fill=white] (-1.1,0.5) rectangle (-.9,0.7) ;
   \draw (0,-0.15) node [left] {$1$};
  \draw (0.6,2.5) node [right] {$2$};
  \draw (-1.05,0.6) node [left] {${3}$};
  \draw (.3,-.4) node {$\open$};
 \draw [black] (2.3,1) circle (0.1)  ;
 \draw (2.4,0.8) node [right] {$1$};
 \coordinate (B) at (5,0);
 \draw [->] (5,-0.5) -- (B) -- (4,1);
 \draw [-] (B) -- (5,1.4);
 \draw [-] (B) -- (6,0.3);
 \draw [<->] (6.4,1) -- (6,0.3) -- (6.5,0);
  \draw [black,fill=white] (B) circle (0.1)  ;
   \draw [black,fill=white] (5,1.4) circle (0.1)  ;
   \draw (6,0.3) node {$\bullet$}; 
   \draw (5,-.15) node [left] {$1$};
   \draw (5.05,1.45) node [right] {$2$};
   \draw (7.5,1) node {$=$}; 
   \coordinate (C) at (9.4,0);
 \draw [-] (9.4,-0.5) -- (C) -- (8.4,0.6);
 \draw [->] (10,0.3) -- (10.5,0);
 \draw [-] (10.1,0.9) -- (10.5,2.5);
 \draw [<->] (9.2,1.5) -- (10.1,0.9) -- (9.2,2.4);
 \draw [-] (C) -- (9,1.1);
 \draw [-] (C) -- (10,0.3) -- (10.1,0.9);
 \draw [->] (10.5,2.5) -- (10.9,2.9);
  \draw [black,fill=white] (9.4,0) circle (0.1)  ;
  \draw [black,fill=white] (10.5,2.5) circle (0.1)  ;
  \draw (10.1,0.9) node {$\bullet$}; 
  \draw (10,0.3) node {$\bullet$}; 
  \draw [black,fill=white] (9,1.1) circle (0.1)  ;
   \draw [black,fill=white] (8.3,0.5) rectangle (8.5,0.7) ;
   \draw (9.4,-0.15) node [left] {$1$};
  \draw (10.6,2.5) node [right] {$3$};
  \draw (8.4,0.6) node [left] {${4}$};   
   \draw (9.05,1.) node [right] {$2$};
   \draw (9.7,-.4) node {$\open$};
    \draw (11.5,1) node {$=$}; 
   \coordinate (D) at (13.4,0);
 \draw [-] (13.4,-0.5) -- (D) -- (12.4,0.6);
 \draw [->] (14.1,0.9) -- (14.5,0.6);
 \draw [-] (14.1,0.9) -- (14.5,2.5);
 \draw [<->] (13.2,1.5) -- (14.1,0.9) -- (13.2,2.4);
 \draw [-] (D) -- (13,1.1);
 \draw [-] (D)  -- (14.1,0.9);
 \draw [->] (14.5,2.5) -- (14.9,2.9);
  \draw [black,fill=white] (13.4,0) circle (0.1)  ;
  \draw [black,fill=white] (14.5,2.5) circle (0.1)  ;
  \draw (14.1,0.9) node {$\bullet$}; 
  \draw [black,fill=white] (13,1.1) circle (0.1)  ;
   \draw [black,fill=white] (12.3,0.5) rectangle (12.5,0.7) ;
   \draw (13.4,-0.15) node [left] {$1$};
  \draw (14.6,2.5) node [right] {$3$};
  \draw (12.4,0.6) node [left] {${4}$};   
   \draw (13.05,1.) node [right] {$2$};
   \draw (13.7,-.4) node {$\open$};
\end{tikzpicture}\caption{An example of composition in $\ca{RT}$.}\label{fig: compo}
\end{figure}

By construction one has: 
\begin{prop}\label{prop: tree}
 The coloured operads  $\ca{RL}_2$ and $\ca{RT}$  are isomorphic.
\end{prop}

\subsection{The operad $\ca{RL}'_2$ in term of trees}
This section is the analogue to the previous section for $\ca{RL}'_2$. 
We define an operad $\ca{RT}'$ as follows. 
For a $(k+1)$-tuple $(\n_1,...,\n_k;\n)$ of  colours in $Col$, 
 the set $\ca{RT}'(\n_1,...,\n_k;\n)$ is defined as $\ca{RT}(\n_1,...,\n_k;\n)$ is, except that the condition (F2) from Section \ref{sec: pre def RT} is replaced by the condition: 
\begin{itemize}
 \item[(F2')] there is no white vertex below an open vertex 
 \ie if $\nu$ is a white vertex of $T$, then in the tree $T_{\nu}$  either the vertex $\nu$ is the unique open  vertex or there is no open vertex. 
\end{itemize}
In particular, for the closed part, $\ca{RT}'(n_1,...,n_k;{n})=\ca{RT}(n_1,...,n_k;{n})$. 
 \\
 
The operadic composition for $\ca{RT}'$ is as for $\ca{RT}$ and we have the analogue of Proposition \ref{prop: tree}:
\begin{prop}\label{prop: tree right}
 The coloured operads  $\ca{RL}'_2$ and $\ca{RT}'$  are isomorphic.
\end{prop}
Here is an example of an element of $\ca{RT}'$ and its corresponding element in $\ca{RL}'_2$:
\begin{equation*}
 \ca{RT}'(3,2,\un{2};\un{6}) \ni
 \begin{tikzpicture}[>=stealth, arrow/.style={->,shorten >=3pt}, point/.style={coordinate}, pointille/.style={draw=red, top color=white, bottom color=red},scale=0.65,baseline=.9ex]
 \coordinate (A) at (0,0);
 \coordinate (B) at (-1,1);
 \coordinate (C) at (1,1.5);
 \coordinate (D) at (.75,2);
 \draw [-] ($(A)+(0,-.65)$) -- (A);
 \draw [-] (A) -- (B);
 \draw [-] (A) -- (C);
 \draw [-] (C) -- (D);
  \draw [->] (B) -- ($(B)+(-.5,.5)$);
  \draw [->] (B) -- ($(B)+(.25,.75)$);
  \draw [->] (B) -- ($(B)+(.75,.5)$);
  \draw [->] (C) -- ($(C)+(+1.5,.75)$);
  \draw [->] (D) -- ($(D)+(-.5,.5)$);
  \draw [->] (D) -- ($(D)+(.5,.5)$);
 \draw ($(A)+(0,-.2)$) node [below,left] {${3}$};
 \draw ($(B)+(0,-.1)$) node [left] {$1$};
 \draw ($(C)+(.3,0)$) node [below] {$\scriptsize{2}$};
 \draw ($(A)+(.5,-.65)$) node {$\open$};
 \draw [black,fill=white] (-.1,-.1) rectangle (.1,.1) ;
 \draw [black,fill=white] (B) circle (0.11)  ;
 \draw [black,fill=white] (C) circle (0.11)  ;
 \draw [black,fill=black] (D) circle (0.11)  ;
\end{tikzpicture}%
\leftrightarrow (\un{3}1|1|1|1\un{3}2||2|2\un{3})^{\open} \in \ca{RL}_2'(3,2,\un{2};\un{6}).
\end{equation*}

\subsection{A few remarks on the operads $\ca{RS}_2$ and $\ca{RS}'_2$}

\begin{prop}\label{prop: generators RS2}
As an operad,  $\ca{RS}_2$ is generated by the following elements 
\begin{align*}
\mu_\close&=(12), ~~ T_k=(1213\cdots 1k1) \text{ for } k\geq 2, \text{ (closed part)} \\
\mu_{\open}&=(\un{1}\un{2})^{\open}, \text{ (open part)} \\
 T_{\underline{j}}&=(1\un{2}1\un{3}\cdots 1\un{j}1 )^{\open}  \text{ for } j\geq 2, ~~
inc=(1)^{\open},
\end{align*}
and the two unit elements $id_\close=(1)$ and $id_\open=(\un{1})^{\open}$. As an operad,  $\ca{RS}'_2$ is generated by 
\begin{align*}
\mu_\close, ~~ T_k \text{ for } k\geq 2,~~ \mu_{\open}, ~~
 T'_{\underline{j}}=(\un{1}2\un{1}3\cdots \un{1}j\un{1} )^{\open}  \text{ for } j\geq 2, ~~
inc,~~ id_\close \text{ and } id_\open. 
\end{align*} 
\end{prop}

\begin{proof}
We prove the statement for $\ca{RS}_2$, the case $\ca{RS}'_2$ is similar. 
We suppose by induction on $N$ that any integer-string of $\ca{RS}_2$ with $N$ different integers  is obtained by operadic compositions of elements cited in the statement. The cases $N=1$ and $N=2$ are trivially verified. \\
In what follows we abusively do not mark the distinction between underlined and non underlined integers. 
Let $x$ be an integer-string of $\ca{RS}_2$ with $N+1$ different integers. 
Because of the filtration condition \eqref{eq: filtra}, $x$ can be written as a sequence $(A_1\cdots A_n)$ where the $A_i$'s are non empty sequences of integers such that, if $j$ belongs to $A_i$, then $j\notin A_s$ for $s\neq i$. Moreover, because of the symmetric group action, one can suppose that the integers of $A_i$ are smaller than the integers of $A_j$ whenever $i<j$.  In this case, if $n>1$, then $x=(\alpha \circ_1(A_1\cdots A_{n-1}))\circ_{\max A_{n-1}+1} (\widetilde{A}_n)$ where $\widetilde{A}_n$ is obtained from $A_n$ by decreasing each number by $\max A_{n-1}$ and $\alpha$ is $(12), (\un{1}\un{2})^\open,(\un{1}2)^\open$, $(1\un{2})^\open$ or $(12)^\open$.  
Since the $A_i$'s are not empty, $(A_1\cdots A_{n-1})$ as well as $A_n$ have at most $N$ different integers and thus satisfy the induction hypotheses. 
If $n=1$ then either $x$ is $T_k$ (or $T_{\un{k}}$) for some $k$ or, $x$ is such that 
$A_1=jB_1jB_2j\cdots jB_pj$ with $1\leq p<N$ and for some integer $j$. Thus there exists at least one $B_{i_0}$ that contains  $2\leq q\leq N-(p-1)$ elements and  $x=(jB_1j\cdots B_{i_0-1}jajB_{i_0+1}j\cdots jB_{p}j)\circ_a \wid{B}_{i_0}$ for some $a$, which concludes the proof.  
\end{proof}

Via Proposition \ref{prop: tree}, the generators of $\ca{RS}_2$ described in Proposition \ref{prop: generators RS2} above, corresponds to the trees: 
\begin{equation*}
\begin{split}
\left[
id_\close=
\begin{tikzpicture}[>=stealth, arrow/.style={->,shorten >=3pt}, point/.style={coordinate}, pointille/.style={draw=red, top color=white, bottom color=red},scale=0.75,baseline=0ex]
 \coordinate (A) at (0,0);
 \draw [-] (A) -- (0,-.5);
 \draw (A) node [above] {\small{$1$}};
 \draw [black,fill=white] (A) circle (0.1)  ;
 \end{tikzpicture},
 ~~
 \mu_\close=
 \begin{tikzpicture}[>=stealth, arrow/.style={->,shorten >=3pt}, point/.style={coordinate}, pointille/.style={draw=red, top color=white, bottom color=red},scale=0.75,baseline=.5ex]%
 \coordinate (A) at (0,0);
 \coordinate (B) at (-.75,.75);
 \coordinate (C) at (.75,.75);
 \draw [-] (A) -- (0,-.5);
 \draw [-] (B) -- (A) -- (C);
 \draw (B) node [above] {\small{$1$}};
 \draw (C) node [above] {\small{$2$}};
 \draw [black,fill=black] (A) circle (0.1)  ;
 \draw [black,fill=white] (B) circle (0.1)  ;
 \draw [black,fill=white] (C) circle (0.1)  ;
 \end{tikzpicture},
 ~~
 T_k=
 \begin{tikzpicture}[>=stealth, arrow/.style={->,shorten >=3pt}, point/.style={coordinate}, pointille/.style={draw=red, top color=white, bottom color=red},scale=0.75,baseline=.5ex]%
 \coordinate (A) at (0,0);
 \coordinate (B) at (-1.5,.75);
 \coordinate (C) at (-.75,.75);
 \coordinate (D) at (.75,.75);
 \coordinate (E) at (1.5,.75);
 \draw [-] (A) -- (0,-.5);
 \draw [-] (B) -- (A) -- (C);
 \draw [-] (D) -- (A) -- (E);
 \draw (A) node [left,yshift=-1.5mm] {\small{$1$}};
 \draw (B) node [above] {\small{$2$}};
 \draw (C) node [above] {\small{$3$}};
 \draw (D) node [above] {\small{$k-1$}};
 \draw (E) node [above] {\small{$k$}};
 \draw (0,.5) node [above] {$\cdots$};
 \draw [black,fill=white] (A) circle (0.1)  ;
 \draw [black,fill=white] (B) circle (0.1)  ;
 \draw [black,fill=white] (C) circle (0.1)  ;
 \draw [black,fill=white] (D) circle (0.1)  ;
 \draw [black,fill=white] (E) circle (0.1)  ;
 \end{tikzpicture}
 \right]
 & \text{(closed part)}\\
\left[
\underbrace{id_\open=
\begin{tikzpicture}[>=stealth, arrow/.style={->,shorten >=3pt}, point/.style={coordinate}, pointille/.style={draw=red, top color=white, bottom color=red},scale=0.75,baseline=.0ex]%
 \coordinate (A) at (0,0);
 \draw [-] (A) -- (0,-.5);
 \draw (A) node [above] {\small{${1}$}};
 \sq{(A)};
 \draw (A) node [right,yshift=-3mm] {$\open$};
 \end{tikzpicture},
 ~~
 \mu_\open=
 \begin{tikzpicture}[>=stealth, arrow/.style={->,shorten >=3pt}, point/.style={coordinate}, pointille/.style={draw=red, top color=white, bottom color=red},scale=0.75,baseline=.5ex]%
 \coordinate (A) at (0,0);
 \coordinate (B) at (-.75,.75);
 \coordinate (C) at (.75,.75);
 \draw [-] (A) -- (0,-.5);
 \draw [-] (B) -- (A) -- (C);
 \draw (B) node [above] {\small{${1}$}};
 \draw (C) node [above] {\small{${2}$}};
 \draw [black,fill=black] (A) circle (0.1)  ;
 \sq{(B)};
 \sq{(C)};
 \draw (A) node [right,yshift=-3mm] {$\open$};
 \end{tikzpicture}}_{\text{(open part)}}, 
~ inc=
\begin{tikzpicture}[>=stealth, arrow/.style={->,shorten >=3pt}, point/.style={coordinate}, pointille/.style={draw=red, top color=white, bottom color=red},scale=0.75,baseline=0ex]%
 \coordinate (A) at (0,0);
 \draw [-] (A) -- (0,-.5);
 \draw (A) node [above] {$1$};
 \draw [black,fill=white] (A) circle (0.1)  ;
 \draw (A) node [right,yshift=-3mm] {$\open$};
 \end{tikzpicture},
 ~~
 T_{\un{j}}=
 \begin{tikzpicture}[>=stealth, arrow/.style={->,shorten >=3pt}, point/.style={coordinate}, pointille/.style={draw=red, top color=white, bottom color=red},scale=0.75,baseline=.5ex]%
 \coordinate (A) at (0,0);
 \coordinate (B) at (-1.5,.75);
 \coordinate (C) at (-.75,.75);
 \coordinate (D) at (.75,.75);
 \coordinate (E) at (1.5,.75);
 \draw [-] (A) -- (0,-.5);
 \draw [-] (B) -- (A) -- (C);
 \draw [-] (D) -- (A) -- (E);
 \draw (A) node [left,yshift=-1.5mm] {\small{$1$}};
 \draw (B) node [above] {\small{${2}$}};
 \draw (C) node [above] {\small{${3}$}};
 \draw (D) node [above] {\small{${j-1}$}};
 \draw (E) node [above] {\small{${j}$}};
 \draw (0,.5) node [above] {$\cdots$};
 \draw [black,fill=white] (A) circle (0.1)  ;
 \sq{(B)};
 \sq{(C)}; 
 \sq{(D)}; 
 \sq{(E)};  
 \draw (A) node [right,yshift=-3mm] {$\open$};
 \end{tikzpicture}
 \right]
 & \text{(non-closed part).}
 \end{split}
\end{equation*}

For $\ca{RS}_2'$, the element $T'_{\un{j}}$ corresponds to 
$ \begin{tikzpicture}[>=stealth, arrow/.style={->,shorten >=3pt}, point/.style={coordinate}, pointille/.style={draw=red, top color=white, bottom color=red},scale=0.75,baseline=.5ex]%
 \coordinate (A) at (0,0);
 \coordinate (B) at (-1.5,.75);
 \coordinate (C) at (-.75,.75);
 \coordinate (D) at (.75,.75);
 \coordinate (E) at (1.5,.75);
 \draw [-] (A) -- (0,-.5);
 \draw [-] (B) -- (A) -- (C);
 \draw [-] (D) -- (A) -- (E);
 \draw (A) node [left,yshift=-1.5mm] {\small{${1}$}};
 \draw (B) node [above] {\small{$2$}};
 \draw (C) node [above] {\small{$3$}};
 \draw (D) node [above] {\small{$j-1$}};
 \draw (E) node [above] {\small{$j$}};
 \draw (0,.5) node [above] {$\cdots$};
 \sq{(A)}; 
 \draw [black,fill=white] (B) circle (0.1)  ;
 \draw [black,fill=white] (C) circle (0.1)  ;
 \draw [black,fill=white] (D) circle (0.1)  ;
 \draw [black,fill=white] (E) circle (0.1)  ;
 \draw (A) node [right,yshift=-3mm] {$\open$};
 \end{tikzpicture}.$

 The operadic structure of $\ca{RS}_2$ is described in terms of trees by means of Proposition \ref{prop: tree}; 
the whiskering $w$ defined in Section \ref{sec: rel surj operad} has a corresponding map on $\ca{RT}$ (roughly, it consists in adding arrows linked to the white vertices) and the composition of two trees is given by transferring the formula \eqref{eq: compo RS} to $\ca{RT}$. 
One has a similar description for $\ca{RS}'_2$. 
For the closed part, detailed are given in \cite{BBM} with a slightly different convention. 
A full description of the operad $\ca{RS}'_2$ in terms of similar trees is given in \cite{HLQ-SC}.

\subsection{The algebras over $\ca{RL}_2$}\label{sec: def repres}

We describe the algebras over $\ca{RL}_2$ where, implicitly, $\ca{RL}_2$ is seen  in $\cat{C}$ by means of the strong  symmetric monoidal functor $\cat{Set}\to \cat{C}$,  $E\mapsto \coprod_{e\in E} 1_{\cat{C}}$. 
Precisely, we show that  $\ca{RL}_2$ encodes the pairs $(\ca{M},\ca{Z})$ subject to the following conditions.
\begin{enumerate}[I.]
 \item $\ca{M}$ is a multiplicative non-$\Sigma$ operad.\label{A1} 
 \item $\ca{Z}\in \bmodmas$ and there is a morphism $\iota:\ca{M}\to \ca{Z}$ in $\bmodmas$.  \label{A2} 
\end{enumerate}

Let $\cat{E}$ be the category with objects the pairs $(\ca{M},\ca{Z})$ satisfying the two conditions \ref{A1} and \ref{A2} above; morphisms are the pairs 
$(f,g): (\ca{M},\ca{Z})\to (\ca{M}',\ca{Z}')$  subject to the following conditions:
\begin{itemize}
 \item $f:\ca{M}\to \ca{M}'$ is a morphism of multiplicative non-$\Sigma$ operads;
 \item $g:\ca{Z}\to\ca{Z}'$ is an $f$-equivariant morphism of left modules over $\ca{M}$; 
 \item $g$ is a morphism of bimodules over $\ca{A}s$; and, 
 \item $\iota'\circ f=g\circ \iota$. 
\end{itemize}
 
 \begin{rmq}
Note that $\cat{E}$ is well defined since, by the last condition $\iota'\circ f=g\circ \iota$, the morphism  $\iota$ of \ref{A2} associated to $(\ca{M},\ca{Z})$ is unique. 
 Remark that the pair $(\ca{A}s,\ca{A}s)$ is the initial object of $\cat{E}$.   
 \end{rmq}

\begin{rmq}
As observed in Example \ref{ex: weak mod and map}, since here we have a morphism $\ca{A}s\to \ca{Z}$, it follows from \ref{A2}  that $\ca{Z}$ is, in particular, a weak bimodule over $\ca{A}s$ (see also proof of Proposition \ref{prop: hypI,II,III correspondance}). By Lemma \ref{lem: ASSOC-weak=cosimpl} this is equivalent to say that $\ca{Z}$ is endowed with a  cosimplicial structure. Likewise,  the multiplicative structure of $\ca{M}$ endows it with a cosimplicial structure. 
These are the cosimplicial structures involved in the $\delta$-totalization that gives rise to the $\Coend_{\ca{RL}_2}(\delta)$-algebra  $(Tot_{\delta}\ca{M},Tot_{\delta}\ca{Z})$. 
\end{rmq}

\begin{prop}\label{prop: hypI,II,III correspondance} 
 The category of $\ca{RL}_2$-algebras in $\cat{C}$ is isomorphic to the category $\cat{E}$.
\end{prop}

\begin{proof} 
We use the interpretation of $\ca{RL}_2$ in terms of planar trees, see Proposition \ref{prop: tree}. 

Let us start with an $\ca{RL}_2$-algebra $(\ca{M},\ca{Z})$. This means that for each $T\in \ca{RL}_2(\n_1,...,\n_p;\n)$ one has a corresponding operation: 
\begin{equation*}
 T\in \Hom( \textit{products of $\ca{M}(n_i)$ and $\ca{Z}(n_j)$}, \ca{Z}(n)) \text{ or, } T\in \Hom(\textit{products of $\ca{M}(n_i)$}, \ca{M}(n)), 
\end{equation*}
  according to the type of  the  output $\n$, where $n_s=ev(\n_s)$ is the value of $e_s$, $1\leq s\leq p$ and $n=ev(\n)$. 

Let us write explicitly the induced structures on $\ca{M}$ and $\ca{Z}$. 
\begin{enumerate}
 \item Operadic structure on $\ca{M}$: the partial compositions on elementary trees are given by
 \begin{equation*}
\scalebox{.75}{
\begin{tikzpicture}[>=stealth, arrow/.style={->,shorten >=3pt}, point/.style={coordinate}, pointille/.style={draw=red, top color=white, bottom color=red},scale=.8,baseline=-0.9ex]%
 \coordinate (A) at (0,0);
 \coordinate (B) at (1.2,1.5);
 \draw [->] (A) -- (-1,1.2);
 \draw [->] (A) -- (0.2,1.2);
 \draw [-] (A) -- (B);
 \draw [->] (A) -- (1.2,0.9);
 \draw [->] (A) -- (1.8,0.6);
 \draw (-1,1.2) node [above] {$1$};
 \draw (.2,1.2) node [above] {$i-1$};
 \draw (1.2,.9) node [above,right] {$i+1$};
 \draw (1.8,.6) node [above,right] {$k$};
 \draw [->] (B) -- ($(B)+ (-.5,.5)$) ;
 \draw [->] (B) -- ($(B)+ (-.2,.6)$) ;
 \draw [->] (B) -- ($(B)+ (.5,.6)$) ;
 \draw ($(B)+ (-.5,.5)$) node [above,yshift=-1mm] {$1$};
 \draw ($(B)+ (.5,.6)$) node [right] {$l$};
  \draw [-]  (0,-0.5) -- (A);
 \draw [black,fill=white] (B) circle (0.13)  ;
  \draw (-0.2,.8) node {$...$};
  \draw (1.3,.6) node {$...$}; 
  \draw [black,fill=white] (0,0) circle (0.13)  ;
\end{tikzpicture}}
 \leftrightarrow\circ_i :
\ca{M}(k)\ot \ca{M}(l)\to \ca{M}(k+l-1).
\end{equation*} 
The operadic structure of $\ca{RL}_2$ provides the associativity condition for partial compositions of $\ca{M}$.   This is similar to (3) below, which we detail. 
 
 \item The multiplicative structure on $\ca{M}$: the morphism $\alpha:\ca{A}s(k)\to \ca{M}(k)$ is given by the corolla in a neutral  vertex with $k$ inputs.   For $k=1$, one has $\alpha=\eta:1_{\cat{C}}\to \ca{M}(1)$ is the unit.  Note that the isomorphism $\ca{A}s(k)\ot \ca{A}s(l)\to \ca{A}s(k+l-1)$ corresponds to the equivalence relation made on the neutral  vertices (see Definition \ref{de: trees}). 
 The multiplication in $\ca{M}$ is given by operations as: 
 \begin{equation*}
\begin{tikzpicture}[>=stealth, arrow/.style={->,shorten >=3pt}, point/.style={coordinate}, pointille/.style={draw=red, top color=white, bottom color=red},scale=0.5,baseline=-0.9ex]%
 \coordinate (A) at (0,0);
 \coordinate (B) at (-.8,.8);
 \coordinate (C) at (.8,.8);
  \draw [-] (A) -- (0,-.5);
 \draw [-] (A) -- (B) ;
 \draw [-] (A) -- (C) ;
 \draw [->] (B) -- ($(B)+ (-0.6,.7)$);
 \draw [->] (B) -- ($(B)+ (0.6,.7)$);
 \draw [->] (C) -- ($(C)+ (-0.6,.7)$);
 \draw [->] (C) -- ($(C)+ (0.1,.8)$);
 \draw [->] (C) -- ($(C)+ (0.6,.7)$);
  \draw [black,fill=black] (A) circle (0.13)  ;
  \draw [black,fill=white] (B) circle (0.15)  ;
  \draw [black,fill=white] (C) circle (0.15)  ;
\end{tikzpicture}%
:  \ca{M}(2)\ot \ca{M}(3)=  \ca{A}s(2)\ot \ca{M}(2)\ot \ca{M}(3)\to  \ca{M}(5). 
 \end{equation*}
 \item Left action of $\ca{M}$ on $\ca{Z}$. 
 The $k$-corollas in a closed vertex with $k$ open vertices at the inputs give the left action $\lambda$: 
\begin{equation*}
\begin{tikzpicture}[>=stealth, arrow/.style={->,shorten >=3pt}, point/.style={coordinate}, pointille/.style={draw=red, top color=white, bottom color=red},scale=0.7,baseline=-0.9ex]%
 \coordinate (A) at (0,0);
 \coordinate (B) at (-1.55,.7);
 \coordinate (C) at (-.3,1.1);
 \coordinate (D) at (1.5,1);
 \draw [-] (A) -- (B);
 \draw [->] (B) -- ($(B)+ (-.5,.5)$) ;
 \draw [->] (B) -- ($(B)+ (-.2,.6)$) ;
 \draw [->] (B) -- ($(B)+ (.3,.6)$) ;
 \draw [->] (C) -- ($(C)+ (-.5,.5)$) ;
 \draw [->] (C) -- ($(C)+ (.3,.6)$) ;
\draw (A) node [left,yshift=-1mm] {\tiny{$1$}};
  \draw (B) node [left] {\tiny{$2$}};
 \draw (C) node [left] {\tiny{${3}$}};
 \draw [-] (A) -- (C);
  \draw [-]  (0,-0.5) -- (A);
 \sq{(B)};
 \draw [-] (A) -- (D);
 \draw [->] (D) -- ($(D)+ (-.5,.5)$) ;
 \draw [->] (D) -- ($(D)+ (-.2,.6)$) ;
 \draw [->] (D) -- ($(D)+ (.3,.6)$) ;
 \sq{(D)};
 \draw (D) node [left] {\tiny{$k+1$}};
 \draw ($(B)+ (-.5,.5)$)  node [left] {\tiny{$1$}};
 \draw ($(B)+ (.3,.6)$)  node [above,yshift=-1.5mm] {\tiny{$m_1$}};
 \draw ($(D)+ (-.5,.5)$)  node [left,xshift=1mm] {\tiny{$1$}};
 \draw ($(D)+ (.3,.6)$)  node [right,xshift=-1mm] {\tiny{$m_k$}};
 \draw [black,fill=white] ($(C)-(.1,.1)$) rectangle ($(C)+(.1,.1)$) ;
  \draw (.45,.8) node {$...$}; 
  \draw [black,fill=white] (0,0) circle (0.11)  ;
  \draw (0,-.5) node [right] {\scriptsize{$\open$}};
\end{tikzpicture}%
\leftrightarrow \lambda: 
\ca{M}(k)\ot \ca{Z}(m_1)\ot \cdots \ot \ca{Z}(m_k) \to \ca{Z}(m_1+...+m_k).
 \end{equation*}
 The operadic structure of $\ca{RL}_2$ gives the associativity of the left action:  
 \begin{align*}
\scalebox{.9}{
\begin{tikzpicture}[>=stealth, arrow/.style={->,shorten >=3pt}, point/.style={coordinate}, pointille/.style={draw=red, top color=white, bottom color=red},scale=0.65,baseline=3ex]%
 \coordinate (A) at (0,0);
 \coordinate (B) at (-1.6,.3);
 \coordinate (C) at (-.2,1.2);
 \coordinate (D) at (1.5,1.3);
 \coordinate (E) at (1.5,.6);
 \coordinate (F) at (1.4,-.1);
 \draw [-]  (0,-0.5) -- (A);
 \draw [-] (A) -- (B);
 \draw [-] (A) -- (C) ;
 \draw [-] (A) -- (D);
 \draw [-] (A) -- (E) ;
 \draw [-] (A) -- (F) ;
 \draw [->] (B) -- ($(B)+ (-.5,.5)$) ;
 \draw [->] (B) -- ($(B)+ (-.2,.6)$) ;
 \draw [->] (B) -- ($(B)+ (.3,.6)$) ;
 \draw [->] (C) -- ($(C)+ (-.5,.5)$) ;
 \draw [->] (C) -- ($(C)+ (.3,.6)$) ;
 \draw [->] (F) -- ($(F)+ (.5,.1)$) ;
 \draw [->] (F) -- ($(F)+ (.5,.-.2)$) ;
 \draw (B) node [left] {\tiny{$3$}};
 \draw (C) node [left] {\tiny{${i+1}$}};
 \draw (E) node [right] {\tiny{${i+2+l}$}};
 \draw (F) node [below,xshift=2mm] {\tiny{$k+l+2$}};
 \sq{(B)};
 \sq{(C)};
 \sq{(E)};
 \sq{(F)};
  \draw (-0.5,.5) node {$...$};
  \draw [black,fill=white] (0,0) circle (0.13)  ;
  \draw (0,-.5) node [right] {\scriptsize{$\open$}};
 \coordinate (A1) at (D);
 \coordinate (B1) at ($(D)+(-1.,.8)$);
 \coordinate (C1) at ($(D)+(.5,.9)$);
 \coordinate (D1) at ($(D)+(1.2,.7)$);
 \draw [-] (A1) -- (B1);
 \draw [-] (A1) -- (C1);
 \draw [->] (B1) -- ($(B1)+ (-.5,.5)$) ;
 \draw [->] (B1) -- ($(B1)+ (-.2,.6)$) ;
 \draw [->] (B1) -- ($(B1)+ (.3,.6)$) ;
 \draw [->] (C1) -- ($(C1)+ (-.5,.5)$) ;
 \draw [->] (C1) -- ($(C1)+ (.3,.6)$) ;
 \draw [-] (A1) -- (D1);
 \draw [->] (D1) -- ($(D1)+ (-.5,.5)$) ;
 \draw [->] (D1) -- ($(D1)+ (-.2,.6)$) ;
 \draw [->] (D1) -- ($(D1)+ (.3,.6)$) ;
 \draw (B1) node [left] {\scriptsize{$i+2$}};
 \draw (D1) node [right] {\scriptsize{$i+1+l$}};
 \draw (A) node [left,yshift=-1.2mm] {\tiny{$1$}};
 \draw (D) node [below] {\tiny{$2$}};
 \sq{(B1)};
 \sq{(C1)};
 \sq{(D1)};
  \draw ($(D)+(-0.2,.8)$) node {$...$};
  \draw [black,fill=white] (A1) circle (0.12)  ;
\end{tikzpicture}%
}  
=
\sigma_{(2,i),i+1}\left(
\scalebox{.9}{
\begin{tikzpicture}[>=stealth, arrow/.style={->,shorten >=3pt}, point/.style={coordinate}, pointille/.style={draw=red, top color=white, bottom color=red},scale=0.65,baseline=2ex]%
 \coordinate (A) at (0,0);
 \coordinate (B) at (-1.6,.3);
 \coordinate (C) at (-.2,1.2);
 \coordinate (D) at (1.5,1.3);
 \coordinate (E) at (1.5,.6);
 \coordinate (F) at (1.4,-.1);
 \draw [-]  (0,-0.5) -- (A);
 \draw [-] (A) -- (B);
 \draw [-] (A) -- (C) ;
 \draw [-] (A) -- (D);
 \draw [-] (A) -- (E) ;
 \draw [-] (A) -- (F) ;
 \draw [->] (B) -- ($(B)+ (-.5,.5)$) ;
 \draw [->] (B) -- ($(B)+ (-.2,.6)$) ;
 \draw [->] (B) -- ($(B)+ (.3,.6)$) ;
 \draw [->] (C) -- ($(C)+ (-.5,.5)$) ;
 \draw [->] (C) -- ($(C)+ (.3,.6)$) ;
 \draw [->] (F) -- ($(F)+ (.5,.1)$) ;
 \draw [->] (F) -- ($(F)+ (.5,.-.2)$) ;
 \draw (A) node [left,yshift=-1.2mm] {\tiny{$1$}};
 \draw (B) node [left] {\tiny{$2$}};
 \draw (C) node [left] {\tiny{${i}$}};
 \draw (E) node [right] {\tiny{${i+2}$}};
 \draw (F) node [below,xshift=2mm] {\tiny{$k+1$}};
 \sq{(B)};
 \sq{(C)};
 \sq{(E)};
 \sq{(F)};
  \draw (-0.5,.5) node {$...$};
  \draw [black,fill=white] (0,0) circle (0.13)  ;
  \draw (0,-.5) node [right] {\scriptsize{$\open$}};
 \coordinate (A1) at (D);
 \draw [->] (A1) -- ($(A1)+ (-.5,.5)$) ;
 \draw [->] (A1) -- ($(A1)+ (.5,.6)$) ;
 \draw (A1) node [right] {\scriptsize{$i+1$}};
 \draw ($(A1)+ (-.5,.5)$) node [above,yshift=-1mm] {\scriptsize{$1$}};
 \draw ($(A1)+ (.5,.6)$) node [above,yshift=-1mm] {\scriptsize{$x$}};
 \draw ($(A1)+(-0,.4)$) node {$...$};
 \sq{(A1)};
\end{tikzpicture}%
}
\circ_{i+1}
\scalebox{.9}{
\begin{tikzpicture}[>=stealth, arrow/.style={->,shorten >=3pt}, point/.style={coordinate}, pointille/.style={draw=red, top color=white, bottom color=red},scale=0.65,baseline=1.5ex]%
 \coordinate (D) at (0,0);
\coordinate (A1) at (D);
 \coordinate (B1) at ($(D)+(-1.,1.)$);
 \coordinate (C1) at ($(D)+(.5,1.1)$);
 \coordinate (D1) at ($(D)+(1.2,.8)$);
 \draw [-] (A1) -- (0,-.5);
 \draw [-] (A1) -- (B1);
 \draw [-] (A1) -- (C1);
 \draw [->] (B1) -- ($(B1)+ (-.5,.5)$) ;
 \draw [->] (B1) -- ($(B1)+ (-.2,.6)$) ;
 \draw [->] (B1) -- ($(B1)+ (.3,.6)$) ;
 \draw [->] (C1) -- ($(C1)+ (-.5,.5)$) ;
 \draw [->] (C1) -- ($(C1)+ (.3,.6)$) ;
 \draw [-] (A1) -- (D1);
 \draw [->] (D1) -- ($(D1)+ (-.5,.5)$) ;
 \draw [->] (D1) -- ($(D1)+ (-.2,.6)$) ;
 \draw [->] (D1) -- ($(D1)+ (.3,.6)$) ;
 \draw (A) node [left,yshift=-1.2mm] {\tiny{$1$}};
 \draw (B1) node [left] {\scriptsize{$2$}};
 \draw (D1) node [right] {\scriptsize{$l+1$}};
 \sq{(B1)};
 \sq{(C1)};
 \sq{(D1)};
  \draw ($(D)+(-0.2,.8)$) node {$...$};
  \draw [black,fill=white] (A1) circle (0.12)  ;
  \draw (0,-.5) node [right] {\scriptsize{$\open$}};
\end{tikzpicture}%
}
\right)
\\
=
\scalebox{.9}{
\begin{tikzpicture}[>=stealth, arrow/.style={->,shorten >=3pt}, point/.style={coordinate}, pointille/.style={draw=red, top color=white, bottom color=red},scale=0.65,baseline=1.5ex]%
 \coordinate (A) at (0,0);
 \coordinate (B) at (-1.6,.3);
 \coordinate (C) at (-.2,1.2);
 \coordinate (D) at (.9,1.5);
 \coordinate (E) at (1.85,1.5);
 \coordinate (F) at (2.6,1.);
 \coordinate (G) at (1.5,.3);
 \coordinate (H) at (1.4,-.2);
 \draw [-] (A) -- (B);
 \draw [-] (A) -- (C);
 \draw [-] (A) -- (G);
 \draw [-] (A) -- (H);
 \draw [->] (B) -- ($(B)+ (-.5,.5)$) ;
 \draw [->] (B) -- ($(B)+ (-.2,.6)$) ;
 \draw [->] (B) -- ($(B)+ (.3,.6)$) ;
 \draw [->] (C) -- ($(C)+ (-.5,.5)$) ;
 \draw [->] (C) -- ($(C)+ (.3,.6)$) ;
 \draw (B) node [left] {\tiny{$2$}};
 \draw (C) node [left] {\tiny{$i+1$}};
   \draw [-]  (0,-0.5) -- (A);
 \draw [->] (D) -- ($(D)+ (-.5,.5)$) ;
 \draw [->] (D) -- ($(D)+ (-.2,.6)$) ;
 \draw [->] (D) -- ($(D)+ (.3,.6)$) ;
 \draw (D) node [left] {\tiny{$i+2$}};
 \draw (F) node [right] {\tiny{$i+1+l$}};
 \draw [line width=1.2,-] (A) -- (D);
 \draw [line width=1.2,-] (A) -- (F);
 \draw [line width=1.2,-] (A) -- (E);
 \draw [->] (E) -- ($(E)+ (-.5,.5)$) ;
 \draw [->] (E) -- ($(E)+ (.3,.6)$) ;
 \draw [->] (F) -- ($(F)+ (-.5,.5)$) ;
 \draw [->] (F) -- ($(F)+ (-.2,.6)$) ;
 \draw [->] (F) -- ($(F)+ (.3,.6)$) ;
 \draw [->] (H) -- ($(H)+ (.5,.1)$) ;
 \draw [->] (H) -- ($(H)+ (.5,.-.2)$) ;
  \draw (-0.5,.5) node {$...$};
  \draw (G) node [right] {\tiny{$i+2+l$}};
  \draw (H) node [below,xshift=1.5mm] {\tiny{$k+l+2$}};
  \draw [black,fill=white] (0,0) circle (0.12)  ;
  \draw (0,-.5) node [right] {\scriptsize{$\open$}};
  \draw (0,-.2) node [left] {\scriptsize{$1$}};
  \sq{(B)};
 \sq{(C)};
 \sq{(D)};
 \sq{(E)};
 \sq{(F)};
 \sq{(G)};
 \sq{(H)};
\end{tikzpicture}%
}
\circ_{1}
\scalebox{.9}{
\begin{tikzpicture}[>=stealth, arrow/.style={->,shorten >=3pt}, point/.style={coordinate}, pointille/.style={draw=red, top color=white, bottom color=red},scale=.65,baseline=1.5ex]%
 \coordinate (A) at (0,0);
 \coordinate (B) at (1.2,1.5);
 \draw [->] (A) -- (-1,1.2);
 \draw [->] (A) -- (0.2,1.2);
 \draw [-] (A) -- (B);
 \draw [->] (A) -- (1.2,0.9);
 \draw [->] (A) -- (1.8,0.6);
 \draw (-1,1.2) node [above,yshift=-1mm] {\scriptsize{$1$}};
 \draw (.2,1.2) node [above,yshift=-1mm] {\scriptsize{$i-1$}};
 \draw (1.2,.9) node [above,right] {\scriptsize{$i+1$}};
 \draw (1.8,.6) node [above,right] {\scriptsize{$k$}};
 \draw [line width=1.2,->] (B) -- ($(B)+ (-.5,.5)$) ;
 \draw [line width=1.2,->] (B) -- ($(B)+ (-.2,.6)$) ;
 \draw [line width=1.2,->] (B) -- ($(B)+ (.5,.6)$) ;
 \draw ($(B)+ (-.5,.5)$) node [above,yshift=-1mm] {\scriptsize{$1$}};
 \draw ($(B)+ (.5,.6)$) node [above,yshift=-1mm] {\scriptsize{$l$}};
  \draw [-]  (0,-0.5) -- (A);
 \draw [black,fill=white] (B) circle (0.13)  ;
  \draw (-0.2,.8) node {$...$};
  \draw (1.3,.6) node {$...$}; 
  \draw [black,fill=white] (0,0) circle (0.13)  ;
  \draw (A) node [left,yshift=-1.2mm] {\tiny{$1$}};
  \draw (B) node [right] {\tiny{$2$}};
\end{tikzpicture}
}
,
\end{align*} 
where the notations $i,k,l,x$ are as in Definition \ref{de: modules} Item 2 and $\sigma_{(2,i),i+1}$ is the block permutation ($\sigma_{(2,i),i+1}(s)=s+1$ for $2\leq s\leq i$; $\sigma_{(2,i),i+1}(i+1)=2$ and $\sigma_{(2,i),i+1}(s)=s$ otherwise). 
The first decomposition corresponds to the top-right path in the diagram of Definition \ref{de: modules} Item 2; the second decomposition  corresponds to the left-bottom path. 

Note that, by pre-composing $\lambda^{\ca{M}}$ with $\alpha:\ca{A}s\to \ca{M}$, one has a left action of $\ca{A}s$: 
\begin{equation*}
\begin{tikzpicture}[>=stealth, arrow/.style={->,shorten >=3pt}, point/.style={coordinate}, pointille/.style={draw=red, top color=white, bottom color=red},scale=0.7,baseline=-0.9ex]%
 \coordinate (A) at (0,0);
 \coordinate (B) at (-1.55,.7);
 \coordinate (C) at (-.3,1.1);
 \coordinate (D) at (1.5,1);
 \draw [-] (A) -- (B);
 \draw [->] (B) -- ($(B)+ (-.5,.5)$) ;
 \draw [->] (B) -- ($(B)+ (-.2,.6)$) ;
 \draw [->] (B) -- ($(B)+ (.3,.6)$) ;
 \draw [->] (C) -- ($(C)+ (-.5,.5)$) ;
 \draw [->] (C) -- ($(C)+ (.3,.6)$) ;
 \draw (B) node [left] {\tiny{$1$}};
 \draw (C) node [left] {\tiny{${2}$}};
 \draw [-] (A) -- (C);
  \draw [-]  (0,-0.5) -- (A);
 \sq{(B)};
 \draw [-] (A) -- (D);
 \draw [->] (D) -- ($(D)+ (-.5,.5)$) ;
 \draw [->] (D) -- ($(D)+ (-.2,.6)$) ;
 \draw [->] (D) -- ($(D)+ (.3,.6)$) ;
 \sq{(D)};
 \draw (D) node [left] {\tiny{$k$}};
 \draw ($(B)+ (-.5,.5)$)  node [left] {\tiny{$1$}};
 \draw ($(B)+ (.3,.6)$)  node [above,yshift=-1.5mm] {\tiny{$m_1$}};
 \draw ($(D)+ (-.5,.5)$)  node [left,xshift=1mm] {\tiny{$1$}};
 \draw ($(D)+ (.3,.6)$)  node [right,xshift=-1mm] {\tiny{$m_k$}};
 \draw [black,fill=white] ($(C)-(.1,.1)$) rectangle ($(C)+(.1,.1)$) ;
  \draw (.45,.8) node {$...$}; 
  \draw [black,fill=black] (0,0) circle (0.11)  ;
  \draw (0,-.5) node [right] {\scriptsize{$\open$}};
\end{tikzpicture}%
\leftrightarrow \lambda^{\ca{A}s}: 
\ca{A}s(k)\ot \ca{Z}(m_1)\ot \cdots \ot \ca{Z}(m_k) \to \ca{Z}(m_1+...+m_k).
 \end{equation*}
 \item  The right action of $\ca{A}s$ on $\ca{Z}$ is given by: 
  \begin{equation*}
  \begin{tikzpicture}[>=stealth, arrow/.style={->,shorten >=3pt}, point/.style={coordinate}, pointille/.style={draw=red, top color=white, bottom color=red},scale=0.65,baseline=-0.9ex]%
 \coordinate (A) at (0,0);
 \coordinate (B) at (0.7,1.5);
 \draw [->] (A) -- (-1,1.2);
 \draw [->] (A) -- (0.2,1.2);
 \draw [-] (A) -- (B);
 \draw [->] (A) -- (1.2,0.9);
 \draw [->] (A) -- (1.8,0.6);
  \draw [-]  (0,-0.5) -- (A);
  \draw [->] (B) -- ($(B)+(-.2,0.5)$);
  \draw [->] (B) -- ($(B)+(+.9,0.5)$);
  \draw (-1,1.2) node [above] {\scriptsize{$1$}};
  \draw (0.1,1) node [above] {\scriptsize{$t-1$}};
  \draw (1.3,0.8) node [above] {\scriptsize{$t+1$}};
  \draw (1.8,.6) node [above,right] {\scriptsize{$k$}};
 \draw [black,fill=white] (-0.1,-0.1) rectangle (0.1,.1) ;
 \draw ($(B)+(-.2,0.5)$) node [left] {\scriptsize{$1$}};
 \draw ($(B)+(+.9,0.5)$) node [right] {\scriptsize{$l$}};
 \draw (-0.2,.8) node {$...$};
 \draw ($(B)+(.3,0.5)$) node {$...$};
 \draw (1.3,.6) node {$...$}; 
  \draw [black,fill=black] (B) circle (0.1)  ;
  \draw (0,-.5) node [right] {\scriptsize{$\open$}};
\end{tikzpicture}
\leftrightarrow \rho_t: \ca{Z}(k)=\ca{Z}(k)\ot \ca{A}s(l)\to \ca{Z}(k+l-1). 
 \end{equation*}
 The associativity for the right action is given by operadic composition of $\ca{RL}_2$. Precisely, for the left-sided square of Definition \ref{de: modules} Item 4, the left-bottom path results from $\ca{A}s(k)\ot \ca{A}s(l)\cong \ca{A}s(k+l-1)$  while the top-right path corresponds to the decomposition of the above tree as  
 \begin{equation*}
  \begin{tikzpicture}[>=stealth, arrow/.style={->,shorten >=3pt}, point/.style={coordinate}, pointille/.style={draw=red, top color=white, bottom color=red},scale=0.65,baseline=-0.9ex]%
 \coordinate (A) at (0,0);
 \coordinate (B) at (0.7,1.5);
 \draw [->] (A) -- (-1,1.2);
 \draw [->] (A) -- (0.2,1.2);
 \draw [-] (A) -- (B);
 \draw [->] (A) -- (1.2,0.9);
 \draw [->] (A) -- (1.8,0.6);
  \draw [-]  (0,-0.5) -- (A);
  \draw [->] (B) -- ($(B)+(-.2,0.5)$);
  \draw [->] (B) -- ($(B)+(+.9,0.5)$);
  \draw (-1,1.2) node [above] {\scriptsize{$1$}};
  \draw (0.1,1) node [above] {\scriptsize{$i-1$}};
  \draw (1.3,0.8) node [above] {\scriptsize{$i+1$}};
  \draw (1.8,.6) node [above,right] {\scriptsize{$m+k-1$}};
 \draw [black,fill=white] (-0.1,-0.1) rectangle (0.1,.1) ;
 \draw ($(B)+(-.2,0.5)$) node [left] {\scriptsize{$1$}};
 \draw ($(B)+(+.9,0.5)$) node [right] {\scriptsize{$l$}};
 \draw (-0.2,.8) node {$...$};
 \draw ($(B)+(.3,0.5)$) node {$...$};
 \draw (1.3,.6) node {$...$}; 
  \draw [black,fill=black] (B) circle (0.1)  ;
  \draw (0,-.5) node [right] {\scriptsize{$\open$}};
  \draw (A) node [left] {\textbf{\scriptsize{${v}$}}};
\end{tikzpicture}
\circ_v
  \begin{tikzpicture}[>=stealth, arrow/.style={->,shorten >=3pt}, point/.style={coordinate}, pointille/.style={draw=red, top color=white, bottom color=red},scale=0.65,baseline=-0.9ex]%
 \coordinate (A) at (0,0);
 \coordinate (B) at (0.7,1.5);
 \draw [->] (A) -- (-1,1.2);
 \draw [->] (A) -- (0.2,1.2);
 \draw [-] (A) -- (B);
 \draw [->] (A) -- (1.2,0.9);
 \draw [->] (A) -- (1.8,0.6);
  \draw [-]  (0,-0.5) -- (A);
  \draw [->] (B) -- ($(B)+(-.2,0.5)$);
  \draw [->] (B) -- ($(B)+(+.9,0.5)$);
  \draw (-1,1.2) node [above] {\scriptsize{$1$}};
  \draw (0.1,1) node [above] {\scriptsize{$t-1$}};
  \draw (1.3,0.8) node [above] {\scriptsize{$t+1$}};
  \draw (1.8,.6) node [above,right] {\scriptsize{$m$}};
 \draw [black,fill=white] (-0.1,-0.1) rectangle (0.1,.1) ;
 \draw ($(B)+(-.2,0.5)$) node [left] {\scriptsize{$1$}};
 \draw ($(B)+(+.9,0.5)$) node [right] {\scriptsize{$k$}};
 \draw (-0.2,.8) node {$...$};
 \draw ($(B)+(.3,0.5)$) node {$...$};
 \draw (1.3,.6) node {$...$}; 
  \draw [black,fill=black] (B) circle (0.1)  ;
  \draw (0,-.5) node [right] {\scriptsize{$\open$}};
\end{tikzpicture}
.
 \end{equation*}
 The right-sided square of Definition \ref{de: modules} Item 4 is obtained similarly. 
 \item The associativity of the left $\ca{M}$-action and right $\ca{A}s$-action on $\ca{Z}$, that is, the square of Definition \ref{de: modules} Item 5, is obtained by considering trees as  
\begin{equation*}
\begin{tikzpicture}[>=stealth, arrow/.style={->,shorten >=3pt}, point/.style={coordinate}, pointille/.style={draw=red, top color=white, bottom color=red},scale=0.65,baseline=1ex]%
 \coordinate (A) at (0,0);
 \coordinate (B) at (-1.5,1);
 \coordinate (C) at (0,1);
 \coordinate (E) at (1.5,1);
 \coordinate (F) at (0,1.5);
 \draw [-] (A) -- (B);
 \draw [-] (A) -- (C);
 \draw [-] (A) -- (E);
 \draw [-] (C) -- (F);
  \draw [-]  (0,-0.5) -- (A);
  \draw [->] (B) -- ($(B)+(-.25,.5)$);
  \draw [->] (B) -- ($(B)+(.25,.5)$);
  \draw [->] (C) -- ($(C)+(-.5,.5)$);
  \draw [->] (C) -- ($(C)+(.5,.5)$);
    \draw [->] (E) -- ($(E)+(-.25,.5)$);
    \draw [->] (E) -- ($(E)+(0,.5)$);
  \draw [->] (E) -- ($(E)+(.25,.5)$);
  \draw [->] (F) -- ($(F)+(-.25,.5)$);
  \draw [->] (F) -- ($(F)+(.25,.5)$);
  \draw (-.5,.7) node {\small{$...$}};
  \draw (.5,.7) node {\small{$...$}};
  \draw [black,fill=white] (A) circle (0.12)  ;
  \sq{(B)};
  \sq{(C)};
  \sq{(E)};
  \draw [black,fill=black] (F) circle (0.10)  ;
  \draw (0,-.5) node [right] {\scriptsize{$\open$}};
  \draw (A) node[left] {\tiny{$1$}};
  \draw (B) node[left] {\tiny{$2$}};
  \draw (C) node[left] {\tiny{$s+1$}};
  \draw (E) node[left] {\tiny{$k+1$}};
\end{tikzpicture}
= 
\begin{tikzpicture}[>=stealth, arrow/.style={->,shorten >=3pt}, point/.style={coordinate}, pointille/.style={draw=red, top color=white, bottom color=red},scale=0.65,baseline=1ex]%
 \coordinate (A) at (0,0);
 \coordinate (B) at (-1.5,1);
 \coordinate (C) at (0,1);
 \coordinate (E) at (1.5,1);
 \draw [-] (A) -- (B);
 \draw [-] (A) -- (C);
 \draw [-] (A) -- (E);
  \draw [-]  (0,-0.5) -- (A);
  \draw [->] (B) -- ($(B)+(-.25,.5)$);
  \draw [->] (B) -- ($(B)+(.25,.5)$);
  \draw [->] (C) -- ($(C)+(-.2,.5)$);
  \draw [->] (C) -- ($(C)+(.2,.5)$);
    \draw [->] (E) -- ($(E)+(-.25,.5)$);
    \draw [->] (E) -- ($(E)+(0,.5)$);
  \draw [->] (E) -- ($(E)+(.25,.5)$);
  \draw [->] (C) -- ($(C)+(-.5,.5)$);
  \draw [->] (C) -- ($(C)+(.5,.5)$);
  \draw (-.5,.7) node {\small{$...$}};
  \draw (.5,.7) node {\small{$...$}};
  \draw [black,fill=white] (A) circle (0.12)  ;
  \sq{(B)};
  \sq{(C)};
  \sq{(E)};
  \draw (0,-.5) node [right] {\scriptsize{$\open$}};
  \draw (A) node[left] {\tiny{$1$}};
  \draw (B) node[left] {\tiny{$2$}};
  \draw (C) node[left] {\tiny{$s+1$}};
  \draw (E) node[left] {\tiny{$k+1$}};
\end{tikzpicture}
\circ_{s+1}
\begin{tikzpicture}[>=stealth, arrow/.style={->,shorten >=3pt}, point/.style={coordinate}, pointille/.style={draw=red, top color=white, bottom color=red},scale=0.65,baseline=0.9ex]%
 \coordinate (B) at (0,0);
 \coordinate (C) at (0,0.6);
 \draw [-] (0,-.5) --  (B) -- (C);
 \draw [->] (B) -- ($(B)+ (-.75,.5)$) ;
 \draw [->] (B) -- ($(B)+ (.75,.5)$) ;
 \draw [->] (C) -- ($(C)+ (-.5,.5)$) ;
 \draw [->] (C) -- ($(C)+ (.5,.5)$) ;
 \sq{(B)};
  \draw [black,fill=black] (C) circle (0.11)  ;
  \draw (0,-.5) node [right] {\scriptsize{$\open$}};
\end{tikzpicture}%
=
\begin{tikzpicture}[>=stealth, arrow/.style={->,shorten >=3pt}, point/.style={coordinate}, pointille/.style={draw=red, top color=white, bottom color=red},scale=0.65,baseline=0.9ex]%
 \coordinate (B) at (0,0);
 \coordinate (C) at (0,0.6);
 \draw [-] (0,-.5) --  (B) -- (C);
  \draw [->] (B) -- ($(B)+ (-1,.25)$) ;
 \draw [->] (B) -- ($(B)+ (-1,.5)$) ;
 \draw [->] (B) -- ($(B)+ (-.25,.5)$) ;
 \draw [->] (B) -- ($(B)+ (.25,.5)$) ;
 \draw [->] (B) -- ($(B)+ (1,.5)$) ;
 \draw [->] (B) -- ($(B)+ (1,.25)$) ;
 \draw [->] (B) -- ($(B)+ (1,.1)$) ;
 \draw [->] (C) -- ($(C)+ (-.5,.5)$) ;
 \draw [->] (C) -- ($(C)+ (.5,.5)$) ;
 \draw (-.5,.5) node {\scriptsize{$...$}};
 \draw (.5,.5) node {\scriptsize{$...$}};
 \sq{(B)};
 \draw (A) node[left,yshift=-2mm] {\tiny{$1$}};
  \draw [black,fill=black] (C) circle (0.11)  ;
  \draw (0,-.5) node [right] {\scriptsize{$\open$}};
\end{tikzpicture}%
\circ_{1}
\begin{tikzpicture}[>=stealth, arrow/.style={->,shorten >=3pt}, point/.style={coordinate}, pointille/.style={draw=red, top color=white, bottom color=red},scale=0.65,baseline=1ex]%
 \coordinate (A) at (0,0);
 \coordinate (B) at (-1.5,1);
 \coordinate (C) at (0,1);
 \coordinate (E) at (1.5,1);
 \draw [-] (A) -- (B);
 \draw [-] (A) -- (C);
 \draw [-] (A) -- (E);
  \draw [-]  (0,-0.5) -- (A);
  \draw [->] (B) -- ($(B)+(-.25,.5)$);
  \draw [->] (B) -- ($(B)+(.25,.5)$);
  \draw [->] (C) -- ($(C)+(0,.5)$);
    \draw [->] (E) -- ($(E)+(-.25,.5)$);
    \draw [->] (E) -- ($(E)+(0,.5)$);
  \draw [->] (E) -- ($(E)+(.25,.5)$);
  \draw [->] (C) -- ($(C)+(-.5,.5)$);
  \draw [->] (C) -- ($(C)+(.5,.5)$);
  \draw (-.5,.7) node {\small{$...$}};
  \draw (.5,.7) node {\small{$...$}};
  \draw [black,fill=white] (A) circle (0.12)  ;
  \sq{(B)};
  \sq{(C)};
  \sq{(E)};
  \draw (0,-.5) node [right] {\scriptsize{$\open$}};
  \draw (A) node[left] {\tiny{$1$}};
  \draw (B) node[left] {\tiny{$2$}};
  \draw (C) node[left] {\tiny{$s+1$}};
  \draw (E) node[left] {\tiny{$k+1$}};
\end{tikzpicture}.
\end{equation*}
\item The morphism  $\iota:\ca{M}\to \ca{Z}$ is given by corollas in a closed vertex with an open output: 
 \begin{equation*}
\begin{tikzpicture}[>=stealth, arrow/.style={->,shorten >=3pt}, point/.style={coordinate}, pointille/.style={draw=red, top color=white, bottom color=red},scale=0.65,baseline=-0.9ex]%
 \coordinate (A) at (0,0);
 \coordinate (B) at (-.75,.75);
 \coordinate (C) at (-.5,.75);
 \coordinate (D) at (.5,.75);
 \coordinate (E) at (.75,.75);
 \draw [->] (A) -- (B);
 \draw [->] (A) -- (C);
 \draw [->] (A) -- (D);
 \draw [->] (A) -- (E);
  \draw [-]  (0,-0.5) -- (A);
  \draw (B) node [left] {\scriptsize{$1$}};
  \draw (E) node [above,right] {\scriptsize{$k$}};
  \draw (0,.6) node {$...$};
  \draw [black,fill=white] (0,0) circle (0.12)  ;
  \draw (0,-.5) node [right] {\scriptsize{$\open$}};
\end{tikzpicture}
\leftrightarrow \iota: \ca{M}(k)\to \ca{Z}(k). 
 \end{equation*}
\end{enumerate}
Note that the operations as those given by $k$-corollas in a closed vertex with $1\leq j\leq k$ open vertices at inputs $1\leq b_1<...<b_j\leq k$, 
 \begin{equation*}
\begin{tikzpicture}[>=stealth, arrow/.style={->,shorten >=3pt}, point/.style={coordinate}, pointille/.style={draw=red, top color=white, bottom color=red},scale=0.65,baseline=0.9ex]%
 \coordinate (A) at (0,0);
 \coordinate (B) at (-1.5,1.3);
 \coordinate (C) at (.2,1.2);
 \coordinate (D) at (1.5,1.3);
 \draw [-] (A) -- (B);
 \draw [->] (B) -- ($(B)+ (-.5,.5)$) ;
 \draw [->] (A) -- ($(-1.5,.5)$) ;
 \draw [->] (B) -- ($(B)+ (-.2,.6)$) ;
 \draw [->] (B) -- ($(B)+ (.3,.6)$) ;
 \draw [->] (C) -- ($(C)+ (-.5,.5)$) ;
 \draw [->] (C) -- ($(C)+ (.3,.6)$) ;
 \draw (-1.5,.5) node [above,left] {\scriptsize{$1$}};
 \draw (B) node [left] {\tiny{$v_{b_1}$}};
 \draw (C) node [left] {\tiny{$v_{b_{r}}$}};
 \draw (1.8,.6) node [below] {\scriptsize{$k$}};
 \draw [-] (A) -- (0.2,1.2);
 \draw [->] (A) -- (0.6,1);
 \draw [->] (A) -- (1.8,0.6);
  \draw [-]  (0,-0.5) -- (A);
 \sq{(B)};
 \draw [-] (A) -- (D);
 \draw [->] (D) -- ($(D)+ (-.5,.5)$) ;
 \draw [->] (A) -- ($(-1.5,.5)$) ;
 \draw [->] (D) -- ($(D)+ (-.2,.6)$) ;
 \draw [->] (D) -- ($(D)+ (.3,.6)$) ;
 \sq{(D)};
 \draw (D) node [right] {\tiny{$v_{b_j}$}};
 \draw [black,fill=white] (0.1,1.1) rectangle (0.3,1.3) ;
  \draw (-0.2,.8) node {$...$};
  \draw (1.3,.6) node {$...$}; 
  \draw [black,fill=white] (0,0) circle (0.13)  ;
  \draw (0,-.5) node [right] {\scriptsize{$\open$}};
\end{tikzpicture}%
\leftrightarrow \lambda_{b_1,...,b_j}: 
\ca{M}(k)\ot \ca{Z}(m_1)\ot \cdots \ot \ca{Z}(m_j) \to \ca{Z}(k-j+\sum_{1\leq p\leq j}m_p),
 \end{equation*}
can be obtained by pre-composing the left action $\lambda$ by $\ca{A}s(1)\to \ca{Z}(1)$ at the all inputs other than the $b_i$'s. This results from the fact that such a $k$-corolla is obtained as: 
\begin{equation}\label{eq: decompo left}
\left(\left(\cdots \left(
\begin{tikzpicture}[>=stealth, arrow/.style={->,shorten >=3pt}, point/.style={coordinate}, pointille/.style={draw=red, top color=white, bottom color=red},scale=0.65,baseline=1.9ex]%
 \coordinate (A) at (0,0);
 \coordinate (B) at (-1.5,1.3);
 \coordinate (C) at (.2,1.2);
 \coordinate (D) at (1.8,1.3);
 \coordinate (E) at (.55,1);
 \draw [-] (A) -- (B);
 \draw [->] (B) -- ($(B)+ (-.5,.5)$) ;
 \draw [-] (A) -- ($(-1.5,.5)$) ;
 \draw [->] ($(-1.5,.5)$) -- ($(-1.5,.5)+(-.75,.25)$) ;
 \draw [->] (B) -- ($(B)+ (-.2,.6)$) ;
 \draw [->] (B) -- ($(B)+ (.3,.6)$) ;
 \draw [->] (C) -- ($(C)+ (-.5,.5)$) ;
 \draw [->] (C) -- ($(C)+ (.3,.6)$) ;
 \draw (-1.5,.5) node [below] {\scriptsize{$v_1$}};
 \draw (B) node [left] {\tiny{$v_{b_1}$}};
 \draw (C) node [left] {\tiny{$v_{b_{r}}$}};
 \draw (1.8,.6) node [below] {\scriptsize{$v_k$}};
 \draw [-] (A) -- (0.2,1.2);
 \draw [-] (A) -- (E);
 \draw [->] (E) -- ($(E)+(.2,.4)$);
 \draw [-] (A) -- (1.8,0.6);
 \draw [->] (1.8,0.6) -- (2.4,.9);
  \draw [-]  (0,-0.5) -- (A);
 \sq{(E)};
 \sq{(B)};
 \draw ($(E)+(-.05,0)$) node [right] {\tiny{$v_{a_s}$}};
 \draw [-] (A) -- (D);
 \draw [->] (D) -- ($(D)+ (-.5,.5)$) ;
 \draw [->] (D) -- ($(D)+ (-.2,.6)$) ;
 \draw [->] (D) -- ($(D)+ (.3,.6)$) ;
 \sq{(D)};
 \draw (D) node [right] {\tiny{$v_{b_j}$}};
 \draw [black,fill=white] (1.7,0.5) rectangle (1.9,0.7) ;
 \draw [black,fill=white] ($(-1.6,.4)$) rectangle ($(-1.4,.6)$) ;
 \draw [black,fill=white] (0.1,1.1) rectangle (0.3,1.3) ;
  \draw (-0.2,.8) node {$...$};
  \draw (1.3,.6) node {$...$}; 
  \draw [black,fill=white] (0,0) circle (0.13)  ;
  \draw (0,-.5) node [right] {\scriptsize{$\open$}};
\end{tikzpicture}%
\circ_{v_k}
\begin{tikzpicture}[>=stealth, arrow/.style={->,shorten >=3pt}, point/.style={coordinate}, pointille/.style={draw=red, top color=white, bottom color=red},scale=0.65,baseline=-0.9ex]%
 \coordinate (A) at (0,0);
 \draw [-] (A) -- ($(0,-.5)$);
 \draw [->] (A) -- ($(0,1)$) ;
 \draw [black,fill=black] (A) circle (0.12)  ;
 \draw (0,-.5) node [right] {\scriptsize{$\open$}};
\end{tikzpicture}%
\right)\cdots  \circ_{v_{a_s}} 
\begin{tikzpicture}[>=stealth, arrow/.style={->,shorten >=3pt}, point/.style={coordinate}, pointille/.style={draw=red, top color=white, bottom color=red},scale=0.65,baseline=-0.9ex]%
 \coordinate (A) at (0,0);
 \draw [-] (A) -- ($(0,-.5)$);
 \draw [->] (A) -- ($(0,1)$) ;
 \draw [black,fill=black] (A) circle (0.12)  ;
 \draw (0,-.5) node [right] {\scriptsize{$\open$}};
\end{tikzpicture}%
\right)
 \circ_{v_{a_{s-1}}} 
 \begin{tikzpicture}[>=stealth, arrow/.style={->,shorten >=3pt}, point/.style={coordinate}, pointille/.style={draw=red, top color=white, bottom color=red},scale=0.65,baseline=-0.9ex]%
 \coordinate (A) at (0,0);
 \draw [-] (A) -- ($(0,-.5)$);
 \draw [->] (A) -- ($(0,1)$) ;
 \draw [black,fill=black] (A) circle (0.12)  ;
 \draw (0,-.5) node [right] {\scriptsize{$\open$}};
\end{tikzpicture} 
 \cdots \circ_{v_1} 
 \begin{tikzpicture}[>=stealth, arrow/.style={->,shorten >=3pt}, point/.style={coordinate}, pointille/.style={draw=red, top color=white, bottom color=red},scale=0.65,baseline=-0.9ex]%
 \coordinate (A) at (0,0);
 \draw [-] (A) -- ($(0,-.5)$);
 \draw [->] (A) -- ($(0,1)$) ;
 \draw [black,fill=black] (A) circle (0.12)  ;
 \draw (0,-.5) node [right] {\scriptsize{$\open$}};
\end{tikzpicture} 
 \right), 
 \end{equation}
 where the $a_s$'s are the inputs other than the $b_j$'s ($a_1=1$ in the picture above). In particular, the weak left action is obtained whenever $j=1$ and the left action corresponds to $j=k$. Doing this to the corollas of $\lambda^{\ca{A}s}$ in  (3) above  endows  $\ca{Z}$ with a structure of weak bimodule over $\ca{A}s$. 
 
Let us remark that, because of the decomposition 
  $\begin{tikzpicture}[>=stealth, arrow/.style={->,shorten >=3pt}, point/.style={coordinate}, pointille/.style={draw=red, top color=white, bottom color=red},scale=0.65,baseline=-0.9ex]%
 \coordinate (A) at (0,0);
 \coordinate (B) at (-.75,.75);
 \coordinate (C) at (-.5,.75);
 \coordinate (D) at (.5,.75);
 \coordinate (E) at (.75,.75);
 \draw [->] (A) -- (B);
 \draw [->] (A) -- (C);
 \draw [->] (A) -- (D);
 \draw [->] (A) -- (E);
  \draw [-]  (0,-0.5) -- (A);
  \draw (B) node [left] {\scriptsize{$1$}};
  \draw (E) node [above,right] {\scriptsize{$k$}};
  \draw (0,.6) node {$...$};
  \draw [black,fill=black] (0,0) circle (0.12)  ;
  \draw (0,-.5) node [right] {\scriptsize{$\open$}};
\end{tikzpicture}
=
  \begin{tikzpicture}[>=stealth, arrow/.style={->,shorten >=3pt}, point/.style={coordinate}, pointille/.style={draw=red, top color=white, bottom color=red},scale=0.65,baseline=-0.9ex]%
 \coordinate (A) at (0,0);
 \coordinate (B) at (-.75,.75);
 \coordinate (C) at (-.5,.75);
 \coordinate (D) at (.5,.75);
 \coordinate (E) at (.75,.75);
 \draw [->] (A) -- (B);
 \draw [->] (A) -- (C);
 \draw [->] (A) -- (D);
 \draw [->] (A) -- (E);
  \draw [-]  (0,-0.5) -- (A);
  \draw (B) node [left] {\scriptsize{$1$}};
  \draw (E) node [above,right] {\scriptsize{$k$}};
  \draw (0,.6) node {$...$};
  \draw [black,fill=white] (0,0) circle (0.12)  ;
  \draw (0,-.5) node [right] {\scriptsize{$\open$}};
  \draw (A) node [right] {\scriptsize{$v$}};
\end{tikzpicture}
\circ_v
\begin{tikzpicture}[>=stealth, arrow/.style={->,shorten >=3pt}, point/.style={coordinate}, pointille/.style={draw=red, top color=white, bottom color=red},scale=0.65,baseline=-0.9ex]%
 \coordinate (A) at (0,0);
 \coordinate (B) at (-.75,.75);
 \coordinate (C) at (-.5,.75);
 \coordinate (D) at (.5,.75);
 \coordinate (E) at (.75,.75);
 \draw [->] (A) -- (B);
 \draw [->] (A) -- (C);
 \draw [->] (A) -- (D);
 \draw [->] (A) -- (E);
  \draw [-]  (0,-0.5) -- (A);
  \draw (B) node [left] {\scriptsize{$1$}};
  \draw (E) node [above,right] {\scriptsize{$k$}};
  \draw (0,.6) node {$...$};
  \draw [black,fill=black] (0,0) circle (0.12)  ;
\end{tikzpicture}$,  
 the  morphism $\ca{A}s\to \ca{Z}$ given by the left-sided tree corresponds to  $\iota\alpha: \ca{A}s\to \ca{M} \to \ca{Z}$. More generally, note that any tree can be obtained as a composition of \emph{elementary trees}:  trees as in (1)-(4), (6) above and corollas in a neutral  vertex.  
  Therefore, for $(\ca{M},\ca{Z})\in \cat{E}$, the operation corresponding to a given a tree $T$ is defined as given by any of the decompositions of $T$ in elementary trees; the independence of this operation regarding the  different decompositions is ensured by the  properties of $(\ca{M},\ca{Z})$.  
 For instance, the two decompositions 
  \begin{equation*}
   \begin{tikzpicture}[>=stealth, arrow/.style={->,shorten >=3pt}, point/.style={coordinate}, pointille/.style={draw=red, top color=white, bottom color=red},scale=0.65,baseline=1ex]%
 \coordinate (A) at (0,0);
 \coordinate (B) at (-1.5,1);
 \coordinate (C) at (-.5,1);
 \coordinate (D) at (.5,1);
 \coordinate (E) at (1.5,1);
 \draw [-] (A) -- (B);
 \draw [-] (A) -- (C);
 \draw [-] (A) -- (D);
 \draw [-] (A) -- (E);
  \draw [-]  (0,-0.5) -- (A);
  \draw [->] (B) -- ($(B)+(-.25,.5)$);
  \draw [->] (B) -- ($(B)+(.25,.5)$);
  \draw [->] (C) -- ($(C)+(-.25,.5)$);
  \draw [->] (C) -- ($(C)+(.25,.5)$);
  \draw [->] (D) -- ($(D)+(-.25,.5)$);
  \draw [->] (D) -- ($(D)+(.25,.5)$);
    \draw [->] (E) -- ($(E)+(-.25,.5)$);
    \draw [->] (E) -- ($(E)+(0,.5)$);
  \draw [->] (E) -- ($(E)+(.25,.5)$);
  \draw [black,fill=white] (A) circle (0.12)  ;
  \sq{(B)};
  \sq{(D)};
  \draw [black,fill=white] (C) circle (0.12)  ;
  \draw [black,fill=white] (E) circle (0.12)  ;
  \draw (0,-.5) node [right] {\scriptsize{$\open$}};
\end{tikzpicture}
=\left(
\begin{tikzpicture}[>=stealth, arrow/.style={->,shorten >=3pt}, point/.style={coordinate}, pointille/.style={draw=red, top color=white, bottom color=red},scale=0.65,baseline=1ex]%
 \coordinate (A) at (0,0);
 \coordinate (B) at (-1.5,1);
 \coordinate (C) at (-.5,1);
 \coordinate (D) at (.5,1);
 \coordinate (E) at (1.5,1);
 \draw [-] (A) -- (B);
 \draw [-] (A) -- (C);
 \draw [-] (A) -- (D);
 \draw [-] (A) -- (E);
  \draw [-]  (0,-0.5) -- (A);
  \draw [->] (B) -- ($(B)+(-.25,.5)$);
  \draw [->] (B) -- ($(B)+(.25,.5)$);
  \draw [->] (C) -- ($(C)+(-.25,.5)$);
  \draw [->] (C) -- ($(C)+(.25,.5)$);
  \draw [->] (D) -- ($(D)+(-.25,.5)$);
  \draw [->] (D) -- ($(D)+(.25,.5)$);
    \draw [->] (E) -- ($(E)+(-.25,.5)$);
    \draw [->] (E) -- ($(E)+(0,.5)$);
  \draw [->] (E) -- ($(E)+(.25,.5)$);
  \draw (C) node [right] {\scriptsize{$v_2$}};
  \draw (E) node [above,right] {\scriptsize{$v_4$}};
  \draw [black,fill=white] (A) circle (0.12)  ;
  \sq{(B)};
  \sq{(C)};
  \sq{(E)};
  \sq{(D)};
  \draw (0,-.5) node [right] {\scriptsize{$\open$}};
\end{tikzpicture}
\circ_{v_2}
\begin{tikzpicture}[>=stealth, arrow/.style={->,shorten >=3pt}, point/.style={coordinate}, pointille/.style={draw=red, top color=white, bottom color=red},scale=0.65,baseline=-0.9ex]%
 \coordinate (A) at (0,0);
 \coordinate (B) at (-.5,.75);
 \coordinate (C) at (.5,.75);
 \draw [->] (A) -- (B);
 \draw [->] (A) -- (C);
  \draw [-]  (0,-0.5) -- (A);
  \draw [black,fill=white] (A) circle (0.12)  ;
  \draw (0,-.5) node [right] {\scriptsize{$\open$}};
\end{tikzpicture}
\right)
\circ_{v_4}
\begin{tikzpicture}[>=stealth, arrow/.style={->,shorten >=3pt}, point/.style={coordinate}, pointille/.style={draw=red, top color=white, bottom color=red},scale=0.65,baseline=-0.9ex]%
 \coordinate (A) at (0,0);
 \coordinate (B) at (-.5,.75);
 \coordinate (C) at (.5,.75);
 \coordinate (D) at (0,.75);
 \draw [->] (A) -- (B);
 \draw [->] (A) -- (C);
 \draw [->] (A) -- (D);
  \draw [-]  (0,-0.5) -- (A);
  \draw [black,fill=white] (A) circle (0.12)  ;
  \draw (0,-.5) node [right] {\scriptsize{$\open$}};
\end{tikzpicture}
=
\begin{tikzpicture}[>=stealth, arrow/.style={->,shorten >=3pt}, point/.style={coordinate}, pointille/.style={draw=red, top color=white, bottom color=red},scale=0.65,baseline=1ex]%
 \coordinate (A) at (0,0);
 \coordinate (B) at (-1.5,1);
 \coordinate (C1) at (-.6,1);
 \coordinate (C2) at (-.3,1);
 \coordinate (D) at (.5,1);
 \coordinate (E1) at (1.2,1);
 \coordinate (E2) at (1.5,.75);
 \coordinate (E3) at (1.7,.5);
 \draw [-] (A) -- (B);
 \draw [->] (A) -- (C1);
 \draw [->] (A) -- (C2);
 \draw [-] (A) -- (D);
 \draw [->] (A) -- (E1);
 \draw [->] (A) -- (E2);
 \draw [->] (A) -- (E3);
  \draw [-]  (0,-0.5) -- (A);
  \draw [->] (B) -- ($(B)+(-.25,.5)$);
  \draw [->] (B) -- ($(B)+(.25,.5)$);
  \draw [->] (D) -- ($(D)+(-.25,.5)$);
  \draw [->] (D) -- ($(D)+(.25,.5)$);
  \draw [black,fill=white] (A) circle (0.12)  ;
  \sq{(B)};
  \sq{(D)};
  \draw (A) node [left,yshift=-1mm] {\scriptsize{$v$}};
  \draw (0,-.5) node [right] {\scriptsize{$\open$}};
\end{tikzpicture}
\circ_{v}
\begin{tikzpicture}[>=stealth, arrow/.style={->,shorten >=3pt}, point/.style={coordinate}, pointille/.style={draw=red, top color=white, bottom color=red},scale=0.65,baseline=1ex]%
 \coordinate (A) at (0,0);
 \coordinate (B) at (-1.5,1);
 \coordinate (C) at (-.5,1);
 \coordinate (D) at (.5,1);
 \coordinate (E) at (1.5,1);
 \draw [->] (A) -- (B);
 \draw [-] (A) -- (C);
 \draw [->] (A) -- (D);
 \draw [-] (A) -- (E);
  \draw [-]  (0,-0.5) -- (A);
  \draw [->] (C) -- ($(C)+(-.25,.5)$);
  \draw [->] (C) -- ($(C)+(.25,.5)$); 
  \draw [->] (E) -- ($(E)+(-.25,.5)$);
  \draw [->] (E) -- ($(E)+(0,.5)$);
  \draw [->] (E) -- ($(E)+(.25,.5)$);
  \draw [black,fill=white] (A) circle (0.12)  ;
  \draw [black,fill=white] (C) circle (0.12)  ;
  \draw [black,fill=white] (E) circle (0.12)  ;
\end{tikzpicture}, 
  \end{equation*}
corresponds to ($2$ iterations of) the diagram 2. of Definition \ref{de: modules}, plus the fact that $\iota$ is a left module map and $\eta:1_{\cat{C}}\to \ca{M}$ is the unit. Here, the penultimate tree is not elementary but it  admits a decomposition into elementary trees as in \eqref{eq: decompo left}. 
\end{proof}

\subsection{The algebras over $\ca{RL}'_2$}\label{sec: def repres prime}

 We show that the operad $\ca{RL}'_2$ encodes the pairs $(\ca{M},\ca{Z})$ subject to the following conditions.
\begin{enumerate}[I'.]
\item $\ca{M}$ is a  multiplicative non-$\Sigma$ operad.\label{B1}
 \item $\ca{Z}\in \bmodasm$ with a morphism $\iota:\ca{M}\to \ca{Z}$ in $\bmodasm$.  \label{B3} 
\end{enumerate}

Let $\cat{E'}$ be the category with objects the pairs $(\ca{M},\ca{Z})$ satisfying the two conditions \ref{B1}' and \ref{B3}' above; morphisms are the pairs 
$(f,g):(\ca{M},\ca{Z}) \to (\ca{M}',\ca{Z}')$  subject to the following conditions:
\begin{itemize}
 \item $f:\ca{M}\to \ca{M}'$ is a morphism of multiplicative non-$\Sigma$ operads;
 \item $g:\ca{Z}\to\ca{Z}'$ is an $f$-equivariant morphism of right modules over $\ca{M}$; 
 \item $g$ is a morphism of bimodules over $\ca{A}s$; and, 
 \item $\iota'\circ f=g\circ \iota$. 
\end{itemize}
\begin{prop}\label{prop: hypI,II,III correspondance right} 
 The category of $\ca{RL}'_2$-algebras in $\cat{C}$ is isomorphic to the category $\cat{E'}$.
\end{prop}
\begin{proof}
 The only significant differences with Proposition \ref{prop: hypI,II,III correspondance} are the following. 
 The right action of $\ca{M}$ on $\ca{Z}$ is given by  the $k$-corollas in an open vertex with a closed vertex at an  input: 
\begin{equation*}
  \begin{tikzpicture}[>=stealth, arrow/.style={->,shorten >=3pt}, point/.style={coordinate}, pointille/.style={draw=red, top color=white, bottom color=red},scale=0.7,baseline=-0.9ex]%
 \coordinate (A) at (0,0);
 \coordinate (B) at (0.7,1.5);
 \draw [->] (A) -- (-1,1.2);
 \draw [->] (A) -- (0.2,1.2);
 \draw [-] (A) -- (B);
 \draw [->] (A) -- (1.2,0.9);
 \draw [->] (A) -- (1.8,0.6);
  \draw [-]  (0,-0.5) -- (A);
  \draw [->] (B) -- ($(B)+(-.2,0.7)$);
  \draw [->] (B) -- ($(B)+(+.9,0.7)$);
  \draw (-1,1.2) node [above] {\scriptsize{$1$}};
  \draw (0.1,1) node [above] {\scriptsize{$i-1$}};
  \draw (1.3,0.8) node [above] {\scriptsize{$i+1$}};
  \draw (1.8,.6) node [above,right] {\scriptsize{$k$}};
 \draw [black,fill=white] (-0.1,-0.1) rectangle (0.1,.1) ;
 \draw ($(B)+(-.2,0.7)$) node [left] {\scriptsize{$1$}};
 \draw ($(B)+(+.9,0.7)$) node [right] {\scriptsize{$l$}};
 \draw ($(B)+(+.3,0.7)$ node {$...$};
  \draw (-0.2,.8) node {$...$};
  \draw (1.3,.6) node {$...$}; 
  \draw [black,fill=white] (B) circle (0.11)  ;
  \draw (0,-.5) node [right] {\scriptsize{$\open$}};
\end{tikzpicture}
\leftrightarrow \rho_i: \ca{Z}(k)\ot \ca{M}(l)\to \ca{Z}(k+l-1). 
 \end{equation*}
 The operations  such as those given by $k$-corollas in a neutral  vertex with $1\leq j\leq k$ open vertices at inputs $1\leq b_1<...<b_j\leq k$, 
 \begin{equation*}
\begin{tikzpicture}[>=stealth, arrow/.style={->,shorten >=3pt}, point/.style={coordinate}, pointille/.style={draw=red, top color=white, bottom color=red},scale=0.65,baseline=0.9ex]%
 \coordinate (A) at (0,0);
 \coordinate (B) at (-1.5,1.3);
 \coordinate (C) at (.2,1.2);
 \coordinate (D) at (1.5,1.3);
 \draw [-] (A) -- (B);
 \draw [->] (B) -- ($(B)+ (-.5,.5)$) ;
 \draw [->] (A) -- ($(-1.5,.5)$) ;
 \draw [->] (B) -- ($(B)+ (-.2,.6)$) ;
 \draw [->] (B) -- ($(B)+ (.3,.6)$) ;
 \draw [->] (C) -- ($(C)+ (-.5,.5)$) ;
 \draw [->] (C) -- ($(C)+ (.3,.6)$) ;
 \draw (-1.5,.5) node [above,left] {\scriptsize{$1$}};
 \draw (B) node [left] {\tiny{$v_{b_1}$}};
 \draw (C) node [left] {\tiny{$v_{b_{r}}$}};
 \draw (1.8,.6) node [below] {\scriptsize{$k$}};
 \draw [-] (A) -- (0.2,1.2);
 \draw [->] (A) -- (0.6,1);
 \draw [->] (A) -- (1.8,0.6);
  \draw [-]  (0,-0.5) -- (A);
 \sq{(B)};
 \draw [-] (A) -- (D);
 \draw [->] (D) -- ($(D)+ (-.5,.5)$) ;
 \draw [->] (A) -- ($(-1.5,.5)$) ;
 \draw [->] (D) -- ($(D)+ (-.2,.6)$) ;
 \draw [->] (D) -- ($(D)+ (.3,.6)$) ;
 \sq{(D)};
 \draw (D) node [right] {\tiny{$v_{b_j}$}};
 \draw [black,fill=white] (0.1,1.1) rectangle (0.3,1.3) ;
  \draw (-0.2,.8) node {$...$};
  \draw (1.3,.6) node {$...$}; 
  \draw [black,fill=black] (0,0) circle (0.13)  ;
  \draw (0,-.5) node [right] {\scriptsize{$\open$}};
\end{tikzpicture}%
\leftrightarrow \lambda_{b_1,...,b_j}: 
\ca{A}s(k)\ot \ca{Z}(m_1)\ot \cdots \ot \ca{Z}(m_j) \to \ca{Z}(k-j+\sum_{1\leq p\leq j}m_p),
 \end{equation*}
 are obtained via the left action of $\ca{A}s$ on $\ca{Z}$, using the fact that $\iota\alpha:\ca{A}s\to \ca{M}\to \ca{Z}$ is a morphism of left-modules. 
\end{proof}

\bibliographystyle{alpha}

\end{document}